\documentclass[reqno,10pt]{amsart}
\usepackage{amsmath,amssymb,mathrsfs,amsthm,amsfonts}
\usepackage[inline]{enumitem}
\usepackage[usenames,dvipsnames]{xcolor}
\usepackage{hyperref}
\usepackage[percent]{overpic}
\usepackage{comment}
\usepackage{algorithm}
\usepackage{algorithmic}
\usepackage{mathtools}
\usepackage{subcaption}
\usepackage{tikz}
\usepackage{bm}
\usetikzlibrary{calc}
\usetikzlibrary{arrows.meta,backgrounds}
\usepackage{multirow,array,longtable,booktabs}
\usetikzlibrary{arrows}

\hypersetup{
	colorlinks=true, linkcolor=blue,
	citecolor=ForestGreen
}
\usepackage[paper=letterpaper,margin=1in]{geometry}

\newtheorem{theorem}{Theorem}[section]
\newtheorem{lemma}[theorem]{Lemma}
\newtheorem{proposition}[theorem]{Proposition}

\newtheorem{prop}[theorem]{Proposition}
\newtheorem{corollary}[theorem]{Corollary}
\theoremstyle{definition}
\newtheorem{definition}[theorem]{Definition}

\newtheorem{assumption}[theorem]{Assumption}
\newtheorem{remark}[theorem]{Remark}
\numberwithin{equation}{section}
\allowdisplaybreaks
\usepackage{acronym}

\acrodef{LDP}{Large Deviation Principle}

\newcommand{\E}{\mathbf{E}}	

\newcommand{\red}{\textcolor{red}}

\newcommand{\blue}{\textcolor{blue}}

\newcommand{\be}{\begin{equation}}
\newcommand{\ee}{\end{equation}}
\newcommand{\bea} {\begin{array}{rl}}
\newcommand{\eea} {\end{array}}
\newcommand{\bepa}{\left\{ \begin{array}{l}}
\newcommand{\eepa} {\end{array}\right.}
\makeatletter

\newcommand\norm[1]{\left\Vert#1\right\Vert}

\newcommand{\R}{\mathbb{R}} 
\newcommand{\N}{\mathbb{N}}

\newcommand{\op}{\mathrm{op}}

\newcommand{\ds}{\displaystyle}

\newcommand{\cB}{\mathcal{B}}
\newcommand{\cC}{\mathcal{C}}
\newcommand{\scrL}{\mathscr{L}}

\renewcommand{\hat}{\widehat}

\usepackage{graphicx}

\newcommand{\bd}{{\bf d}}

\newcommand{\ov}{\overline}
\newcommand{\bx}{\bm{x}}
\newcommand{\bxi}{\bm{\xi}}
\newcommand{\bX}{\bm{X}}

\newcommand{\cP}{\mathcal{P}}
\newcommand{\cT}{\mathcal{T}}

\newcommand{\eps}{\epsilon}

\newcommand{\linf}{L^{\infty}}
\newcommand{\id}{\text{id}}
\newcommand{\com}{\text{com}}
\newcommand{\tr}{\text{tr}}
\newcommand{\cA}{\mathcal{A}}

\newcommand{\dist}{\text{dist}}
\newcommand{\dis}{\text{disp}}

\newcommand{\bv}{\bm{v}}

\newcommand{\cF}{\mathcal{F}}

\newcommand{\cL}{\mathcal{L}}
\newcommand{\bbF}{\mathbb{F}}
\newcommand{\bu}{\bm{u}}
\newcommand{\wt}{\widetilde}

\newcommand{\by}{\bm{y}}
\newcommand{\bz}{\bm{z}}

\newcommand{\balpha}{\bm{\alpha}}
\newcommand{\ba}{\bm{a}}
\newcommand{\bp}{\bm{p}}
\newcommand{\diag}{D_{\text{diag}}}

\newcommand{\cM}{\mathcal{M}}
\newcommand{\cD}{\mathcal{D}}
\newcommand{\cO}{\mathcal{O}}

\newcommand{\spt}{\text{spt}}

\title[Master equation for MFGC]{Unconditional well-posedness of the master equation for monotone mean field games of controls}

\author[J.\ Jackson]{Joe Jackson}
\address{J.\ Jackson,
	Department of Mathematics, University of Chicago,
	\newline\hphantom{\quad \ \ J. Jackson}
	5734 S.~University Avenue, Chicago, Illinois 60637 USA
}
\email{jsjackson@uchicago.edu}

\author[A.R. M\'esz\'aros]{Alp\'ar R. M\'esz\'aros}  
\address{A.R. M\'esz\'aros, Department of Mathematical Sciences, University of Durham, Durham DH1 3LE, United Kingdom}
\email{alpar.r.meszaros@durham.ac.uk} 

\date{\today}

\begin{document}
	\begin{abstract}
     We establish the first {\it unconditional} well-posedness result for the master equation associated with a general class of {\it mean field games of controls}. Our analysis covers games with displacement monotone or Lasry--Lions monotone data, as well as those with a small time horizon. By {\it unconditional}, we mean that all assumptions are imposed solely at the level of the Lagrangian and the terminal cost. In particular, we do not require {\it any} a priori regularity or structural assumptions on the additional fixed-point mappings arising from the control interactions; instead we show that these fixed-point mappings are well-behaved as a consequence of the regularity and the monotonicity of the data. Our approach is {\it bottom-up} in nature, unlike most previous results which rely on a generalized method of characteristics. In particular, we build a classical solution of the master equation by showing that the solutions of the corresponding $N$-player Nash systems are compact, in an appropriate sense, and that their subsequential limit points must be solutions to the master equation. Compactness is obtained via uniform-in-$N$ decay estimates for derivatives of the $N$-player value functions. The underlying games are driven by non-degenerate idiosyncratic Brownian noise, and our results allow for the presence of common noise with constant intensity.
     
	\end{abstract}
	
	
	\maketitle
	{
		}
    \setcounter{tocdepth}{1}
  \tableofcontents

\section{Introduction}

\subsection*{The master equation for MFGC}

The {\it master equation} in Mean Field Games (MFG) theory was introduced by P.-L. Lions in his lectures at Coll\`ege de France. This is a PDE of hyperbolic type written for a scalar function whose variables belong to the infinite dimensional space $(0,T)\times\R^d\times\cP_2(\R^d)$.  Here $T>0$ stands for the time horizon of the game, $\R^d$ represents the state space of a typical representative agent, while the variable in $\cP_2(\R^d)$ (the set of Borel probability measures on $\R^d$ with finite second moment) encodes the distribution of the continuum of agents. The well-posedness of this PDE has emerged as one of the most important issues in MFG theory, in large part because classical solutions of the master equation can be used to obtain sharp quantitative convergence results for Nash equilibria of $N$-player differential games when $N\to+\infty$ (cf. \cite{CarDelLasLio}), as well as related concentration inequalities and large deviations results \cite{Delarue2018FromTM, DelLacRam19}. 

This paper concerns the master equations associated to {\it Mean Field Games of Controls} (MFGC); a terminology borrowed from \cite{CarLeh}. Compared to `classical' MFG where agents interact only through the distributions of their {\it states}, MFGC are more intricate, allowing interactions via both {\it states and controls}. In addition to being a mathematically interesting extension of standard MFG, MFGC are natural in economics and finance, where agents often interaction through an aggregate quantity like a price or interest rate, which is determined by the \textit{actions} of the population of agents, rather than their states. We refer to \cite{Car:21, CarLeh, GraSir, GraIgnNeu, ShrFirJai, FuGraHorPop, AidDumTan, GomGutRib} for some examples of such models.

The data for our MFGC will consist of a Hamiltonian $H : \R^d \times \R^d \times \cP_2(\R^d \times \R^d) \to \R$, which derives from a Lagrangian $ L : \R^d \times \R^d \times \cP_2(\R^d \times \R^d) \to \R$ via the classical formula
\begin{align*}
    H(x,p,\mu) := \sup_{a \in \R^d} \left\{- a \cdot p - L(x,a,\mu) \right\},
\end{align*}
together with a terminal cost $G : \R^d \times \cP_2(\R^d) \to \R$, and a constant $\sigma_0 \geq 0$ which represents the intensity of the common noise. Compared with classical MFG, the main challenge in treating MFGC is the presence of a fixed point equation set on $\cP_2(\R^d \times \R^d)$, given by
\begin{align} \label{def.Phi}
 \mu = \Big( (x,p) \mapsto \big(x, - D_p H(x, p, \mu)\big) \Big)_{\#} \nu. 
\end{align}
In order to make sense of the MFGC master equation, we will need to know that \eqref{def.Phi} has a unique solution $\mu$ for each fixed $\nu$. The resulting solution mapping $\nu \mapsto \Phi(\nu) = \mu$ plays a key role in our story. Indeed, as we will discuss below, the solvability of \eqref{def.Phi} is the analogue in MFG theory of the \textit{generalized Isaacs condition} for $N$-player games. If $\Phi$ is well-defined, the master equation can be written as
\begin{align} \label{ME} \tag{ME}
    \begin{cases}
        \ds - \partial_t U - \Delta_{\id} U - \sigma_0 \Delta_{\com } U + \hat{H}\Big(x, D_x U(t,x,m), \big(\text{Id}, D_xU(t,\cdot,m) \big)_{\#} m \Big) \vspace{.2cm}
        \\
        \ds \qquad + \int_{\R^d} D_p \hat{H}\Big(y,D_xU(t,y,m), \big(\text{Id}, D_xU(t,\cdot,m) \big)_{\#} m \Big) \cdot D_m U(t,x,m,y) m(dy) = 0 
        \vspace{.2cm} \\ 
        \qquad \qquad \qquad (t,x,m) \in (0,T) \times \R^d \times \cP_2(\R^d), 
        \vspace{.2cm} \\ \ds 
        U(T,x,m) = G(x,m), \quad (x,m) \in \R^d \times \cP_2(\R^d),
    \end{cases}
\end{align}
where throughout the paper we use the notation $\text{Id}:\R^d\to\R^d$ to denote the identity map and we define $\hat{H} : \R^d \times \R^d \times \cP_2(\R^d \times \R^d) \to \R$ by
\begin{align} \label{HatHdef}
    \hat{H}(x,p,\nu) := H\big( x,p, \Phi(\nu) \big).
\end{align}
In \eqref{ME}, $\Delta_{\id}$ and $\Delta_{\com}$ are the idiosyncratic and common noise operators, defined by 
\begin{align*}
    &\Delta_{\id} U (t,x,m) = \Delta_x U(t,x,m) + \int_{\R^d} \tr\big( D_{ym} U(t,x,m,y) \big)m(dy), 
    \\
    &\Delta_{\com}U(t,x,m) = \Delta_{\id}U(t,x,m) + \int_{\R^d} \int_{\R^d} \tr\big(D_{mm} U(t,x,m,y,y') \big) m(dy) m(dy')
    \\
    &\qquad \qquad \qquad \qquad + 2 \int_{\R^d} \tr\big(D_{xm} U(t,x,m,y) \big) m(dy).
\end{align*}
In this paper, we will avoid giving a rigorous discussion the corresponding mean field and $N$-player stochastic differential games for the sake of brevity, and refer to \cite{JacMes} for details. Informally, a representative player in this game aims to minimize a cost functional of the form
\begin{align*}
   J_{\mu}(\alpha) =  \E\bigg[ \int_{t_0}^T L\big(X_t, \alpha_t, \mu_t \big) dt + G(X_T, \mu_T^x) \bigg],
\end{align*}
subject to the dynamics
\begin{align*}
    dX_t = \alpha_t dt + \sqrt{2} dW_t + \sqrt{2\sigma_0} dW_t^0, \quad X_{t_0} \sim m_0.
\end{align*}
Here $W$ and $W^0$ are independent Brownian motions, and the parameter $(\mu_{t})_{t\in[t_{0},T]}$ is a $W^0$-adapted $\cP_2(\R^d \times \R^d)$-valued process representing the joint distribution of the states and actions of the population of agents, conditionally on the common signal $W^0$. We also used the notation $\mu^x_{T}$ for the first marginal of $\mu_{T}$. An equilibrium is a flow of measures $(\mu)_{t\in [t_{0},T]}$ such that the for some optimizer $(\alpha_{t})_{t\in [t_{0},T]}$ of this optimization problem, $\mu_t = \cL((X_t,\alpha_t) | W^0)$, $t\in[t_{0},T]$. We refer to e.g. \cite{CarDel:vol2} for a thorough discussion of the role played by the master equation in MFG theory, and emphasize only that $U(t,x,m)$ can be interpreted as the value of a representative player at position $x$ and at time $t$, given that the continuum of agents is initially distributed according to $m$.

Compared to the usual MFG master equation, the main new challenge for MFGC is that the non-linearity appearing in the master equation depends implicitly on the fixed-point relation \eqref{def.Phi}, through the map $\Phi$. In order to obtain regularity estimates for $U$, it thus becomes crucial to understand the equation \eqref{def.Phi}, and in particular to understand under what conditions the new, implicitly defined Hamiltonian $\hat{H}$ is smooth. Showing the regularity of $\Phi$ which gives sufficient smoothness and monotonicity for $\hat{H}$ is one of the main results of the present work. To the best of our knowledge, existing results either address only specific classes of models (see, e.g., \cite{GraSir}) or rely on assumptions that are {\it conditional} on the regularity of the fixed-point map $\Phi$ or the implicitly defined $\hat{H}$ (see, e.g., \cite{MouZha:2022, LiaMou, LiuLiuMouSun}). To the best of the authors' knowledge, such conditions on $\Phi$ or $\hat{H}$ have only been rigorously verified when the interactions are in some sense finite-dimensional or under the structural condition $H(x,p,\mu) = H_0(x,\mu) + H_1(x,p,\mu^x)$, with $\mu^x$ the first marginal of $\mu$ (in which case the fixed-point equation becomes trivial). A main contribution of this work is to show that the smoothness of $\Phi$ and $\hat{H}$ is a \textit{consequence} of the regularity and monotonicity of $L$, rather than an additional assumption which must be imposed.

\subsection*{The $N$-player Nash system}

In addition to the master equation \eqref{ME}, we will study the systems of PDEs which describe equilibria of the corresponding $N$-player stochastic differential games. Unlike when the interactions are only through the states, in order to write down the $N$-player Nash system in this setting, we must first verify that the so-called ``generalized Isaacs condition" is satisfied. This means that for each $\bx = (x^1,\dots,x^N), \bp = (p^1,\dots,p^N) \in (\R^d)^N$, we can find a unique point $\ba = (a^1,\dots,a^N) \in (\R^d)^N$ such that 
\begin{align} \label{a.fixedpoint.1}
    a^i \in \text{argmax} \left\{ a \mapsto - a \cdot p^i - L(x^i,a,m_{\bx,\ba}^{N,-i}) \right\}
\end{align}
for all $i = 1,\dots, N$. Here, for $n\in\N$ and $\bz=(z^1,\dots,z^N)\in(\R^n)^N$, we use the notation $m_{\bz}^{N,-i}:=\frac{1}{N-1}\sum_{j\neq i}\delta_{z^j}$. See e.g. Chapter 5 of \cite{carmonabsde} for an explanation of the role of this generalized Isaacs condition in $N$-player stochastic differential game theory.  Under mild convexity and regularity conditions on $L$, this property is equivalent to 
\begin{align} \label{fixedpointintro}
    a^i = - D_pH \left( x^i, p^i, m_{\bx,\ba}^{N,-i}\right).
\end{align}
for each $i=1,\dots,N$ and $\bx,\bp \in (\R^d)^N$. If the equation \eqref{fixedpointintro} has a unique solution for each $\bx, \bp \in (\R^d)^N$, we obtain a well-defined solution mapping $\ba^N : (\R^d)^N \times (\R^d)^N \to (\R^d)^N$, and then the $N$-player Nash system can be written as follows:
\begin{align} \label{nashsystem}
    \begin{cases}
       \ds  - \partial_t u^{N,i} - \sum_{j = 1}^N \Delta_{j} u^{N,i} - \sigma_0 \sum_{j,k = 1}^N \tr\big(D_{jk} u^{N,i} \big) + H\big(x^i, D_{i} u^{N,i}, m_{\bx, \ba^N(\bx, \diag \bu^N)}^{N,-i} \big)
       \\
       \ds \qquad + \sum_{j \neq i} D_p H\big(x^j, D_j u^{N,j}, m_{\bx, \ba^N(\bx, \diag \bu^N)}^{N,-j} \big) \cdot D_j u^{N,i} = 0, 
      \quad  (t,\bx) \in (0,T) \times (\R^d)^N, 
       \\
       \ds u^{N,i}(T,\bx) = G(x^i,m_{\bx}^{N,-i}), \quad \bx \in (\R^d)^N, 
    \end{cases}
\end{align}
where we use the notation 
\begin{align*}
    \diag \bu^N = (D_1u^{N,1},\dots,D_N u^{N,N}).
\end{align*}
The solution is a tuple $(u^{N,i})_{i = 1,\dots,N}$, with $u^{N,i} : [0,T] \times (\R^d)^N \to \R$, and $u^i(t_0,\bx_0)$ represents the value in equilibrium of a stochastic differential game in which player $i$ chooses a feedback $\alpha^i : [t_0,T] \times (\R^d)^N \to \R^d$, and aims to minimize a cost function of the form
\begin{align*}
    J^{N,i}(\alpha^1,\dots,\alpha^N) = \E\bigg[ \int_{t_0}^T L\Big(X_t^i,\alpha^i(t,\bX_t), m_{\bX_t, \bm{\alpha}(t,\bX_t)}^{N,-i} \Big) dt + G(X_T^i,m_{\bX_T}^{N,-i}) \bigg], 
\end{align*}
subject to the dynamics 
\begin{align*}
    dX_t^i = \alpha^i(t,\bX_t) dt + \sqrt{2} dW_t^i + \sqrt{2\sigma_0} dW_t^0,
\end{align*}
with $W^{0}$ and $(W^i)_{i \in \N}$ being independent Brownian motions. Given a smooth enough solution to \eqref{nashsystem}, we can produce an equilibrium for the relevant game via the formula
\begin{align*}
    \alpha^i(t,\bx) = - D_p H\big(x^i,D_i u^{N,i}(t,\bx), m_{\bx,\ba^N(\bx,\diag \bu^N(t,\bx))}^{N,-i}\big) = a^{N,i}\big(\bx, \diag \bu^N(t,\bx)\big).
\end{align*}
In particular, the equilibrium trajectories $\bX = (X^1,\dots,X^N)$ for the corresponding $N$-player game evolve according to the dynamics 
\begin{align*}
    dX_t^i = a^{N,i}\big(\bx, \diag \bu^N(t,\bX_t)\big) dt + dW_t^i + \sqrt{2\sigma_0} dW_t^0, 
\end{align*}
where $W$ and $(W^i)_{i=  1,\dots,N}$ are independent Brownian motions which drive the game; we again refer to \cite{JacMes} for details on the $N$-player games corresponding to \eqref{nashsystem}. 

At first glance, it may not be obvious in what sense the map $\ba^N$ relates to the map $\Phi$. The heuristic connection is that formally $m^N_{\bx,\ba^N(\bx,\bp)}$ satisfies 
\begin{align*}
    m_{\bx,\ba^N(\bx,\bp)}^N \approx \Big( (x,p) \mapsto \big(x, - D_p H(x, p, m^N_{\bx,\ba^N(\bx,\bp)})\big) \Big)_{\#} m_{\bx,\bp}^N, 
\end{align*}
and so we would expect that $m_{\bx,\ba^N(\bx,\bp)}^N \approx \Phi\big( m_{\bx,\bp}^N\big)$ for $N$ large. In this sense, the solvability of \eqref{def.Phi} is the mean-field analogue of the solvability of \eqref{fixedpointintro}, and so the solvability of \eqref{def.Phi} is a sort of mean field analogue of the generalized Isaacs condition. Meanwhile, the connection between the master equation and the $N$-player Nash system is that we expect 
\begin{align*}
    u^{N,i}(t,\bx) \approx U(t,x^i,m_{\bx}^{N,-i}), \quad \text{for } N \gg 1.
\end{align*}

\subsection*{Related literature} We now provide a brief survey of existing results on the master equation, first for ``standard" MFG and then for MFGC.
\newline \newline 
\textit{Standard Mean Field Games.} The existence and uniqueness of classical solutions to master equations associated with MFGs without interactions through controls are by now relatively well understood. Assuming sufficient regularity of the data $H$ and $G$, well-posedness can be established under suitable smallness conditions -- typically on the time horizon $T$ -- even in the absence of non-degenerate idiosyncratic noise. We refer to the non-exhaustive list of works \cite{GanSwi, CarDel:vol2, May, CarCirPor, AmbMes} for results obtained under such assumptions.

To ensure global well-posedness of the master equation, additional structural conditions -- typically {\it monotonicity conditions} on $H,G$ -- are needed. Historically, the first such global well-posedness results were established in the breakthrough works \cite{ChaCriDel22, CarDelLasLio} in the presence of non-degenerate idiosyncratic noise and under the so-called {\it Lasry--Lions monotonicity (LL-monotonicity) conditions} on $G$ and on {\it separable} $H$. We refer also to the recent work \cite{JakRut} for some clarifications and extensions to nonlocal diffusion settings. The so-called {\it displacement monotonicity} (D-monotonicity; -- \cite{Ahu, MesMou} -- stemming from displacement convexity widely used in optimal transport, cf. \cite{McC,Par}) framework is another influential setup leading to the global well-posedness of classical master equations. D-monotonicity leads also to quantified long time behavior and quantitative convergence of equilibria with finitely many agents when the number of agents tends to infinity (\cite{CirMes, JacTan,JacMes}). This allows in general non-separable Hamiltonians and possibly degenerate idiosyncratic noise, see \cite{GanMes, GanMesMouZha, BanMesMou, BanMes:25-FMS}. We mention also the recent work \cite{GraMes:23} that proposes further two monotonicity conditions (distinct from the LL- and D-monotonicity settings) and \cite{GraMes:24} which presents a global well-posedness theory for a toy deterministic master equation within one of these additional frameworks.
\newline \newline 
\textit{Mean Field Games of Controls.} MFGC were first introduced in \cite{GomPatVos, GomVos:13, GomVos:16} under the terminology of {\it extended MFGs} (the terminology of {\it MFG of controls} became widely used since \cite{CarLeh}) and ever since they continue to receive great attention both in applications (in macroeconomics, mathematical finance, modeling of energy markets, etc.) and in mathematical investigations. By now significant advancement has been made in this theory both from the PDE and probabilistic viewpoints when it comes to existence, uniqueness (cf. \cite{Kob:22-cpde, Kob:22-nodea, AchKob, GraMulPfe, GraMat, GraMatRui, GraRos, CD1, CarLeh}) and convergence of Nash equilibria for games with a finite number of player, when the number of agents tends to infinity (cf. \cite{Dje:22, Dje:23-esaim, Dje:23-aap, LauTan, PosTan, JacMes}).

In contrast, classical approaches do not extend in a straightforward manner to the study of master equations arising in MFGC, even when smallness conditions are imposed. Therefore, the literature on the master equation for MFGCs is much more limited. The underlying difficulty is structural: the relevant data are $(\hat H,G)$ -- or equivalently $(H,G,\Phi)$ -- rather than $(H,G)$ alone, and obtaining the required smoothness properties for $\Phi$ is itself an intrinsic part of the problem. As a consequence, existing techniques fail to yield well-posedness results for MFGC master equations in general. In \cite{GraSir}, the authors successfully adapt techniques from \cite{CarDelLasLio} to address a particular model with interactions through a finite-dimensional quantity. Beyond this particular framework, there are only a couple of works addressing the solvability of the master equation in MFGC, \cite{MouZha:2022,LiaMou, LiuLiuMouSun}. The main well-posedness results in these papers are {\it conditional on the smoothness of $\Phi$}: assuming the smoothness of the fixed point maps (or the implicitly defined Hamiltonian $\hat{H}$) these works study the propagation of the LL- and D-monotonicities, which lead to the global solvability of the master equation. We emphasize again that we are not aware of any sufficient conditions for these smoothess conditions on $\Phi$ or $\hat{H}$, except in the separable case $H(x,p,\mu) = H_0(x,\mu) + H_1(x,p,\mu^x)$, with $\mu^x$ being the first marginal of $\mu$. The well-posedness of the master equation for MFGC has thus remained an open and challenging problem, even under smallness assumptions, until the present work.

\subsection*{Main results and the novelties of our approach}

The main results of this work cover both regimes of LL- and D-monotone data, and can be informally summarized as follows (for the specific statements we refer to Theorems \ref{thm.main.disp} and \ref{thm.main.LL} below).

\begin{theorem}[Informal summary of main results] \label{thm:intro}
    Suppose that the Lagrangian $L$ and terminal cost function $G$ satisfy suitable regularity and growth properties at infinity.
    \begin{enumerate}
    \item Suppose furthermore that $L$ and $G$ satisfy a D-semi-monotonicity assumption. Then the master equation \eqref{ME} has a unique classical solution on a time interval that depends only on the D-semi-monotonicity constant. In particular, if the data are D-monotone, the time horizon can be taken arbitrary long.
    \item If instead $L$ and $G$ satisfy the LL-monotonicity assumption, then the master equation \eqref{ME} has a unique classical solution, for arbitrary long time horizons.
    \end{enumerate}
\end{theorem}

\medskip

Many previous results on the existence of classical solutions to the master equation rely on a generalized method of characteristics. For example, the seminal work \cite{CarDelLasLio} (in the setting of standard MFGs) proceeds by defining a candidate solution $U$ as $U(t_0,x_0,m_0) := u^{t_0,m_0}_{t_0}(x_0)$, where $(u^{t_0,m_0}, m^{t_0,m_0})$ is the solution of the (stochastic) MFG system, a forward-backward system of PDEs which characterizes the equilibria of the game starting from initial conditions $(t_0,m_0)$. Then it is shown that (i) if this candidate solution $U$ is smooth enough, then it solves the master equation and (ii) smoothness can be rigorously obtained via stability estimates for the MFG system and its ``linearizations". It is also possible to use the same general approach, but replace the MFG system with a forward-backward system of  McKean--Vlasov SDEs, see e.g. \cite{ChaCriDel22}. The works \cite{CarDelLasLio, ChaCriDel22, CarDel:vol2, May, AmbMes} all use some variation of this generalized method of characteristics to build classical solutions to the master equation. Alternatively, one can first use the method of characteristics to build a short-time solution, then patch together local solutions via suitable a-priori estimates (cf. \cite{CarDel:vol2, GanMesMouZha}). Apart from the method of characteristics, another approach has been developed in \cite{GanMes, LiaMesMouZho} in the special case of potential mean field games, in which regularity for $U$ is obtained via estimates on discretizations of the corresponding mean field control problem. This last approach is the closest in spirit to the one pursued in this paper. Finally, \cite{CarCirPor} introduces a novel ``splitting method" to produce short-time solutions for several version of the master equation.

Our strategy for producing classical solutions to \eqref{ME} is based on discretization and compactness, rather than the method of characteristics. In particular, we first obtain uniform derivative decay estimates both on $u^{N,i}$ and $a^{N,i}$. We then argue that these estimates imply that $(u^{N,i})_{i = 1,\dots,N}$ converges along a subsequence, in an appropriate sense, to a classical solution $U$ of the master equation \eqref{ME}. The estimates we obtain on the $N$-player Nash system are a continuation of the ones obtained in our earlier paper \cite{JacMes}, and similar to the ones appearing (without interaction through the controls) in \cite{CirRed, CirJacRed}, but we need finer estimates here in order to execute our compactness argument. We will now describe the roadmap towards the proof of our main theorem.
\newline \newline 
\noindent 
\textit{Step 1 - uniform in $N$ derivative decay estimates on $\ba^{N}.$} The first main result of this work addresses the regularity of the fixed point maps. Since $\Phi:\cP_{2}(\R^{d}\times\R^{d})\to\cP_{2}(\R^{d}\times\R^{d})$ is a nonlinear map, mapping an infinite dimensional space to itself, the study of its regularity is highly nontrivial. As $\cP_{2}(\R^{d}\times\R^{d})$ is not a linear space, and the regularity of $\Phi$ needs to be established in the intrinsic sense, a natural way would be to consider $\cP_{2}(\R^{d}\times\R^{d})$ as an infinite dimensional manifold which is positively curved in the sense of Alexandrov. We bypass such potentially intricate and involved arguments by studying instead the regularity of the {\it finite dimensional fixed point maps} $\ba^{N}:(\R^{d})^{N}\times(\R^{d})^{N}\to (\R^{d})^{N}\times(\R^{d})^{N}$, defined in \eqref{fixedpointintro}, uniformly in $N$, and showing that these maps converge to $\Phi$ in some sense.

Relying on the analysis performed in \cite{JacMes}, we have that the (D- or LL-) monotonicity assumption on $L$ implies that for $N$ large enough the fixed point relation \eqref{fixedpointintro} has a unique solution which is Lipschitz continuous, uniformly in $N$. Furthermore, the fixed point equation \eqref{def.Phi} has a unique Lipschitz continuous solution, and $\ba^{N}$ converges $\Phi$ in a suitable sense. A careful rewriting of the monotonicity condition on $L$ for empirical probability measures yields that $\ba^{N}$ can be defined via {\it monotone operators} on $(\R^{d})^{N}\times (\R^{d})^{N}$. This crucial observation leads to apply a global implicit function theorem. In particular this in addition implies that the maps $\ba^{N}$ are as smooth as the data. Performing an implicit differentiation and a delicate bootstrapping procedure, we establish that derivatives of $\ba^{N}$ up to order three decay uniformly as $N^{-1}$, $N^{-2}$, etc. These results are collected in Proposition \ref{prop.aderivscaling}.
\newline \newline 
\noindent {\it Step 2 - uniform in $N$ derivative decay estimates on $u^{N,i}$ under D-semi-monotonicity}. Armed with the uniform in $N$ derivative estimates on $\ba^{N}$, we can turn to the Nash system \eqref{nashsystem}. We show that the composition of the Hamiltonian and the fixed point maps $\ba^{N}$ have similar derivative decay estimates up to order three. This then lets us obtain the desired derivative estimates of order up to three on $u^{N,i}$. These results are summarized in Theorem \ref{thm.uniform.nash}. Since the master equation is of second order, it is necessary to obtain derivatives up to order three on the $u^{N,i}$ functions. This analysis is very delicate and it requires carefully considering derivatives and their representation formulas via FBSDEs in a well designed precise {\it order}, and in appropriate norms. For instance, we first obtain appropriate $L^{\infty}$ bounds on derivatives of the form $D_j u^{N,i}$, $i \neq j$, together with $L^2$-type bounds \textit{along equilibrium trajectories} for $(D_{kj} u^{N,i})_{k = 1,\dots,N}$, then these estimates together are used to obtain a uniform bound on derivatives of the form $D_{ji} u^{N,i}$ for $i \neq j$, and these in turn play a key role in estimates of other second order derivatives.

\medskip

\noindent {\it Step 3 - existence in the D-semi-monotone case via compactness arguments.} We show that a solution to the master equation \eqref{ME} can be obtained as the limit of $u^{N,i}$; roughly speaking, this comes from applying an Arzel\`a--Ascoli type argument simultaneously to the functions $u^{N,i}$ and their derivatives up to order $2$. We rely on the philosophy that $\cP_{2}(\R^{d})$ can be seen as a suitable limit of the spaces $(\R^{d})^{N}$ quotiented with respect to the equivalence relation that any two vectors $(x_{1},\dots,x_{N}), (x_{\sigma(1)},\dots,x_{\sigma(N)})\in (\R^{d})^{N}$ are equivalent for any permutation $\sigma:\{1,\dots,N\}\to\{1,\dots,N\}.$
We show that the previously established uniform decay estimates on the derivatives of $u^{N,i}$ provide the sufficient compactness arguments to perform this limiting process and as a result we obtain the first part of our main Theorem \ref{thm:intro}.

\medskip

\noindent {\it Step 4 - Existence in the LL-monotone case.} This case is in fact a straightforward consequence of the D-semi-monotone case. First, if $L$ and $G$ and LL-monotone, with uniformly bounded second order derivatives, then they are also D-semi-monotone, and so we obtain a local in time classical solution. Our strategy then is to use propagation of LL-monotonicity to show that any local in time classical solution to master equation in fact remains uniformly D-semi-monotone. Then, any such local solution can be extended to a global in time solution, which is LL-monotone and D-semi-monotone.
\medskip

\subsection*{The structure of the rest of the paper.} In Section \ref{sec:2} we recall some classical notation, list all of our standing assumptions on $L$ and $G$ (separately for the D-semi-monotone and LL-monotone cases) and state the main theorems of the paper. Section \ref{sec:3} is dedicated to the study of the quantitative derivative estimates on the finite dimensional fixed point maps $\ba^{N}$. Sections \ref{sec.uniform} and \ref{sec:5} can be seen as the main parts of the paper. In Section \ref{sec.uniform} we establish the uniform in $N$ derivative decay estimates for the solutions $u^{N,i}$ of the $N$-player Nash system, while in Section \ref{sec:5} we prove the first part of Theorem \ref{thm:intro}, showing the existence and uniqueness of a solution to the master equation in the case of D-semi-monotone data. Section \ref{sec:6} ends the main part of the paper, and establishes the proof of the second part of Theorem \ref{thm:intro} in the case of LL-monotone data. The paper ends with two appendix sections. First, in Appendix \ref{app:A} we establish the existence of a solution for the $N$-player Nash system, under our standing assumptions, while Appendix \ref{app:B} contains the proof of two technical lemmas from the main part.
\medskip

\noindent {\bf Acknowledgments.} J.J. is supported by the NSF under Grant No. DMS2302703. A.R.M. has been supported by the EPSRC New Investigator Award ``Mean Field Games and Master equations'' under award no. EP/X020320/1.

\section{Notation, assumptions and main results}\label{sec:2}

\subsection{Notation} We denote by $\cP(\R^d)$ the space of Borel probability measures supported on $\R^d$, and by $\cP_2(\R^d)$ the subset of  $\cP(\R^d)$ such that 
\begin{align*}
    M_2(m) \coloneqq \int_{\R^d} |x|^2 m(dx) < \infty,\ \ \forall m\in \cP_{2}(\R^{d}). 
\end{align*}
We denote by $\bd_1$ and $\bd_2$ the standard 1-Wasserstein and 2-Wasserstein distances on $\cP_2(\R^d)$. We use analogous notation for $\cP(\R^d \times \R^d)$ and $\cP_2(\R^d \times \R^d)$. With a slight abuse of notation we denote by $\bd_1$ and $\bd_2$ the 1-Wasserstein and 2-Wasserstein distances also on $\cP_2(\R^d \times \R^d)$. Given a function $\Psi : \cP_2(\R^d) \to \R$, we denote by $D_m \Psi$ the Wasserstein or intrinsic derivative, see e.g. \cite[Section 5]{CD1}. In particular, if $\Psi$ is $C^1$, then $D_m \Psi$ is a map
\begin{align*}
  \cP_2(\R^d) \times \R^d \ni  (m,y) \mapsto D_m \Psi(m,y) \in \R^d.
\end{align*}
We also use similar notation for higher derivatives, e.g. if $m \mapsto D_m \Psi$ is differentiable, then $D_{mm} \Psi$ is a map of the form
\begin{align*}
    \cP_2(\R^d) \times \R^d \times \R^d \ni (m,y,z) \mapsto D_{mm} \Psi(m,y,z) \coloneqq D_m \Big[ D_m \Psi(\cdot ,y) \Big](m,z) \in \R^{d \times d}.
\end{align*}
For functions which depend on a Euclidean parameter and a measure, e.g. for functions $\Psi : \R^d \times \cP_2(\R^d) \to \R$, smoothness is understood with respect to both the position variable $x$ and the measure variable $m$, e.g. in this case $\Psi$ is $C^2$ if the first derivatives $D_x \Psi$ and $D_m \Psi$, as well as the second derivatives
\begin{align*}
   D_{xx} \Psi : \R^d \times \cP_2(\R^d) \to \R^{d \times d}, \quad  D_{xm} \Psi : \R^d \times \cP_2(\R^d) \times \R^{d \times d} \to \R, \quad D_{mm} \Psi : \R^d \times \cP_2(\R^d) \to \R^{d \times d}.
\end{align*}
The notation $C^k$ is used similarly. Again, very similar notation will be used for functions defined on $\cP_2(\R^d \times \R^d)$ instead of $\cP_2(\R^d)$. We note that if $\Psi : \cP_2(\R^d \times \R^d) \to \R$, $D_{\mu} \Psi$ takes values in $\R^d \times \R^d$. At times, it will be convenient to write it as $(D_{\mu}^x \Psi, D_{\mu}^a \Psi)$, where $D_{\mu}^x \Psi$ and $D_{\mu}^a \Psi$ each take values in $\R^d$.

We will also be working extensively with functions defined on $(\R^d)^N$ or $(\R^d \times \R^d)^N$. We use bold for elements of these spaces, and superscripts for their $\R^d$-valued components, e.g. we write $\bx = (x^1,\dots,x^N) \in (\R^d)^N$, and each $x^i$ can be further expanded as $x^i = (x^i_1,\dots,x^i_d) \in \R^d$. Given $\phi : (\R^d)^N \to \R$, we write $D_i \phi$ or $D_{x^i} \phi$ for the gradient in the direction $x^i$, i.e. $D_i \phi(\bx) = (D_{x^i_1} \phi(\bx),\dots,D_{x^i_d} \phi(\bx)) \in \R^d$. A similar notation is used for the second derivatives $D_{ij} \phi = D_{x^ix^j} \phi \in \R^{d \times d}$, i.e. 
\begin{align*}
    \left[ D_{ij} \phi(\bx) \right]_{p,q} = D_{x^i_q x^j_p} \phi(\bx), \quad i,j = 1,\dots,N, \quad p,q = 1,\dots,d.
\end{align*}

\subsection{Monotonicity and regularity assumptions}

We start with some regularity assumptions imposed on the data $L$, $H$, and $G$.

\begin{assumption}[Regularity] \label{assump.regularity.disp}
The terminal condition $G$ is of class $C^5$ and bounded from below, with all derivatives of order $2,3$, $4$, and $5$ being uniformly bounded, and in addition the first derivative $D_m G$ is uniformly bounded. The functions $L$ and $H$ are of class $C^5$, with derivatives of order $2$, $3$, $4$, and $5$ uniformly bounded, and in addition $D_{\mu} L$ is bounded. Moreover, $L$ is uniformly strictly convex in the control variable $a$, i.e. there is a constant $C> 0$ such that 
\begin{align} \label{strictconvexity}
    D_{aa} L(x,a,\mu) \geq C I_{d \times d}, \,\, \text{for all } (x,a,\mu) \in \R^d \times \R^d \times \cP_2(\R^d \times \R^d),
\end{align}
and we also have the coercivity/growth condition 
\begin{align} \label{coercive}
  \frac{1}{C} |a|^2 - C \leq  L(x,a,\mu) \leq C\Big( 1 + |x|^2 + |a|^2 +  M_2^{1/2}(\mu) \Big).
\end{align}
Finally, the map
\begin{align*}
 \R^d \times \R^d \times \cP_2(\R^d \times \R^d) \times \R^d \times \R^d \ni  (x,a,\mu,x',a') \mapsto \frac{\delta}{\delta m} D_x L \big(x,a,\mu,x',a')
\end{align*}
is uniformly Lipschitz continuous, and likewise $\frac{\delta}{\delta m} D_a L$ and $\frac{\delta}{\delta m} D_x G$ are uniformly Lipschitz continuous, with Lipschitz continuity being understood with respect to $\bd_1$.
\end{assumption}

Notice that Assumption \ref{assump.regularity.disp} allows $G$ and $L$ to have quadratic growth in $x$, since $D_x G$ and $D_x L$ are not supposed to be bounded. In the displacement monotone case, we will see that it is possible to produce a classical solution under these conditions, but in the Lasry--Lions monotone case we will need Lipschitz bounds, which we will be obtained under the following strengthening of Assumption \ref{assump.regularity.disp}. 

\begin{assumption}[Regularity + Lipschitz] \label{assump.regularity.LL}
    Assumption \ref{assump.regularity.disp} holds, and in addition there is a constant $C>0$ such that we have the bounds
    \begin{align*}
        |D_xH(x,p,\mu)| \leq C(1 + |p|), \quad |D_x G(x,\nu)| \leq C,\ \ \forall x,p\in\R^{d},\ \forall \mu\in\cP_{2}(\R^{d}\times\R^{d}),\ \forall \nu\in\cP_{2}(\R^{d}).
    \end{align*}
\end{assumption}

The next assumption will allow us to prove that the maps $\Phi$ and $\ba^N$ discussed above are well-defined (at least for large enough $N$, in the latter case) and satisfy some important regularity estimates.

\begin{assumption} \label{assump.fixedpoint}
   There are constants $C_{L,a} > 0$, $C_{L,x} \geq 0$ such that $L$ satisfies
    \begin{align} \label{dispmonotone.L}
    &\E\bigg[ \Big( D_aL\big(X,\alpha, \cL(X,\alpha)\big) - D_aL\big(X',\alpha', \cL(X',\ov{\alpha})\big)  \Big) \cdot \Delta \alpha
   \nonumber  \\
    &\qquad + \Big(  D_x L\big(X,\alpha, \cL(X,\alpha)\big) - D_x L\big(X',\alpha', \cL(X',\alpha')\big) \Big) \cdot \Delta X \bigg] \geq C_{L,a} \E\big[ |\Delta \alpha|^2 \big] - C_{L,x} \E\big[ |\Delta X|^2 \big], 
    \nonumber \\
    &\text{ where } \Delta X \coloneqq X - X', \quad \Delta \alpha \coloneqq \alpha - \alpha',
\end{align}
for all square-integrable random vectors $X, X',\alpha, \alpha'$.
\end{assumption}

We next state two monotonicity conditions, which are both stronger than Assumption \ref{assump.fixedpoint}. The first is a displacement semi-monotonicity condition, and the second is a Lasry--Lions monotonicity condition. We will construct classical solutions under either of the two conditions.

\begin{assumption}[Displacement semi-monotonicity] \label{assump.disp}
There are constants $C_{L,a} > 0, C_{L,x}, C_G \geq 0$ such that $L$ satisfies \eqref{dispmonotone.L} and $G$ satisfies
\begin{align} \label{dispmonotone.G}
    \E\bigg[ \Big( D_x G\big(X, \cL(X) \big) - D_x G \big(X', \cL(X') \big) \Big) \cdot  \Delta X \bigg] \geq -C_G \E\big[ |\Delta X|^2 \big], \quad \text{where } \Delta X \coloneqq X - X'
\end{align}
for all square-integrable $\R^d$-valued random vectors $X,X', \alpha, \alpha'$. Moreover, the constants $C_{L,a}$, $C_{L,x}$, and $C_G$ satisfy
\begin{align*}
    C_{\dis} \coloneqq C_{L,a} - T C_G - \frac{T^2}{2} C_{L,x} > 0.
\end{align*}
\end{assumption}

\begin{remark} \label{rmk.dmonotone.secondorder}
    Following computations in \cite{MesMou} (see Remark 2.1 therein), one can show that under the regularity Assumption \ref{assump.regularity.disp}, \eqref{dispmonotone.L} is equivalent to the following: given square-integrable random vectors $X, \alpha, Y, \beta, \hat{X}, \hat{\alpha},  \hat{Y}, \hat{\beta}$ such that $(X,\alpha, Y, \beta)$ and $(\hat{X}, \hat{\alpha}, \hat{Y}, \hat{\beta})$ are i.i.d., we have 
    \begin{align*}
        &\E\bigg[ \begin{pmatrix} Y \\ \beta \end{pmatrix}^\top D_{x,a}^2 L\big(X,\alpha, \cL(X,\alpha) \big) \begin{pmatrix} Y \\ \beta \end{pmatrix} + \begin{pmatrix} Y \\ \beta \end{pmatrix}^\top D_{\mu} D_{x,a} L\big(X,\alpha, \cL(X,\alpha), \hat{X},\hat{\alpha}\big) \begin{pmatrix} \hat{Y} \\ \hat{\beta} \end{pmatrix} \bigg]
        \\
        &\qquad \qquad \qquad \geq C_{L,a} \E\big[ |\beta|^2 \big] - C_{L,x} \E\big[ |Y|^2 \big],
    \end{align*}
    where 
    \begin{align*}
        D_{x,a}^2 L = \begin{pmatrix}
            D_{xx} L & D_{ax} L 
            \\
            D_{xa} L & D_{aa} L
        \end{pmatrix}, \quad
        D_{\mu} D_{x,a} L = \begin{pmatrix}
            D_{\mu}^x D_x L & D_{\mu}^a D_x L 
            \\
            D_{\mu}^x D_a L & D_{\mu}^a D_a L.
        \end{pmatrix}
    \end{align*}
    Similarly, \eqref{dispmonotone.G} is equivalent to 
    \begin{align*}
        \E\big[ Y^\top D_{xx} G\big(X, \cL(X) \big) Y + Y^\top D_{\mu} D_x G\big(X,\cL(X), \hat{X}\big) \hat{Y} \big] \geq C_G \E[|Y|^2]
    \end{align*}
    for all square-integrable random vectors $X,Y,\hat{X},\hat{Y}$ such that $(X,Y)$ and $(\hat{X}, \hat{Y})$ are i.i.d. 
\end{remark}

Finally, we will also consider the Lasry--Lions monotonicity conditions.

\begin{assumption}[Lasry--Lions monotonicity] \label{assump.LL}
The Lagrangian $L$ satisfies 
\begin{align} \label{llmonotone.L}
    \int_{\R^d} \int_{\R^d} \left[ L(x,a,\mu) - L(x,a,\mu') \right] d(\mu - \mu')(x,a) \geq 0
\end{align}
for all $\mu,\mu' \in \cP_2(\R^d \times \R^d)$. In addition, $G$ satisfies
\begin{align} \label{llmonotone.g}
    \int_{\R^d} \left[ G(x,m) - G(x,m') \right] d(m - m')(x) \geq 0, 
\end{align}
for each $m,m' \in \cP_2(\R^d)$.
\end{assumption}

\begin{remark} \label{rmk.llmonotone.secondorder}
   Following computations in \cite{MesMou} (see Remark 2.1 therein) one can show that under Assumption \ref{assump.regularity.disp}, \eqref{llmonotone.L} is equivalent to 
    \begin{align*}
        \E\bigg[ \begin{pmatrix} Y \\ \beta \end{pmatrix}^\top D_{\mu} D_{x,a} L\big(X,\alpha, \cL(X,\alpha), \hat{X},\hat{\alpha}\big) \begin{pmatrix} \hat{Y}  \\ \hat{\beta} \end{pmatrix} \Big] \geq 0
    \end{align*}
   for any square-integrable random vectors $X, \alpha, Y, \beta, \hat{X}, \hat{\alpha},  \hat{Y}, \hat{\beta}$ such that $(X,\alpha, Y, \beta)$ and $(\hat{X}, \hat{\alpha}, \hat{Y}, \hat{\beta})$ are i.i.d. Similarly, \eqref{llmonotone.g} is equivalent to  \begin{align*}
        \E\big[ Y^\top D_{\mu} D_x G\big(X,\alpha, \cL(X,\alpha), \hat{X}, \hat{\alpha}\big) \hat{Y} \big] \geq 0
    \end{align*}
    for all square-integrable random vectors $X,Y,\hat{X},\hat{Y}$ such that $(X,Y)$ and $(\hat{X}, \hat{Y})$ are i.i.d. 
\end{remark}

We emphasize that both Assumption \ref{assump.disp} and Assumption \ref{assump.LL}, when combined with Assumption \ref{assump.regularity.disp}, imply Assumption \ref{assump.fixedpoint}. This is recorded in the following Lemma. The proof is transparent in view of the second order characterizations in Remarks \ref{rmk.dmonotone.secondorder} and \ref{rmk.llmonotone.secondorder}, and so is omitted. 

\begin{lemma} \label{lem.lagrangianLL}
Suppose that $L:\R^d\times\R^d\times\cP_2(\R^d\times\R^d)\to\R$ is Lasry--Lions monotone, of class $C^2$ and that there are constants $C_a>0$ and $C>0$ such that $D_{aa}L(x,a,\mu)\ge C_a I_{d\times d}$ for all $(x,a,\mu)\in\R^d\times\R^d\times\cP_2(\R^d\times\R^d)$, $\max\{\|D_{aa}L\|_{L^\infty}, \|D_{xx}L\|_{L^\infty}, \|D_{ax}L\|_{L^\infty}\}\le C$. Then, there exist $C_{L,a}>0$ and $C_{L,x}\ge 0$ (depending on $C_a$ and $C$) such that \eqref{dispmonotone.L} takes place.
\end{lemma}

\subsection{Main results} 

By a classical solution to \eqref{ME}, we mean a function $U \in C^{1,2,2}([0,T] \times \R^d \times \cP_2(\R^d))$ with $D_{xx} U$, $D_{xm} U$, $D_{ym} U$ and $D_{mm} U$ all uniformly bounded, which satisfies $|U(t,x,m)| \leq C \big(1 + |x|^2 + M_1(m) \big)$, for all $(t,x,m)\in [0,T] \times \R^d \times \cP_2(\R^d)$, and satisfies the equation \eqref{ME} in a pointwise sense. 

The main results of this paper are as follows.
\begin{theorem}[Classical solutions under displacement semi-monotonicity] \label{thm.main.disp}
    Suppose that Assumptions \ref{assump.regularity.disp} and \ref{assump.disp} hold. Then there exists a unique classical solution to \eqref{ME}. 
\end{theorem}

\begin{theorem}[Classical solutions under Lasry--Lions monotonicity] \label{thm.main.LL}
    Suppose that Assumptions \ref{assump.regularity.LL} and \ref{assump.LL} hold.
    Then there exists a unique classical solution to \eqref{ME}. 
\end{theorem}

\begin{remark} \label{rmk.uniqueness}
    Uniqueness of classical solutions in the setting of Theorems \ref{thm.main.disp} and \ref{thm.main.LL} can be reduced (by a verification-type argument) to the uniqueness of mean field equilibria for the corresponding MFGC. Uniqueness of equilibria under monotonicity conditions, in turn, can be proved by standard arguments. Thus the main issue addressed in this paper is the \textit{existence} of classical solutions.
\end{remark}

\section{Analysis of the discrete fixed point equation}\label{sec:3}

This section is about the properties of the maps
\begin{align*}
    \ba^N : (\R^d)^N \times (\R^d)^N \to (\R^d)^N
\end{align*}
defined formally by the fixed point relation \eqref{fixedpointintro}, as well their connection to the limiting object $\Phi$, defined by the fixed point problem \eqref{def.Phi}. We start by recalling some preliminary results which were obtained in \cite[Proposition 1.5]{JacMes}:
\begin{proposition} \label{prop.fixedpoint.prelims}
    Suppose that Assumptions \ref{assump.regularity.disp} and \ref{assump.fixedpoint} hold. Then there is a constant $C>0$ such that
    \begin{itemize}
        \item For large enough $N$, the fixed-point equation \eqref{fixedpointintro} uniquely defines a map $\ba^N : (\R^d)^N \times (\R^d)^N \to (\R^d)^N$, and this map is $C$-Lipschitz, in the sense that 
        \begin{align*}
            \sum_{i = 1}^N |a^{N,i}(\bx,\bp) - a^{N,i}(\ov{\bx},\ov{\bp})|^2 \leq C^2 \sum_{i = 1}^N \Big(|x^i - \ov{x}^i|^2 + |p^i - \ov{p}^i|^2\Big).
        \end{align*}
        \item The fixed-point equation \eqref{def.Phi} uniquely defines a map $\Phi : \cP_2(\R^d \times \R^d) \to \cP_2(\R^d \times \R^d)$, and this map satisfies 
        \begin{align*}
            \bd_2 \big( \Phi(\nu), \Phi(\nu') \big) \leq C \bd_2 \big( \nu,\nu' \big), 
        \end{align*}
        for each $\nu,\nu' \in \cP_2(\R^d \times \R^d)$.
        \item For large enough $N$,
            and all $\bx, \bp \in (\R^d)^N$, we have 
    \begin{align*}
        \bd_2\big( m_{\bx,\ba^N(\bx,\bp)}^N, \Phi(m_{\bx,\bp}^N) \big) \leq \frac{C}{\sqrt{N}} M_2^{1/2}\big(m_{\bx,\bp}^N\big). 
    \end{align*}
    \end{itemize}
\end{proposition}
It was shown in \cite{JacMes} that Proposition \ref{prop.fixedpoint.prelims} is enough to treat the convergence problem in the displacement monotone setting, but it is certainly not enough to obtain a classical solution to the master equation, because it only implies that the Hamiltonian $\hat{H}$ is $\bd_{2}$-Lipschitz continuous in $\mu$. Therefore we need to obtain finer regularity estimates on $\ba^N$.

\subsection{Uniform in $N$ estimates on $\ba^N$}
Given a tuple $\bm{n} = (n_1,\dots,n_M) \in \N^M$, we use $\dist(\bm{n})$ to denote the number of distinct integers appearing in the tuple $(n_1,\dots,n_M)$, e.g. $\dist(1,3,7) = 3$ but $\dist(1,3,3) = 2$. Then for $N \in \N$, and $i,j,k,l \in \{1,\dots,N\}$, we introduce the notation
\begin{align} \label{omegadef}  
    &\omega_{i,j}^N = N^{1 - \dist(i,j)}, \quad \omega_{i,j,k}^N = N^{1 - \dist(i,j,k)}, \quad \omega_{i,j,k,l}^N = N^{1 - \dist(i,j,k,l)}
\end{align}
To be clear, $\omega_{i,j}^N = 1$ if $i = j$ and $1/N$ if $i \neq j$, and the definition of $\omega_{i,j,k}^N$, $\omega_{i,j,k,l}^N$ are similar. When $N$ is clear from context, we will often write $\omega_{i,j} = \omega^N_{i,j}$, $\omega_{i,j,k} = \omega_{i,j,k}^N$, and $\omega_{i,j,k,l} = \omega_{i,j,k,l}^N$

\begin{prop} \label{prop.aderivscaling}
    Let Assumptions \ref{assump.regularity.disp} and \ref{assump.fixedpoint} hold. Then $\ba^N$ is of class $C^4$, and there is a constant $C$ such that for all $N$ large enough, and each $i,j,k,l \in \left\{1,\dots,N\right\}$, we have 
    \begin{itemize}
        \item $\|D_{x^j} a^{N,i} \|_{\infty} + \|D_{p^j} a^{N,i}\|_{\infty} \leq C\omega^N_{i,j}$,
        \\
        \item $\|D_{x^kx^j} a^{N,i}\|_{\infty} + \|D_{x^kp^j} a^{N,i}\|_{\infty} + \|D_{p^kp^j} a^{N,i}\|_{\infty} \leq C \omega^N_{i,j,k} $, 
        \\
        \item $\|D_{x^lx^kx^j} a^{N,i}\|_{\infty} + \|D_{x^lx^kp^j} a^{N,i}\|_{\infty} + \|D_{x^lp^kp^j} a^{N,i}\|_{\infty} + \|D_{p^lp^kp^j} a^{N,i}\|_{\infty} \leq C \omega^N_{i,j,k,l} $.
    \end{itemize}
\end{prop}

In order to prove Proposition \ref{prop.aderivscaling}, we need the following lemma, which will allow us to ``see" the monotonicity condition in Assumption \ref{assump.fixedpoint} at a discrete level.

\begin{lemma}
    \label{lem.discretemonotone}
    Let Assumption \ref{assump.fixedpoint} hold. There is a constant $C > 0$ such that for all $N$ large enough, we have 
    \begin{align} \label{disp.discrete}
        \sum_{i = 1}^N (\xi^i)^\top D_{aa} L\big(x^i,a^i,m_{\bx,\ba}^{N,-i} \big) \xi^i + \frac{1}{N-1} \sum_{i \neq j} (\xi^i)^\top D_{\mu}^a D_a L\big(x^i,a^i,m_{\bx,\ba}^{N,-i}, x^j, a^j \big) \xi^j \geq \frac{1}{C} \sum_{i = 1}^N |\xi^i|^2
    \end{align}
    for each $\bx,\ba, \bxi \in (\R^d)^N$.
\end{lemma}

\begin{proof}
    Applying \eqref{dispmonotone.L} when 
    \begin{align*}
        \cL\big(X,X',\alpha, \alpha'\big) = m_{\bx,\bx, \ba, \ov{\ba}}^N,
    \end{align*}
    for some $\bx, \ba, \ov{\ba} \in (\R^d)^N$, gives
    \begin{align*}
        \sum_{i = 1}^N \Big(D_a L(x^i, a^i, m_{\bx,\ba}^N) - D_a L(x^i, \ov{a}^i, m_{\bx,\ov{\ba}}^N) \Big) \cdot (a^i - \ov{a}^i) \geq C_{L,a} \sum_{i = 1}^N |a^i - \ov{a}^i|^2.
    \end{align*}
    As a consequence of \cite[Lemma 3.4]{JacMes} and using the Lipschitz continuity of $\frac{\delta}{\delta m} D_a L$, we find that for $N$ large enough we have 
     \begin{align*}
        \sum_{i = 1}^N \Big(D_a L(x^i, a^i, m_{\bx,\ba}^{N,-i}) - D_a L(x^i, \ov{a}^i, m_{\bx,\ov{\ba}}^{N,-i}) \Big) \cdot (a^i - \ov{a}^i) \geq \frac{C_{L,a}}{2} \sum_{i = 1}^N |a^i - \ov{a}^i|^2.
    \end{align*}
    Considering  $(\ba,\ov{\ba}):=(\ba + \varepsilon \bxi,\ba)$ \blue{$(\ov{\ba}, \ba)$?}, for $\varepsilon >0$, replacing into the previous inequality, dividing by $\varepsilon^{2}$ and sending $\varepsilon\downarrow 0$, we find the desired inequality, 
    with $\frac{1}{C} = \frac{C_{L,a}}{2}$. 
\end{proof}

In what follows, it will be useful to rephrase the estimate \eqref{disp.discrete} as follows. First, we define the functions 
\begin{align*}
   \bm{\cD}^N, \bm{\cO}^N, \wt{\bm{\cD}}^N, \wt{\bm{\cO}}^N  : (\R^d)^N \times (\R^d)^N \to (\R^{d \times d})^{N \times N} \simeq \R^{dN \times dN}
\end{align*}
by
\begin{align*}
    \bm{\cD}^N = (\cD^N_{i,j})_{i,j = 1,\dots,N},\,\, \bm{\cO}^N  = (\cO^N_{i,j})_{i,j = 1,\dots,N}, \,\, \wt{\bm{\cD}}^N = (\wt{\cD}^N_{i,j})_{i,j = 1,\dots,N}, \,\, \wt{\bm{\cO}}^N = (\wt{\cO}^N_{i,j})_{i,j = 1,\dots,N},
\end{align*}
with
\begin{align*}
    \cD^{N}_{i,j}(\bx,\ba) = D_{aa} L(x^i,a^i,m_{\bx,\ba}^{N,-i}) 1_{i = j}, \quad \cO^N_{i,j}(\bx,\ba) = \frac{1}{N-1} D_{\mu}^a D_a L(x^i,a^i,m_{\bx,\ba}^{N,-i},x^j, a^j) 1_{i \neq j}, 
    \\
    \wt{\cD}^{N}_{i,j}(\bx,\ba) = D_{xa} L(x^i,a^i,m_{\bx,\ba}^{N,-i}) 1_{i = j}, \quad \wt{\cO}^N_{i,j}(\bx,\ba) = \frac{1}{N-1} D_{\mu}^x D_a L(x^i,a^i,m_{\bx,\ba}^{N,-i},x^j, a^j) 1_{i \neq j}
\end{align*}
Finally, we define 
\begin{align*}
    \bm{\cM}^N := \bm{\cD}^N + \bm{\cO}^N, \quad \wt{\bm{\cM}}^N =  \wt{\bm{\cD}}^N + \wt{\bm{\cO}}^N.
\end{align*}
Then we can write \eqref{disp.discrete} as 
\begin{align*}
    \bm{\cM}^N(\bx,\ba) \geq \frac{1}{C} I_{Nd \times Nd}, \quad \forall \,\, \bx,\ba \in (\R^d)^N.
\end{align*}

\begin{proof}[Proof of Proposition \ref{prop.aderivscaling}]
    Our starting point is the equation 
    \begin{align} \label{anirelation}
        a^{N,i}(\bx,\bp) = - D_p H(x^i,p^i,m_{\bx,\ba^N(\bx,\bp)}^{N,-i}).
    \end{align}
    \textbf{Step 1: implicit differentiation and qualitative smoothness.}
    By Proposition \ref{prop.fixedpoint.prelims}, $\ba^N$ is Lipschitz continuous, and so it is differentiable almost everywhere. At any point $(\bx,\bp)$ of differentiability of $\ba^N$, we can implicitly differentiate \eqref{anirelation} to find that 
    \begin{align} \label{implicitdiff}
        D_{p^j} a^{N,i}(\bx,\bp) &= - D_{pp} H(x^i,p^i,m_{\bx,\ba^N(\bx,\bp)}^{N,-i}) 1_{i = j}
        \nonumber \\
        &\qquad - \frac{1}{N-1} \sum_{k \neq i}  D_{\mu}^a D_p H\big(x^i,p^i,m_{\bx,\ba^N(\bx,\bp)}^{N,-i}, x^k, a^{N,k}(\bx,\bp) \big) D_{p^j} a^{N,k}(\bx,\bp).
    \end{align}
    Now, by convex duality arguments and the envelope theorem, we know that 
    \begin{align*}
        D_{\mu} H\big(x,p,\mu,x', a' \big) = - D_{\mu} L\big(x, - D_pH(x,p,\mu), \mu,x', a'), 
    \end{align*}
    and thus 
    \begin{align*}
        D_p D_{\mu}^a H\big(x,p,\mu,x',a' \big) = D_a D_{\mu}^a L\big(x,-D_pH(x,p,\mu),\mu,x',a') D_{pp} H(x, p,\mu).
    \end{align*}
    It follows that 
    \begin{align*}
        D_{\mu}^a D_p H\big(x,p,\mu,x', a' \big) = \Big(D_p D_{\mu}^a H\big(x,p,\mu,a' \big)\Big)^{\top} = D_{pp} H(x,p,\mu) D_{\mu}^a D_a L\big(x,- D_p H(x,p,\mu),\mu,x',a' \big).
    \end{align*}
    Thus \eqref{implicitdiff} rewrites as 
    \begin{align} \label{implicitdiff.equiv}
        D_{p^j} a^{N,i} &= - D_{pp} H(x^i,p^i,m_{\bx,\ba^N}^{N,-i}) 1_{i = j}
        \nonumber \\
        &\qquad - D_{pp} H\big(x^i,p^i,m_{\bx, \ba^N}^{N,-i}\big)\frac{1}{N-1} \sum_{k \neq i} D_{\mu}^a D_a L\big(x^i, a^{N,i},m_{\bx,\ba^N}^{N,-i}, x^k,a^{N,k}\big) D_{p^j} a^{N,k} .
    \end{align}
    Since 
    \begin{align*}
        D_{pp} H\big(x^i,p^i,m_{\bx,\ba^N}^{N,-i}) = \big(D_{aa} L\big(x^i, a^{N,i}, m_{\bx, \ba^N}^{N,-i} \big) \big)^{-1}, 
    \end{align*}
    we can rewrite \eqref{implicitdiff.equiv} as
    \begin{align*}
        D_{\bp} \ba^N + \big({\bm{\cD}}^N(\bx,\ba^N) \big)^{-1} {\bm{\cO}}^N(\bx,\ba^N) D_{\bp} \ba^N = - \big({\bm{\cD}}^N(\bx,\ba^N) \big)^{-1}, 
    \end{align*}
    where $D_{\bp} \ba^N = \big(D_{p^j} a^{N,i}\big)_{i,j = 1,\dots,N} \in (\R^{d \times d})^{N \times N}$.
   We find that the previous equation writes as 
    \begin{align*}
        {\bm{\cD}}^N(\bx,\ba^N) D_{\bp} \ba^N + {\bm{\cO}}^N(\bx,\ba^N) D_{\bp} \ba^N = {\bm{\cM}}^N(\bx,\ba^N) D_{\bp} \ba^N =  - I_{Nd\times Nd}.
    \end{align*}
    By Lemma \ref{lem.discretemonotone} we find that 
    \begin{align} \label{jacobian.p}
        D_{\bp} \ba^N = - \big({\bm{\cM}}^N(\bx,\ba^N) \big)^{-1}.
    \end{align}
    A very similar computation shows that at every point of differentiability of $\ba^N$, we have
    \begin{align} \label{implicitdiff.x}
        D_{x^j} a^{N,i} &= - D_{xp} H(x^i,p^i, m_{\bx,\ba^N}^{N,-i})1_{i = j} - \frac{1}{N-1} D_{\mu}^x D_p H \big(x^i,p^i,m_{\bx,\ba^N}^{N,-i}, x^j,p^j) 1_{i \neq j}
        \nonumber \\
        &\qquad - D_{pp}H(x^i,p^i,m_{\bx,\ba^N}^{N,-i}) \frac{1}{N-1} \sum_{k \neq i} D_{\mu}^a D_a L\big(x^i,\ba^{N,i}, m_{\bx,\ba^N}^{N,-i}, x^k, a^{N,k} \big) D_{x^j} a^{N,k}, 
    \end{align}
    which, using the identity $D_{xp} H(x,p,\mu) = D_{pp} H(x,p,\mu) D_{xa} L(x,-D_pH(x,p,\mu),\mu)$, can be written as 
   \begin{align*}
   \left[{\bm{\cD}}^N(\bx,\ba^N) + {\bm{\cO}}^N(\bx,\ba^N)\right] D_{\bx} \ba^N = {\bm{\cM}}^N(\bx,\ba^N) D_{\bx} \ba^N = - \wt{\bm{\cM}}^N(\bx, \ba^N),
   \end{align*}
   from which we find
    \begin{align} \label{jacobian.x}
        D_{\bx}\ba^N = - \big(\bm{\cM}^N(\bx,\ba^N) \big)^{-1} \wt{\bm{\cM}}^N(\bx, \ba^N), 
    \end{align}
    $D_{\bx} \ba^N = (D_{x^j} a^{N,i})_{i,j = 1,\dots,N}$ is the Jacobian matrix in $\bx$ of the function $\ba^N$. Combining \eqref{jacobian.p} and \eqref{jacobian.x} with Lemma \ref{lem.discretemonotone} and Proposition \ref{prop.fixedpoint.prelims}, we see that the weak derivative of the Lipschitz continuous function $\ba^N$ has a continuous (in fact, Lipschitz continuous) representative, and so $\ba^N$ is $C^1$, and the formulas \eqref{implicitdiff.equiv}, \eqref{implicitdiff.x}, \eqref{jacobian.p}, \eqref{jacobian.x} hold everywhere on $(\R^d)^N \times (\R^d)^N$. Once we have this, we easily bootstrap to infer that $\ba^N$ is $C^4$.
    \newline \newline 
    \textbf{Step 2: bounds on the first derivatives.} First, note that we already have the bound $|D_{p^i} a^{N,i}| \leq C$ for large enough $N$ thanks to Proposition \ref{prop.fixedpoint.prelims}. We now fix $j$, and use \eqref{implicitdiff.equiv} to find (since we consider also $i\neq j$) that 
    \begin{align*}
        \sum_{\substack{i = 1,\dots,N \\ i \neq j}} &\tr\Big( \big(D_{p^j} a^{N,i}\big)^\top D_{aa} L(x^i,a^i,m_{\bx,\ba^N}^{N,-i}) D_{p^j} a^{N,i} \Big) 
        \\
        & = - \frac{1}{N-1} \sum_{\substack{i = 1,\dots,N \\ i \neq j}} \sum_{\substack{k = 1,\dots,N \\ k \neq i}} \tr\Big( \big(D_{p^j} a^{N,i}\big)^\top D_{\mu}^a D_a L \big(x^i,a^{N,i}, m_{\bx,\ba^N}^{N,-i}, x^k, a^{N,k} \big) D_{p^j} a^{N,k} \Big)
        \\
        &= - \frac{1}{N-1} \sum_{\substack{i = 1,\dots,N \\ i \neq j}} \sum_{\substack{k = 1,\dots,N \\ k \neq j}} 1_{i \neq k} \tr\Big( \big(D_{p^j} a^{N,i}\big)^\top D_{\mu}^a D_a L \big(x^i,a^{N,i}, m_{\bx,\ba^N}^{N,-i}, a^{N,k} \big) D_{p^j} a^{N,k} \Big)
        \\
        &\quad - \frac{1}{N-1} \sum_{\substack{i = 1,\dots,N \\ i \neq j}} \tr\Big( \big(D_{p^j} a^{N,i}\big)^\top D_{\mu}^a D_a L \big(x^i,a^{N,i}, m_{\bx,\ba^N}^{N,-i}, a^{N,j} \big) D_{p^j} a^{N,j} \Big), 
    \end{align*}
    and so using Lemma \ref{lem.discretemonotone} as well as the uniform bounds on $D_{\mu}^a D_a L$ and $D_{p^j} a^{N,j}$, we find that for each fixed $j$,
    \begin{align*}
        \sum_{\substack{i = 1,\dots,N \\ i \neq j}} |D_{p^j} a^{N,i}|^2 &\leq C \bigg\{  \sum_{\substack{i = 1,\dots,N \\ i \neq j}} \tr\Big( \big(D_{p^j} a^{N,i}\big)^\top D_{aa} L(x^i,a^i,m_{\bx,\ba^N}^{N,-i}) D_{p^j} a^{N,i} \Big)
        \\
        &\qquad + \frac{1}{N-1} \sum_{\substack{i = 1,\dots,N \\ i \neq j}} \sum_{\substack{k = 1,\dots,N \\ k \neq j}} 1_{i \neq k} \tr\left( \big(D_{p^j} a^{N,i}\big)^\top D_{\mu}^a D_a L \big(x^i,a^{N,i}, m_{\bx,\ba^N}^{N,-i}, a^{N,k} \big) D_{p^j} a^{N,k} \right)\bigg\}
        \\
        &\leq \frac{C}{N} \sum_{\substack{i = 1,\dots,N \\ i \neq j}} |D_{p^j} a^{N,i}| \leq \frac{C}{\sqrt{N}} \left(\sum_{i \neq j} |D_{p^j} a^{N,i}|^2 \right)^{1/2}, 
    \end{align*}
    from which we deduce the bound 
    \begin{align*}
        \sum_{\substack{i = 1,\dots,N \\ i \neq j}} |D_{p^j} a^{N,i}|^2 \leq \frac{C}{N}.
    \end{align*}
    Now we again return to \eqref{implicitdiff}, to find that for $i \neq j$, 
    \begin{align*}
        |D_{p^j} a^{N,i}| \leq \frac{C}{N}\left(1 + \sum_{k \neq j} |D_{p^j} a^{N,k}| \right) \leq \frac{C}{N}.
    \end{align*}
    Thus we have $|D_{p^j} a^{N,i}| \leq C\omega_{i,j}^N$, as desired. The bound on $D_{x^j} a^{N,i}$ is very similar, and is omitted.
    \newline \newline 
    \textbf{Step 3: bounds on higher derivatives.} We next aim to prove the bound 
    \begin{align*}
        |D_{p^kp^j} a^{N,i}| \leq C \omega^N_{i,j,k}.
    \end{align*}
    We are going to further differentiate \eqref{implicitdiff} and \eqref{implicitdiff.equiv}. To avoid notational difficulties when dealing with derivatives like $D_{p^kp^j} a^{N,i}$ (which come with three indices), we argue as if $d = 1$, but it is easy to check that the same argument works for any $d \in \N$. We differentiate again to find that
    \begin{align} \label{implicitdiff.equiv2}
        D_{p^k p^j} a^{N,i} &= - D_{pp} H\big(x^i,p^i,m_{\bx, \ba^N}^{N,-i}\big)\frac{1}{N-1} \sum_{l \neq i} D_{\mu}^a D_a L\big(x^i, a^{N,i},m_{\bx,\ba^N}^{N,-i}, x^l, a^{N,l}\big) D_{p^k p^j} a^{N,l} 
        \nonumber \\
        &\qquad \qquad  + T^{i,j,k}, 
    \end{align}
    where 
    \begin{align*}
        T^{i,j,k} &= - D_{ppp} H(x^i,p^i,m_{\bx,\ba^N}^{-i}) 1_{i = j = k} - \frac{1}{N-1} \sum_{l \neq i} D_{\mu}^a D_{pp} H(x^i,p^i,m_{\bx,\ba^N}^{-i},x^l, a^{N,l}) D_{p^k} a^l 1_{i = j}
        \\
        &\quad - \frac{1}{N-1} \sum_{l \neq i} D_{\mu}^a {D_{pp}} H(x^i,p^i,m_{\bx,\ba}^{N,-i},x^l, a^l) D_{p^j} a^l 1_{i = k}
        \\
        &\quad - \frac{1}{N-1} \sum_{l \neq i} D_{a'} D_{\mu}^a D_p H (x^i,p^i,m_{\bx,\ba^N}^{N,-i},x^l, a^l) D_{p^j} a^{N,l} D_{p^k} a^{N,l}
        \\
        &\quad - \frac{1}{(N-1)^2} \sum_{l \neq i} \sum_{q \neq i} D_{\mu}^a D_{\mu}^a D_p H(x^i,p^i,m_{\bx,\ba^N}^{N,-i},x^l, a^{N,l}, a^{N,q}) D_{p^j} a^{N,l} D_{p^k} a^{N,q}.
    \end{align*}

    Using the bounds from \textbf{Step 2} and the boundedness of the third derivatives of $H$, one can check that
    \begin{align*}
        |T^{i,j,k}| \leq C \omega^N_{i,j,k}.
    \end{align*}
    Using \eqref{implicitdiff.equiv2}, we find that for each fixed $j,k \in \{1,...,N\}$,
    \begin{align*}
        \sum_{i = 1}^N &\left(D_{p^k p^j} a^{N,i}\right)^{\top} D_{aa} L(x^i,a^i,m_{\bx,\ba^N}^{N,-i}) D_{p^k p^j} a^{N,i}
        \\
        &\leq  - \frac{1}{N-1} \sum_{i=1}^N \sum_{l \neq i} \left(D_{p^k p^j} a^{N,i}\right)^{\top} D_{\mu}^a D_a L \big(x^i,a^{N,i}, m_{\bx,\ba^N}^{N,-i}, x^l, a^{N,l} \big) D_{p^k p^j} a^{N,l}
        \\
        &\qquad + C \sum_{i = 1}^N |D_{p^k p^j} a^{N,i}| \omega_{i,j,k}^N, 
    \end{align*}
    and so using Lemma \ref{lem.discretemonotone} again, we deduce that 
    \begin{align*}
        \sum_{i = 1}^N |D_{p^k p^j} a^{N,i}|^2 \leq C \sum_{i = 1}^N |D_{p^k p^j} a^{N,i}| \omega_{i,j,k}^N \leq C \big(\sum_{i = 1}^N |D_{p^k p^j} a^{N,i}|^2 \big)^{1/2} \big(\sum_{i = 1}^N |\omega^N_{i,j,k}|^2 \big)^{1/2} 
        \\
        \leq C \omega_{k,j}^N \big(\sum_{i = 1}^N |D_{p^k p^j} a^{N,i}|^2 \big)^{1/2}, 
    \end{align*}
    where we have used the fact that $\big(\sum_{i = 1}^N |\omega^N_{i,j,k}|^2 \big)^{1/2}\le C \omega_{k,j}^N$.
    We thus obtain the bound 
    \begin{align} \label{repeatedindexbound}
       \sum_{i = 1}^N |D_{p^k p^j} a^{N,i}|^2 \leq C|\omega^N_{j,k}|^2, 
    \end{align}
   for each fixed $j,k \in \{1,...,N\}$, and in particular 
    \begin{align*}
        |D_{p^k p^j} a^{N,j}| \leq C \omega^N_{j,k} = C \omega^N_{j,j,k}, \quad |D_{p^k p^j} a^{N,k}| \leq C \omega^N_{j,k} = C\omega^N_{k,j,k}.
    \end{align*}
    Now we fix $k \neq j$, and use \eqref{implicitdiff.equiv2} to estimate
    \begin{align*}
       \sum_{i \neq j,k} & \left(D_{p^k p^j} a^{N,i}\right)^{\top} D_{aa} L(x^i,a^i,m_{\bx,\ba^N}^{N,-i}) D_{p^k p^j} a^{N,i} 
       \\
       &\leq  - \frac{1}{N-1} \sum_{i \neq j,k} \sum_{l \neq i} \big(D_{p^k p^j} a^{N,i}\big)^\top D_{\mu}^a D_a L \big(x^i,a^{N,i}, m_{\bx,\ba^N}^{N,-i}, x^l, a^{N,l} \big) D_{p^k p^j} a^{N,l} 
      \\ & \qquad + C\sum_{i \neq j,k} |D_{p^k p^j} a^{N,i}| \omega_{i,j,k}^N
      \\
      &= - \frac{1}{N-1} \sum_{i \neq j,k} \sum_{l \neq j,k} 1_{i \neq l} \left(D_{p^k p^j} a^{N,i}\right)^{\top} D_{\mu}^a D_a L \big(x^i,a^{N,i}, m_{\bx,\ba^N}^{N,-i}, x^l, a^{N,l} \big) D_{p^k p^j} a^{N,l} 
      \\ & \quad - \frac{1}{N-1} \sum_{i \neq j,k} \left(D_{p^k p^j} a^{N,i}\right)^{\top} D_{\mu}^a D_a L \big(x^i,a^{N,i}, m_{\bx,\ba^N}^{N,-i}, x^j, a^{N,j} \big) D_{p^k p^j} a^{N,j} 
      \\
      &\quad - \frac{1}{N-1} \sum_{i \neq j,k}\left(D_{p^k p^j} a^{N,i}\right)^{\top} D_{\mu}^a D_a L \big(x^i,a^{N,i}, m_{\bx,\ba^N}^{N,-i}, x^k, a^{N,k} \big) D_{p^k p^j} a^{N,k}  + C\sum_{i \neq j,k} |D_{p^k p^j} a^{N,i}| \omega_{i,j,k}^N,
    \end{align*}
    and so again using Lemma \ref{lem.discretemonotone}, and then using \eqref{repeatedindexbound}, we find 
    \begin{align*}
        &\sum_{i \neq j,k} |D_{p^k p^j} a^{N,i}|^2 \leq \frac{C}{N} \sum_{i \neq j,k} |D_{p^k p^j} a^{N,i}| |D_{p^k p^j} a^{N,j}|
        \\
        &\qquad + \frac{C}{N} \sum_{i \neq j,k} |D_{p^k p^j} a^{N,i}||D_{p^k p^j} a^{N,k}| + C\sum_{i \neq j,k} |D_{p^k p^j} a^{N,i}| \omega^N_{i,j,k}
        \\
        &\leq \frac{C}{N} \sum_{i \neq j,k} |D_{p^k p^j} a^{N,i}| \omega^N_{j,k} + C \sum_{i \neq j,k} |D_{p^k p^j} a^{N,i}| \omega_{i,j,k}^N
        \\
        &\leq C\big(\sum_{i \neq j,k} |D_{p^k p^j} a^{N,i}|^2 \big)^{1/2} \Big(\frac{\omega_{j,k}^N}{\sqrt{N}} + \big(\sum_{i \neq j,k} |\omega_{i,j,k}|^2\big)^{1/2} \Big), 
    \end{align*}
    from which we deduce that
    \begin{align*}
        \sum_{i \neq j,k} |D_{p^k p^j} a^{N,i}|^2 \leq C \Big(\frac{|\omega_{j,k}^N|^2}{N} + \sum_{i \neq j,k} |\omega_{i,j,k}|^2 \Big).
    \end{align*}
    Now we return to \eqref{implicitdiff.equiv2}, and apply the Cauchy--Schwarz inequality to find that  
    \begin{align*}
        |D_{p^kp^j} a^{N,i}| &\leq C \Big(\omega_{i,j,k}^N + \frac{1}{N} |D_{p^k p^j} a^{N,j}| + \frac{1}{N} |D_{p^k p^j} a^{N,k}| + \frac{1}{\sqrt{N}} \big( \sum_{l \neq j,k} |D_{p^kp^j} a^{N,l}|^2 \big)^{1/2} \Big)
        \\
        &\leq C \omega_{i,j,k}^N + \frac{C}{N} \omega_{j,k}^N + \frac{C}{\sqrt{N}} \Big( \sum_{l \neq j,k} |\omega_{l,j,k}^N|^2 \Big)^{1/2} \leq C \omega_{i,j,k}^N,
    \end{align*}
    as desired. 
    
    In a very similar manner, we obtain the same bounds on the other second derivatives $D_{x^kx^j} a^{N,i}$ and $D_{x^kp^j} a^{N,i}$.

    For the bounds on $D_{p^l p^k p^j} a^{N,i}$, another tedious but straightforward computation (using the bounds already obtained on $D_{p^k p^j} a^{N,i}$) shows that 
    \begin{align} \label{Dplpkpj}
        D_{p^l p^k p^j} a^{N,i} &= - D_{pp} H\big(x^i,p^i,m_{\bx, \ba^N}^{N,-i}\big)\frac{1}{N-1} \sum_{q \neq i} D_{\mu}^a D_a L\big(x^i, a^{N,i},m_{\bx,\ba^N}^{N,-i}, x^q, a^{N,q}\big) D_{p^l p^k p^j} a^{N,q} 
        \nonumber \nonumber  \\
        &\qquad \qquad  + T^{i,j,k,l}, 
    \end{align}
    where $T^{i,j,k,l}$ is a function satisfying $|T^{i,j,k,l}| \leq C \omega_{i,j,k,l}^N$. Again, we deduce from Lemma \ref{lem.discretemonotone} that for any fixed $k,j,l \in \{1,\dots,N\}$, 
    \begin{align*}
        \sum_{i = 1}^N &|D_{p^l p^k p^j} a^{N,i}|^2 \leq C \sum_{i = 1}^N |D_{p^l {p^k} p^j} a^{N,i}| \omega_{i,j,k,l}^N \leq C \Big( \sum_{i = 1}^N |D_{p^l p^k p^j} a^{N,i}|^2 \Big)^{1/2} \Big(\sum_{i = 1}^N |\omega_{i,j,k,l}^N|^2 \Big)^{1/2}
        \\
        &\leq C \omega_{k,j,l}^N \Big( \sum_{i = 1}^N |D_{p^l p^k p^j} a^{N,i}|^2 \Big)^{1/2}
    \end{align*}
    from which we deduce the bound 
    \begin{align*}
        \sum_{i = 1}^N |D_{p^l p^k p^j} a^{N,i}|^2 \leq C |\omega_{k,j,l}^N|^2.
    \end{align*}
    In particular, this gives the desired $|D_{p^lp^kp^j} a^{N,i}| \leq C \omega_{i,j,k,l}^N$ whenever $i \in \{j,k,l\}$. To handle the case $i \notin \{l,k,j\}$, we rewrite \eqref{Dplpkpj} as
    \begin{align} \label{ttilde.est}
        D_{p^l p^k p^j} a^{N,i} &= - D_{pp} H(x^i, p^i, m_{\bx, \ba^N}^{N,-i}) \frac{1}{N-1} \sum_{q \neq l,k,j} 1_{q \neq i} D_{\mu}^a D_a L\big(x^i, a^{N,i}, m_{\bx, \ba^N}^{N,-i}, x^l, a^{N,l} \big) D_{p^l p^k p^j} a^{N,q} 
        \nonumber \\
        &\qquad + \wt{T}^{i,j,k,l}, 
    \end{align}
    where 
    \begin{align*}
      \wt{T}^{i,j,k,l} &= 
         - D_{pp} H(x^i,p^i,m_{\bx, \ba^N}^{N,-i}) \frac{1}{N-1}\bigg( D_{\mu}^a D_a L\big(x^i, a^{N,i}, m_{\bx, \ba^N}^{N,-i}, x^j, a^{N,j}\big) D_{p^lp^kp^j} a^{N,j}
        \\
        &\quad + D_{\mu}^a D_a L\big(x^i, a^{N,i}, m_{\bx, \ba^N}^{N,-i}, x^k, a^{N,k}\big) D_{p^lp^kp^j} a^{N,k}
      + D_{\mu}^a D_a L\big(x^i, a^{N,i}, m_{\bx, \ba^N}^{N,-i}, x^l, a^{N,l}\big) D_{p^lp^kp^j} a^{N,l} \bigg)
      \\
      &\qquad + T^{i,j,k,l}. 
    \end{align*}
    Using the bound on $D_{p^lp^kp^j} a^{N,q}$ for $q \in \{j,k,l\}$, we find that $|\wt{T}^{i,j,k,l}| \leq C \omega_{i,j,k,l}^N$. Then again using Lemma \ref{lem.discretemonotone} as above, we find that 
    \begin{align*}
        \sum_{i \neq j,k,l} |D_{p^lp^kp^j} a^{N,i}|^2 \leq C \sum_{i \neq j,k,l} |D_{p^lp^kp^j} a^{N,i}| |\wt{T}^{i,j,k,l}|, 
    \end{align*}
    which yields 
    \begin{align*}
        \sum_{i \neq j,k,l} |D_{p^lp^k p^j} a^{N,i}|^2 \leq C \sum_{i \neq j,k,l} |\omega_{i,j,k,l}^N|^2 \leq \frac{C}{N} |\omega_{j,k,l}|^2. 
    \end{align*}
    Coming back to \eqref{ttilde.est}, we estimate 
    \begin{align*}
        |D_{p^l p^k p^j} a^{N,i}| &\leq \frac{C}{N} \sum_{q \neq l,k,j} |D_{p^lp^kp^j} a^{N,q}| + C\omega_{i,j,k,l}^N
        \\
        &\leq \frac{C}{\sqrt{N}} \Big( \sum_{q \neq l,k,j} |D_{p^lp^kp^j} a^{N,q}|^2 \Big)^{1/2} + C\omega_{i,j,k,l}^N 
        \\
        &\leq \frac{C}{N} \omega_{j,k,l}^N + C\omega_{i,j,k,l}^N \leq C \omega_{i,j,k,l}^N, 
    \end{align*}
    as desired. 

    The bounds on 
    the other third derivatives $D_{x^lx^kx^j}a^{N,i}$, $D_{x^lx^kp^j}$, and $D_{x^lp^kp^j} a^{N,i}$ are very similar, and so we omit the proof. 
    
\end{proof}

\section{Uniform in $N$ estimates under displacement semi-monotonicity}\label{sec.uniform}
We find it useful in this section to introduce notation 
\begin{align*}
    \hat{H}^{N,i} : (\R^d)^N \times (\R^d)^N \to \R, \quad \hat{H}^{N,i}(\bx,\bp) = H\big(x^i,p^i,m_{\bx,\ba^N(\bx,\bp)}^{N,-i} \big),
\end{align*}
and to rewrite the Nash system as
\begin{align} \label{nashsystem2}
    \begin{cases}
       \ds  - \partial_t u^{N,i} - \sum_{j = 1}^N \Delta_{j} u^{N,i} - \sigma_0 \sum_{j,k = 1}^N \tr\big(D_{jk} u^{N,i} \big) + \hat{H}^{N,i}(\bx,\diag \bu^N)
       \\
       \ds \qquad - \sum_{j \neq i} a^{N,j}(\bx, \diag \bu^N) \cdot D_j u^{N,i} = 0, 
      \quad  (t,\bx) \in [0,T] \times (\R^d)^N, 
       \\
       \ds u^{N,i}(T,\bx) = G(x^i,m_{\bx}^{N,-i}), \quad \bx \in (\R^d)^N.
    \end{cases}
\end{align}

\begin{definition} \label{defn.admissible}
    We will say that $\bu^N = (u^{N,1},\dots,u^{N,N})$ is an admissible solution to \eqref{nashsystem2} if it is a classical solution and for each $i,j,k = 1,\dots,N$, we have 
    \begin{align*}
        u^{N,i} \in C^{1,2}_{\text{loc}}\big([0,T] \times (\R^d)^N ; \R \big), \quad D_j u^{N,i} \in C^{1,2}_{\text{loc}}\big([0,T] \times (\R^d)^N ; \R^d \big), \quad D_{kj} u^{N,i} \in C^{1,2}_{\text{loc}}\big([0,T] \times (\R^d)^N ; \R^{d \times d} \big), 
    \end{align*}
   and, furthermore, the spatial derivatives of $u^{N,i}$ of order $2$, $3$, and $4$ are bounded, and we have the growth condition
    \begin{align*}
        |u^{N,i}(t,\bx)| \leq C\big(1 + |\bx|^2\big), 
    \end{align*}
    for some $C > 0$ and all $i = 1,\dots,N$, $(t,\bx) \in [0,T] \times (\R^d)^N$. We say that an admissible solution $(u^{N,i})_{i = 1,\dots,N}$ is symmetric if 
     \begin{align} \label{symmetry.1}
        u^{N,i}(t,\bx) = u^{N,1}\big(t, (x^i, \bx^{-i}) \big), \quad i = 1,\dots,N, \quad \bx \in (\R^d)^N, 
    \end{align} and in addition
    \begin{align} \label{symmetry.2}
        u^{N,1}(t,\bx) = u^{N,1}(t,x^1, x^{\sigma(2)},\dots,x^{\sigma(N)}).
    \end{align}
    Equivalently, $\bu^N$ is symmetric if there is a map $\Psi : [0,T] \times \cP_2(\R^d) \to \R$ such that
    \begin{align} \label{symmetry}
        u^{N,i}(t,\bx) = \Psi(t,x^i,m_{\bx}^{N,-i})
    \end{align}
    for each $i = 1,\dots,N$ and $(t,\bx) \in [0,T] \times (\R^d)^N$.
\end{definition}

The goal of this section is to prove the following uniform in $N$ a-priori estimates on admissible solutions to the Nash system.
\begin{theorem} \label{thm.uniform.nash}
    Suppose that Assumptions \ref{assump.disp} and \ref{assump.regularity.disp} hold. Then there is a constant $C_0>0$ such that for all $N$ large enough and all $i,j,k,l = 1,\dots,N$, any admissible solution $\bu^N$ to the Nash system \eqref{nashsystem2} satisfies
    \begin{align} \label{nash.deriv.scaling}
        \norm{D_{j} u^{N,i}}_{\infty} \leq C_0\left(\frac{1}{N} + \big(1 + |x^i| + M_1(m_{\bx}^N)\big)1_{i = j}\right), \quad \norm{D_{kj} u^{N,i}}_{\infty} \leq C_0\omega_{i,j,k}^N, \quad \norm{D_{lkj} u^{N,i}}_{\infty} \leq C_0 \omega_{i,j,k,l}^N.
    \end{align}
    In addition, we have growth estimate
    \begin{align} \label{nash.quadgrowth}
        |u^{N,i}(t,\bx)| \leq C_0 \Big(1 + |x^i|^2 + M_1\big(m_{\bx}^{N} \big) \Big), 
    \end{align}
    the time regularity estimates
    \begin{align} \label{nash.timereg}
        &|\partial_t u^{N,i}(t,\bx)| \leq C\big(1 + |x^i|^2 + M^2_1(m_{\bx}^N)\big), \quad  |\partial_t D_j u^{N,i}(t,\bx)| \leq C\omega_{i,j}^N \big(1 + |x^i| + |x^j| + M_1(m_{\bx}^N) \big), 
     \end{align}
     as well as 
     \begin{align}
         |D_{kj} u^{N,i}(t,\bx) - D_{kj}u^{N,i}(s,\bx)| \leq C\big(1 + |x^i| + |x^j| + |x^k| + M_1(m_{\bx}^N) \big) \omega^N_{i,j,k} |t-s|^{1/2}, 
    \end{align}
    and finally
    \begin{align}
        |\partial_t u^{N,i}(t,\bx) - \partial_t u^{N,i}(s,\bx)| &\leq C\big(1 + |x^i|^2 + M_1^2(m_{\bx}^N) \big)|t-s|^{1/2}
    \end{align}
\end{theorem}

\subsection{Preliminaries}

First, we emphasize that we are going to fix throughout this section an admissible solution $\bu^N= (u^{N,i})_{i = 1,\dots,N}$ to the Nash system \eqref{nashsystem2}.
The estimates appearing in Theorem \ref{thm.uniform.nash} will be obtained by repeated differentiation of the Nash system. As such, we introduce the notation
\begin{align*}
    u^{i,j} = D_j u^{N,i}, \quad u^{i,j,k} = D_{kj} u^{N,i}, \quad u^{i,j,k,l} = D_{lkj} u^{N,i}.
\end{align*}
To avoid a proliferation of indices and notational burden, we will argue as if $d = 1$, so that $u^{i,j}, u^{i,j,k}, u^{i,j,k,l}$ are scalar-valued functions, but the reader can check that the same arguments go through in higher dimensions (roughly speaking, when $d \geq 2$, everywhere we write $u^{i,j} = D_j u^i$ below, one can replace this with $u^{i,j_q} = D_{j_q} u^i$ for $q \in \{1,\dots,d\}$, and perform the same computations, eventually taking a maximum over $q \in \{1,\dots,d\}$). 

Throughout the remainder of this section, for any $(t_0,\bx_0) \in [0,T) \times (\R^d)^N$, we denote by $\bX^{t_0,\bx_0}$ the state process associated to the closed-loop Nash equlibria started from $(t_0,\bx_0)$, i.e. the unique solution of the SDE 
\begin{align}\label{proc:closed-loop}
    dX_t^{t_0,\bx_0,i} = a^{N,i}(t,\bX_t^{t_0,\bx_0}) dt + \sqrt{2} dW_t^i + \sqrt{2\sigma_0} dW_t^0, \,\, t_0 \leq t \leq T, \quad X_{t_0}^{t_0,\bx_0,i} = x_0^i.
\end{align}
Given a function $\phi : [0,T] \times (\R^d)^N \to \R^q$, for some $q \in \N$ and $0 \leq s \leq t \leq T$, we find it useful to introduce the notation
\begin{align}\label{inf:norms}
    \| \phi \|_{\infty,t} := \| \phi(t,\cdot)\|_{L^{\infty}((\R^d)^N)} , \quad \|\phi\|_{\infty,s,t} := \sup_{r \in [s,t]} \|\phi\|_{\infty,r}.
\end{align}
We also introduce the notation $\scrL^N$ for the generator of the processes $\bX^{t_0,\bx_0}$, i.e. the differential operator which acts on $C^2$ functions $f : [0,T] \times (\R^d)^N \to \R$ via 
\begin{align}\label{def:L^N}
    \scrL^N f := \sum_{j = 1}^N \Delta_j f + \sum_{j,k = 1}^N \tr\big(D_{jk} f \big) + \sum_{j = 1}^N a^{N,j} (\bx, \diag \bu^N) \cdot D_j f.
\end{align}
Finally, we introduce some notation which indicates integrability against (the law of) $\bX^{t_0,\bx_0}$, uniformly in $(t_0,\bx_0)$. In particular, given $\phi : [0,T] \times (\R^d)^N \to \R^q$, we adopt the notation
\begin{align*}
    \norm{\phi}_{L^1(\scrL^N)} := \sup_{t_0\in[0,T],\bx_0\in(\R^d)^N} \E\bigg[ \int_{t_0}^T |\phi(t,\bX_t^{t_0,\bx_0})| dt \bigg]
\end{align*}
and
\begin{align*}    
     \norm{\phi}_{L^2(\scrL^N)} := \left(\sup_{t_0\in[0,T],\bx_0\in(\R^d)^N} \E\bigg[ \int_{t_0}^T |\phi(t,\bX_t^{t_0,\bx_0})|^2 dt \bigg]\right)^{1/2}.
\end{align*}
We remark here right away that $\norm{\phi}_{L^1(\scrL^N)}\le \sqrt{T} \norm{\phi}_{L^2(\scrL^N)}.$

Using this notation, the following estimate was obtained in \cite[Proposition 6.2, Proposition 6.5]{JacMes}.
\begin{prop} \label{prop.nash.preliminaries}
    Let Assumptions \ref{assump.regularity.disp} and \ref{assump.disp} hold. Then there is a constant $C>0$ such that for all $N$ large enough, we have
    \begin{align*}
        \norm{\sum_{j \neq i} |D_j u^{N,i}|^2}_{\infty} + \norm{ \Big(\sum_{j \neq i} \sum_{k=1}^{N} |D_{kj}u^{N,i}|^2\Big)^{1/2}}^2_{L^2(\scrL^{N})} \leq C/N, 
\end{align*}
for each $i = 1,\dots,N$, as well as 
\begin{align*}
        \norm{ \big | (D_{ji}u^{N,i})_{i,j = 1,\dots,N} \big|_{\op} }_{L^2(\scrL^N)} \leq C.
    \end{align*}
\end{prop}

\subsection{Estimates on the first derivatives}
We first differentiate \eqref{nashsystem2} with respect to $x_j$, keeping in mind the identity
\begin{align*}
    D_{p^j} \hat{H}^{N,i} &= 1_{i = j} D_p H\big(x^i,p^i,m_{\ba^N(\bx,\bp)}^{N,-i} \big) + \frac{1}{N-1} \sum_{k \neq i} D_{\mu}^a H \big(x^i,p^i,m_{\bx,\ba^N}^{N,-i}, x^k, a^{N,k} \big) D_{p^j} a^k
    \\
    &= - 1_{i = j} a^{N,i}(\bx,\bp) + \frac{1}{N-1} \sum_{k \neq i} D_{\mu}^a H \big(x^i,p^i,m_{\bx,\ba^N}^{N,-i}, x^k, a^{N,k} \big) D_{p^j} a^k
\end{align*}
to find that $u^{i,j}:=D_{x_j}u^{N,i}$ satisfies 
    \begin{align} \label{nashfirstder2}
    \begin{cases} \ds - \partial_t u^{i,j} - \scrL^N u^{i,j}
   - \sum_{k \neq i} \sum_{l = 1}^N D_{jl} u^{l} D_{p^l} a^{N,k} u^{i,k} - \sum_{k \neq i} D_{x^j} a^{N,k} u^{i,k}   \\ \ds
    \quad  + \sum_{k \neq i}  D_{jk} u^{N,k}  D_{p^k} \hat{H}^{N,i} + \frac{1}{N-1} \sum_{k \neq i} D_{\mu}^a H\big(x^i,p^i,m_{\bx, \ba^N}^{N,-i}, x^k, a^{N,k} \big) D_{p^i} a^{N,k} D_{ji} u^{N,i}
    \\
    \ds \quad  
    + D_{x^j} \hat{H}^{N,i} = 0, \quad (t,\bx) \in [0,T] \times (\R^d)^N, \vspace{.2cm} 
    \\ \ds 
    u^{i,j}(T,\bx) = D_{x^j} G^{N,i}(\bx), \quad \bx \in (\R^d)^N, 
    \end{cases}
\end{align}
where we have used 
$$D_{ji} u^{N,i}  D_{p^i} \hat{H}^{N,i} = -  D_{ji} u^{N,i}  a^{N,i} + {D_{ji} u^{N,i}} \frac{1}{N-1} \sum_{k \neq i} D_{\mu}^a H \big(x^i,p^i,m_{\bx,\ba^N}^{N,-i}, x^k, a^{N,k} \big) D_{p^i} a^k,$$
and the term $-  D_{ji} u^{N,i}  a^{N,i}$ has been incorporated into $\scrL^N u^{i,j}$.

The previous system we rewrite as 
 \begin{align} \label{nashfirstder2}
    \begin{cases} \ds - \partial_t u^{i,j} - \scrL^N u^{i,j}
   - \sum_{k \neq i} \sum_{l = 1}^N D_{jl} u^{l} D_{p^l} a^{N,k}(\bx, \diag \bu^N) u^{i,k} - \sum_{k \neq i} D_{x^j} a^{N,k}(\bx, \diag \bu^N) u^{i,k}   \\ \ds
    \quad  + \sum_{k=1}^N  D_{jk}u^{N,k}  \hat{H}^{N,i,k}(\bx, \diag \bu^N)
    + D_{x^j} \hat{H}^{N,i}(\bx, \diag \bu^N) = 0, \quad (t,\bx) \in [0,T] \times (\R^d)^N, \vspace{.2cm} 
    \\ \ds 
    u^{i,j}(T,\bx) = D_{x^j} G^{N,i}(\bx), \quad \bx \in (\R^d)^N, 
    \end{cases}
\end{align}
where we have set $G^{N,i}(\bx) := G(x^i,m_{\bx}^{N,-i})$ and 
\begin{align}\label{def:hatHik}
   \hat{H}^{N,i,k}(\bx,\bp) := \frac{1}{N-1} \sum_{l \neq i} D_{\mu}^a H \big(x^i,p^i,m_{\bx,\ba^N}^{N,-i}, x^l, a^{N,l} \big) D_{p^k} a^{N,l}.
\end{align}

It will also be useful to record the fact that the maps $\hat{H}^{N,i}$ and $\hat{H}^{N,i,j}$ satisfy bounds very similar to $a^{N,i}$:
\begin{lemma} \label{lem.hatH} Suppose that Assumptions \ref{assump.regularity.disp} and \ref{assump.fixedpoint} hold.
    Then there is a constant $C>0$ such that for all large enough $N$, and each $i,j,k = 1,\dots,N$, we have 
    \begin{align*}
        &\|D_{x^j} \hat{H}^{N,i}\|_{\infty} + \|D_{p^j} \hat{H}^{N,i}\|_{\infty} \leq C \left(\frac{1}{N} + 1_{i = j} \big(1 + |x^i| + |p^i| + M_1(m_{\bx}^N) + M_1(m_{\bm{p}}^N) \big) \right), 
         \vspace{.2cm} \\
        &\|D_{x^kx^j} \hat{H}^{N,i}\|_{\infty} + \|D_{x^kp^j} \hat{H}^{N,i}\|_{\infty} + \|D_{p^kp^j} \hat{H}^{N,i} \|_{\infty} \leq C \omega_{i,j,k}^N, 
        \vspace{.2cm} \\
        &\|D_{x^lx^kx^j} \hat{H}^{N,i}\|_{\infty} + \|D_{x^lx^kp^j} \hat{H}^{N,i}\|_{\infty} +  \|D_{x^lp^kp^j} \hat{H}^{N,i}\|_{\infty} + \|D_{p^l p^kp^j} \hat{H}^{N,i} \|_{\infty} \leq C \omega_{i,j,k,l}^N
        \\
        &\|\hat{H}^{N,i,j}\|_{\infty} \leq C/N,
        \\
        &\|D_{x^k} \hat{H}^{N,i,j}\|_{\infty} + \|D_{p^k} \hat{H}^{N,i,j}\|_{\infty} \leq \frac{C}{N} \big(\omega_{i,k}^N + \omega^N_{j,k} \big) \leq C \begin{cases}
            \omega_{i,j,k}^N & i \neq j, 
            \\
            \frac{1}{N} \omega_{i,k}^N & i = j,
        \end{cases} 
        \\
        & \|D_{x^lx^k} \hat{H}^{N,i,j}\|_{\infty} + \|D_{x^l p^k} \hat{H}^{N,i,j}\|_{\infty} + \|D_{p^lp^k} \hat{H}^{N,i,j} \|_{\infty} \leq C\begin{cases}
            \omega_{i,j,k,l}^N & i \neq j, 
            \\
            \frac{1}{N} \omega_{i,k,l}^N & i = j.
        \end{cases}
    \end{align*}
\end{lemma}

\begin{proof}
    These bounds are obtained in a straightforward way by combining the assumptions on $H$ with the bounds on the derivatives of $\ba^N$ obtained in Proposition \ref{prop.aderivscaling}. We omit the details.
\end{proof}

\begin{remark}
We underline the fact that the first terms in Lemma \ref{lem.hatH} have a linear growth in $p^{i}$, whenever $i=j$. However, $p^{i}$ will stand always for the placeholder of terms (such as $D_{i}u^{N,i}$) growing at most linearly in $x^{i}$, and therefore all these compositions will eventually see a growth of order $|x^{i}|.$ 
\end{remark}

We are now ready to prove an estimate on $D_j u^i$. 
\begin{proposition} \label{prop.uij}
    Let Assumptions \ref{assump.regularity.disp} and \ref{assump.disp} hold. Then there is $C>0$ such that for all $N$ large enough, and all $i \neq j$,
    \begin{align*}
        \norm{ D_{j} u^{N,i} }_{\infty} + \norm{ \Big(\sum_{k = 1}^N |D_{kj} u^{N,i}|^2 \Big)^{1/2}}_{L^2(\scrL^N)} \leq C/N.
    \end{align*}
\end{proposition}

\begin{proof}
For $i \neq j$, we can rewrite the equation \eqref{nashfirstder2} as 
\begin{align*}
        - \partial_t u^{i,j} - \scrL^N u^{i,j} + \cA^{i,j} u^{i,j} + \sum_{k \neq j}  D_{jk} u^{k} \cB^{i,j,k} + \cC^{i,j} = 0, 
\end{align*}
with 
\begin{align*}
    \cA^{i,j} = - D_{jj}u^j D_{p^j} a^{N,j} - D_{x^j} a^{N,j}, \quad \cB^{i,j,k} = -\sum_{l \neq i} D_{p^k} a^{N,l} D_l u^i + \hat{H}^{N,i,k} 1_{k \neq j}, 
    \\
    \cC^{i,j} = -\sum_{k \neq i,j} D_{jj} u^j D_{p^j} a^{N,k} D_k u^i - \sum_{k \neq i,j} D_{x^j} a^{N,k} u^{i,k} +  \hat{H}^{N,i,j} D_{jj} u^j + D_{x^j} \hat{H}^{N,i}. 
\end{align*}

We are now going to show that for $i \neq j$, the coefficients $\cA^{i,j}$, $\cB^{i,j,k}$, $\cC^{i,j}$ satisfy the bounds 
\begin{align} \label{coeffbounds}
    \|\cA^{i,j}\|_{L^2(\scrL^N)} \leq C, \quad \left\|\sum_{k \neq j} |\cB^{i,j,k}|^2 \right\|_{\infty} \leq C/N, \quad \|\cC^{i,j}\|_{L^1(\scrL^N)} \leq C/N.
\end{align}
We start by noting that for each $t_0 \in [0,T]$, $\bx_0 \in (\R^d)^N$, and the process $(\bX_t^{t_0,\bx_0})_{t\in [t_{0},T]}$ defined in \eqref{proc:closed-loop}  we have
    \begin{align*}
        \E\bigg[\int_{t_0}^T |\cA^{i,j}(t,\bX_t^{t_0,\bx_0})|^2 dt \bigg] &\leq \|D_{p^j} a^{N,j}\|_{\infty}^2 \E\bigg[\int_{t_0}^T |D_{jj} u^j(t,\bX_t^{t_0,\bx_0})|^2 dt\bigg] + (T-t_{0}) \|D_{x^j} a^{N,j}\|_{\infty}^2
        \\
        &\leq C \norm{ |D \ba^N|_{\op} }_{\infty}^2\Big(1 +  \| |(D_{kj}u^j)_{k,j = 1,\dots,N}|_{\op} \|_{L^2(\scrL^N)}^2\Big) \leq C, 
    \end{align*}
    where we used the Lipschitz bound on $\ba^N$ from Proposition \ref{prop.fixedpoint.prelims} and the bound on $(D_{kj} u^j)_{k,j = 1,\dots,N}$ from Proposition \ref{prop.nash.preliminaries}. Next, we estimate
    \begin{align*}
        \sum_{k=1}^{N} |\cB^{{i,j,k}}|^2 &\leq C \sum_{k = 1}^N \Big| \sum_{l \neq i} D_{p^k} a^{N,l} D_l u^i\Big|^2 + C \sum_{k=1}^{N} | \hat{H}^{N,i,k}|^2 
        \\
        &\leq C \| D \ba^N\|_{\op}^{2} \sum_{k \neq i} |D_k u^i|^2 + \frac{C}{N} \leq C/N,
    \end{align*}
    where we used the bound on $\sum_{j \neq i} |D_j u^{N,i}|^2$ from Proposition \ref{prop.nash.preliminaries}, as well as the definition of $\hat{H}^{N,i,k}$ from \eqref{def:hatHik} and the decay bounds on $D_{p^{k}}a^{N,i}$ from Proposition \ref{prop.aderivscaling}.
    Finally, we note that using the same bound on $\sum_{j \neq i} |D_j u^{N,i}|^2$ from Proposition \ref{prop.nash.preliminaries} and on $D_{p^{k}}a^{N,i}, D_{x^{k}}a^{N,i}$ from Proposition \ref{prop.aderivscaling}
    \begin{align*}
        &\Big|\sum_{k \neq i,j} D_{jj} u^j D_{p^j} a^{N,k} D_k u^i\Big| \leq \frac{C}{N} |D_{jj} u^j|, \quad  \sum_{k \neq i,j} |D_{x^j} a^{N,k} D_k u^i| \leq C/N, 
        \\
        &\qquad  |\hat{H}^{N,i,j}| |D_{jj} u^j| \leq \frac{C}{N} |D_{jj} u^j|, \quad |D_{x^j} \hat{H}^{N,i}| \leq C/N, 
    \end{align*}
    from which we deduce that
    \begin{align*}
        \| \cC^{i,j} \|_{L^1(\scrL^N)} \leq \frac{C}{N} \Big( 1  + \|D_{jj} u^j \|_{L^1(\scrL^N)} \Big) \leq \frac{C}{N} \Big( 1  + \|D_{jj} u^j \|_{L^2(\scrL^N)} \Big) \leq C/N. 
    \end{align*}
    We have thus established the bounds in \eqref{coeffbounds}. 

    We now aim to use \eqref{coeffbounds} to get the desired estimates. We fix $T_0 \in [0,T)$, then $\eps \in (0,T - T_0)$. Next, we choose $t_0 \in [T_0,T_0 + \eps)$ and $\bx \in (\R^d)^N$. For $i,j,k = 1,\dots,N$, we use the notation 
    \begin{align*}
        \bX := \bX^{t_0,\bx_0}, \quad Y^{i,j} := u^{i,j}(\cdot, \bX), \quad Z^{i,j,k} := \sqrt{2} D_k u^{i,j}(\cdot,\bX), \quad Z^{i,j,0} := \sqrt{2\sigma_0} \sum_{k = 1}^N D_ku^{i,j}(\cdot, \bX). 
    \end{align*}
    By It\^o's formula, we have 
    \begin{align*}
        dY_t^{i,j} = \left( \cA^{i,j}_t Y_t^{i,j} + \sum_{k \neq j} Z_t^{k,j,k} \cB_t^{i,j,k} + \cC_t^{i,j} \right) dt + \sum_{k = 0}^N Z_t^{i,j,k} dW_t^k, \quad t_0 \leq t \leq T,
    \end{align*}
    where we have defined, by abuse of notation,
    \begin{align*}
        \cA^{i,j}_t := \cA^{i,j}(t,\bX_t), \quad \cB_t^{i,j,k} := \frac{1}{\sqrt{2}} \cB^{i,j,k}(t,\bX_t), \quad \cC_t^{i,j} := \cC^{i,j}(t,\bX_t). 
    \end{align*}
    By It\^o's formula and using the previously obtained bounds on $\cA^{i,j}, \cB^{i,j,k}$ and $ \cC^{i,j}$ and recalling the notations \eqref{inf:norms}, we deduce that for each $i \neq j$,
    \begin{align*}
        \E\bigg[ &|Y_{t_0}^{i,j}|^2 + \int_{t_0}^{T_0 + \eps} \sum_{k=1}^{N} |Z_t^{i,j,k}|^2 dt \bigg] \leq  \E\bigg[ |Y^{i,j}_{T_0 + \eps}|^2 + 2 \int_{t_0}^{T_0 + \eps} |Y_t^{i,j}| \Big|\cA^{i,j}_t Y_t^{i,j} + \sum_{k \neq j} Z_t^{k,j,k} \cB_t^{i,j,k} + \cC_t^{i,j} \Big| dt  \bigg]
        \\
        &\leq \|u^{i,j}\|^{2}_{\infty,T_0 + \eps} + C \| u^{i,j} \|_{\infty,T_0, T_0 + \eps} \E\bigg[ \int_{t_0}^{T_0 + \eps} \Big( |\cA_t^{i,j}| |Y_t^{i,j}| + \frac{1}{\sqrt{N}} \big(\sum_{k \neq j} |Z_t^{k,j,k}|^2\big)^{1/2} + |\cC_t^{i,j}| \Big) dt \bigg]
        \\
        &\leq \|u^{i,j}\|^{2}_{\infty,T_0 + \eps} + C \| u^{i,j} \|^{2}_{\infty,T_0, T_0 + \eps} \E\bigg[ \int_{t_0}^{T_0 + \eps} |\cA_t^{i,j}| dt \bigg] 
        \\
        &\qquad + C \| u^{i,j} \|_{\infty,T_0, T_0 + \eps} \E\bigg[  \int_{t_0}^{T_0 + \eps} \Big( \frac{1}{N} \sum_{k \neq j} |Z_t^{k,j,k}|^2 \Big)^{1/2}\bigg] 
        +  C \| u^{i,j} \|_{\infty,T_0, T_0 + \eps} \|\cC^{i,j}\|_{L^1(\scrL^N)}
        \\
        &\leq \|u^{i,j}\|^{2}_{\infty,T_0 + \eps} + C \| u^{i,j} \|^2_{\infty,T_0, T_0 + \eps} \sqrt{\eps} \|\cA^{i,j}\|_{L^2(\scrL^N)}
        \\
        &\qquad + C \| u^{i,j} \|_{\infty,T_0, T_0 + \eps}  \sqrt{\eps} \E\bigg[  \int_{t_0}^{T_0 + \eps} \frac{1}{N} \sum_{k \neq j} |Z_t^{k,j,k}|^2 dt \bigg]^{1/2}
        +  \frac{C}{N} \| u^{i,j} \|_{\infty,T_0, T_0 + \eps}
        \\
        &\leq \|u^{i,j}\|^{2}_{\infty, T_0 + \eps} + \Big(\frac{1}{2} + C\eps + C \sqrt{\eps} \Big) \| u^{i,j} \|^2_{\infty,T_0, T_0 + \eps} + \frac{1}{2N} \sum_{k \neq j} \E\bigg[  \int_{t_0}^{T_0 + \eps}  |Z_t^{k,j,k}|^2 dt \bigg] + \frac{C}{N^2}, 
    \end{align*}
    where in the last line we used Young's inequality. We now take a maximum over $i \neq j$, to find that 
    \begin{align*}
        \max_{i \neq j} \E\bigg[ &|Y_{t_0}^{i,j}|^2 + \int_{t_0}^{T_0 + \eps} \sum_{k=1}^{N} |Z_t^{i,j,k}|^2 dt \bigg] \leq \max_{i \neq j} \bigg( \|u^{i,j}\|^{2}_{\infty, T_0 + \eps} + \Big(\frac{1}{2} + C\eps + C \sqrt{\eps} \Big) \| u^{i,j} \|^2_{\infty,T_0, T_0 + \eps} \bigg) 
        \\
        & \qquad \qquad \qquad + \frac{1}{2N} \max_{i \neq j} \sum_{k=1}^{N} \E\bigg[  \int_{t_0}^{T_0 + \eps}  |Z_t^{i,j,k}|^2 dt \bigg] + \frac{C}{N^2}, 
    \end{align*}
    and so by absorbing the penultimate term we arrive at
         \begin{align} \label{yij.est}
        \max_{i \neq j} \E\bigg[ &|Y_{t_0}^{i,j}|^2 + \int_{t_0}^{T_0 + \eps} \sum_k |Z_t^{i,j,k}|^2 dt \bigg] \leq \max_{i \neq j} \bigg(\|u^{i,j}\|^{2}_{\infty,T_0 + \eps} + \Big(\frac{1}{2} + C\eps + C \sqrt{\eps} \Big) \| u^{i,j} \|^2_{\infty,T_0, T_0 + \eps} \bigg) + \frac{C}{N^2}. 
    \end{align}
    Now we take a supremum over $t_0 \in [T_0, T_0 + \eps]$ and $\bx_0 \in (\R^d)^N$ {on the left hand side}, and recall the definition of $Y^{i,j}$, to find 
    \begin{align*}
        \max_{i \neq j}  \|u^{i,j}\|^{2}_{\infty,T_0, T_0 + \eps} \leq \max_{i \neq j} \|u^{i,j}\|^{2}_{\infty,T_0 + \eps} +  \Big(\frac{1}{2} + C\eps + C \sqrt{\eps} \Big) \max_{i \neq j} \| u^{i,j} \|^2_{\infty,T_0, T_0 + \eps} + \frac{C}{N^2}.
    \end{align*}
    We deduce that there is a constant $\eps_0 > 0$ independent of $N$ such that if $\eps < \eps_0$, then we have 
    \begin{align*}
         \max_{i \neq j}  \|u^{i,j}\|^{2}_{\infty,T_0, T_0 + \eps} \leq \max_{i \neq j} \|u^{i,j}\|^{2}_{\infty,T_0 + \eps} + C/N^2. 
    \end{align*}
    Iterating this inequality and recalling that by assumption $\|u^{i,j}(T,\cdot)\|_{\infty} \leq C/N$, we get the desired bound on $u^{i,j}$. To get the corresponding bound on $(D_ku^{i,j})_{k = 1,\dots,N}$, we return to \eqref{yij.est} and recall the definition of $Z^{i,j,k}$.
\end{proof}

\begin{corollary}\label{cor:kij}
The bounds established in Proposition \ref{prop.uij} imply that there exists a constant $C>0$ such that for $N\in\mathbb{N}$ large enough we have that
\begin{align*}
 \norm{D_k u^{i,j}}_{L^2(\scrL^N)} = \norm{u^{i,j,k}}_{L^2(\scrL^N)}  \leq C/N,\ \ \forall i,j,k\in\{1,\dots,N\},\ i\neq j,
\end{align*}
and
\begin{align*}
 \norm{|D_k u^{i,j}|^2}_{L^1(\scrL^N)} = \norm{|u^{i,j,k}|^2}_{L^1(\scrL^N)}  \leq C/N^{2},\ \ \forall i,j,k\in\{1,\dots,N\},\ i\neq j.
\end{align*}
\end{corollary}

\subsection{Estimates on the second derivatives}

We now need to differentiate further the Nash system, to find an equation for the second derivatives $u^{i,j,k} := D_{kj} u^{N,i}$. By explicit computation, we find that
\begin{align*}
    &D_{k} \Big( \scrL^N u^{i,j} \Big) = \scrL^N u^{i,j,k} + \sum_{l = 1}^N \sum_{q = 1}^N D_{p^q} a^{N,l} u^{q,k,q} u^{i,j,l} + \sum_{l = 1}^N D_{x^k} a^{N,l} u^{i,j,l}, 
    \\
    &D_k \Big( \sum_{l \neq i} \sum_{q = 1}^N D_{jq } u^q D_{p^q} a^{N,l} D_l u^i \Big) = \sum_{l \neq i} \sum_{q = 1}^N D_k u^{q,j,q} D_{p^q} a^{N,l} u^{i,l} + \sum_{l \neq i} \sum_{q = 1}^N u^{q,j,q} D_{p^q} a^{N,l} u^{i,k,l} 
    \\
    &\qquad + \sum_{l \neq i} \sum_{q,n = 1}^N D_{p^n p^q} a^{N,l}  u^{q,j,q} u^{n,k,n} u^{i,l} + \sum_{l \neq i} \sum_{q = 1}^N D_{x^k p^q} a^{N,l} u^{q,j,q} u^{i,l}
    \\
    &D_k \Big( \sum_{l \neq i} D_{x^j} a^{N,l} u^{i,l} \Big) = \sum_{l \neq i} D_{x^j} a^{N,l} u^{i,k,l} + \sum_{l \neq i} \sum_{q = 1}^N D_{p^q x^j} a^{N,l} u^{q,k,q} u^{i,l} + \sum_{l \neq i} D_{x^kx^j} a^{N,l} u^{i,l}
    \\
    &D_k \Big(  \sum_{l=1}^{N} \hat{H}^{N,i,l} D_{jl} u^{N,l} \Big) = \sum_{l=1}^{N} \hat{H}^{N,i,l} D_k u^{l,j,l} + \sum_{l,q=1}^{N} D_{p^q} \hat{H}^{N,i,l} u^{q,k,q} u^{l,j,l} 
    \\
    &\qquad + \sum_{l=1}^{N} D_{x^k} \hat{H}^{N,i,l} u^{l,j,l}
    \\
    &D_k \big( D_{x^j} \hat{H}^{N,i} \big) = \sum_{l = 1}^N D_{p^l x^j} \hat{H}^{N,i} u^{l,k,l} + D_{x^kx^j} \hat{H}^{N,i}.
\end{align*}
We deduce that for each $i,j,k = 1,\dots,N$, $u^{i,j,k}$ satisfies
    \begin{align} \label{nash.seconderiv}
    \begin{cases} \ds - \partial_t u^{i,j,k} - \scrL^N u^{i,j,k} + \sum_{q = 1}^4 \cT^{i,j,k,q} = 0, \quad (t,\bx) \in [0,T] \times (\R^d)^N, \vspace{.2cm}
    \\ \ds 
    u^{i,j,k}(T,\bx) = D_{x^k x^j} G^{N,i}(\bx), \quad \bx \in (\R^d)^N, 
    \end{cases}
\end{align}
where 
\begin{align*}
   &\cT^{i,j,k,1} := - \sum_{l \neq i} \sum_{q = 1}^N D_k u^{q,j,q} D_{p^q} a^{N,l} u^{i,l} + \sum_{l=1}^{N} \hat{H}^{N,i,l} D_k u^{l,j,l}, 
   \\
   &\cT^{i,j,k,2} := - \sum_{l = 1}^N \sum_{q = 1}^N D_{p^q} a^{N,l} u^{q,k,q} u^{i,j,l} - \sum_{l = 1}^N D_{x^k} a^{N,l} u^{i,j,l} - \sum_{l \neq i} \sum_{q = 1}^N D_{p^q} a^{N,l} u^{q,j,q} u^{i,l,k} 
   \\
   &\qquad \qquad \qquad - \sum_{l \neq i} D_{x^j} a^{N,l} u^{i,k,l},
   \\
   &\cT^{i,j,k,3} := - \sum_{l \neq i} \sum_{q,n = 1}^N D_{p^n p^q} a^{N,l} u^{q,j,q} u^{n,k,n} u^{i,l} + \sum_{l,q=1}^N D_{p^q} \hat{H}^{N,i,l} u^{q,k,q} u^{l,j,l} 
   \\
   &\qquad \qquad \qquad - \sum_{l \neq i} \sum_{q = 1}^N D_{x^kp^q} a^{N,l} u^{q,j,q} u^{i,l} - \sum_{l \neq i} \sum_{q = 1}^N D_{p^q x^j} a^{N,l} u^{q,k,q} u^{i,l} 
   \\
   &\qquad \qquad \qquad + \sum_{l=1}^{N} D_{x^k} \hat{H}^{N,i,l} u^{l,j,l} + \sum_{l = 1}^N D_{p^lx^j} \hat{H}^{N,i} u^{l,k,l}
   \\
   &\cT^{i,j,k,4} := - \sum_{l \neq i} D_{x^kx^j} a^{N,l} u^{i,l} + D_{x^kx^j} \hat{H}^{N,i}.
\end{align*}

\begin{proposition} \label{prop.uiji}

    There is a constant $C>0$ such that for all $N$ large enough and all $i \neq j$,
    \begin{align*}
        \norm{u^{i,j,i}}_{\infty} + \norm{\Big( \sum_{k = 1}^N |D_k u^{i,j,i}|^2 \Big)^{1/2}}_{L^2(\scrL^N)} \leq C/N.
    \end{align*}
\end{proposition}

\begin{proof}
For $i \neq j$, we are going to reorganize the equation for $u^{i,j,i}$ as follows:
\begin{align*}
       - \partial_t u^{i,j,i} - \scrL^N u^{i,j,i} + \sum_{k=1}^{N} \cB^{i,j,k} D_i u^{k,j,k} + \cC^{i,j} = 0, 
\end{align*}
with adopting the notations
\begin{align*}
   \cB^{i,j,k} := - \sum_{l \neq i} D_{p^k} a^{N,l} u^{i,l} +  \hat{H}^{N,i,k},
    \quad \cC^{i,j} := \sum_{q = 2}^{4} \cT^{i,j,i,q}. 
\end{align*}
Our aim will be to prove that the coefficients $\cB^{i,j,k}$ and $\cC^{i,j}$ satisfy the bounds 
\begin{align} \label{coeffbounds.iji}
  \norm{ \sum_{k=1}^{N} |\cB^{i,j,k}|^2 }_{\infty} \leq C/N, \quad \norm{\cC^{i,j}}_{L^1(\scrL^N)} \leq C/N.
\end{align}
For the first bound in \eqref{coeffbounds.iji}, we use Propositions \ref{prop.aderivscaling} and \ref{prop.uij} to find that
\begin{align*}
    \sum_{k=1}^{N} \Big| \sum_{l \neq i} D_{p^k} a^{N,l} u^{i,l} \Big|^2 \leq  \frac{C}{N^2} \sum_{k=1}^{N} \Big|\sum_{l=1}^{N} |D_{p^k} a^{N,l}| \Big|^2 \leq C/N,
 \end{align*}
 and then we use Lemma \ref{lem.hatH} to obtain
 \begin{align*}
   \sum_{k=1}^{N} \Big| \hat{H}^{N,i,k } \Big|^2 \leq C/N.
\end{align*}
This gives the first estimate in \eqref{coeffbounds.iji}.
We now turn our attention to bounding $\cC^{i,j}$. We will first bound $\cT^{i,j,i,2}$. Again making use of Proposition \ref{prop.aderivscaling} and Lemma \ref{lem.hatH}, we have
\begin{align} \label{Tiji2.1}
&\left| \sum_{l,q = 1}^N D_{p^q} a^{N,l} u^{q,i,q} u^{i,j,l} \right| \leq C |u^{i,i,i}| \sum_{l = 1}^N |D_{p^i} a^{N,l}| |u^{i,j,l}| + C \left( \sum_{q \neq i} |u^{q,i,q}|^2 \right)^{1/2} \left( \sum_{l = 1}^N |u^{i,j,l}|^2 \right)^{1/2}
   \nonumber  \\
    &\qquad \leq C|u^{i,i,i}| |D_{p^i} a^{N,i}| |u^{i,j,i}| + \frac{C}{\sqrt{N}} |u^{i,i,i}| \left(\sum_{l = 1}^N |u^{i,j,l}|^2 \right)^{1/2} +  C \left( \sum_{q \neq i} |u^{q,i,q}|^2 \right)^{1/2} \left( \sum_{l = 1}^N |u^{i,j,l}|^2 \right)^{1/2}
   \nonumber \\
    &\qquad \leq C |u^{i,i,i}||u^{i,j,i}| + \frac{C}{N} |u^{i,i,i}|^2 + C\sum_{l = 1}^N |u^{i,j,l}|^2 + C\sum_{q \neq i} |u^{q,i,q}|^2
   \nonumber \\
  & \qquad \leq C N |u^{i,j,i}|^{2} + \frac{C}{N} |u^{i,i,i}|^2 + C\sum_{l = 1}^N |u^{i,j,l}|^2 + C\sum_{q \neq i} |u^{q,i,q}|^2,
   \nonumber \\
    &\left| \sum_{l=1}^{N} D_{x^i} a^{N,l} u^{i,j,l} \right| \leq |D_{x^i} a^{N,i}||u^{i,j,i}| + \frac{C}{\sqrt{N}} \left(\sum_{l=1}^{N} |u^{i,j,l}|^2 \right)^{1/2} \leq C |u^{i,j,i}| + \frac{C}{N} + C \sum_{l=1}^{N} |u^{i,j,l}|^2\nonumber \\
    &\qquad \le C N |u^{i,j,i}|^{2} + \frac{C}{N} + C \sum_{l=1}^{N} |u^{i,j,l}|^2.
\end{align}
Similarly, we have 
\begin{align} \label{Tiji2.2}
    &\left| \sum_{l \neq i} \sum_{q = 1}^N D_{p^q} a^{N,l} u^{q,j,q} u^{i,l,i} \right| \leq C |u^{j,j,j}| \sum_{l \neq i} |D_{p^j} a^{N,l}| |u^{i,l,i}| + C \left( \sum_{l \neq i} |u^{i,l,i}|^2 \right)^{1/2}\left( \sum_{q \neq j} |u^{q,j,q}|^{2} \right)^{1/2}
 \nonumber   \\
    &\qquad \leq C |u^{j,j,j}| |u^{i,j,i}| + \frac{C}{\sqrt{N}} |u^{j,j,j}| \left( \sum_{l \neq i} |u^{i,l,i}|^2 \right)^{1/2} + C \left( \sum_{l \neq i} |u^{i,l,i}|^2 \right)^{1/2}\left( \sum_{q \neq j} |u^{q,j,q}|^{2} \right)^{1/2}
  \nonumber  \\
    &\qquad \leq \frac{C}{N} |u^{j,j,j}|^2 + CN |u^{i,j,i}|^2 + C\sum_{l \neq i} |u^{i,l,i}|^2 + C\sum_{q \neq j} |u^{q,j,q}|^2,
 \nonumber   \\
    &\left| \sum_{l \neq i} D_{x^j} a^{N,l} u^{i,l,i} \right| \leq C|u^{i,j,i}| + \frac{C}{\sqrt{N}} \left(\sum_{l \neq i} |u^{i,l,i}|^2\right)^{1/2} \leq C |u^{i,j,i}| + \frac{C}{N} + C\sum_{l \neq i} |u^{i,l,i}|^2 \nonumber\\ 
    & \qquad \leq CN |u^{i,j,i}|^{2} + \frac{C}{N} + C\sum_{l \neq i} |u^{i,l,i}|^2.
\end{align}
Putting together \eqref{Tiji2.1} and \eqref{Tiji2.2} we find that there exists a constant $C>0$ independent of $N$ such that
\begin{align*}
\left| \cT^{i,j,i,2} \right| & \le \frac{C}{N}\left(1 +  |u^{i,i,i}|^2 + |u^{j,j,j}|^2 \right) + C N |u^{i,j,i}|^{2} + C\sum_{l \neq i} |u^{i,l,i}|^2 + C\sum_{l = 1}^N |u^{i,j,l}|^2\\ 
&+ C\sum_{q \neq i} |u^{q,i,q}|^2 + C\sum_{q \neq j} |u^{q,j,q}|^2.
\end{align*}
By Proposition \ref{prop.nash.preliminaries},  $u^{i,i,i}$ and $u^{j,j,j}$ are uniformly bounded in $L^{2}(\scrL^{N})$. To bound the other terms we can apply Proposition \ref{prop.uij}, and we deduce that 
\begin{align} \label{Tiji.2bound}
    \norm{\cT^{i,j,i,2}}_{L^1(\scrL^N)} \leq C/N.
\end{align}
To bound $\cT^{i,j,i,3}$, we use similar arguments to deduce that
\begin{align*}
    & \Big| \sum_{l \neq i} \sum_{q,n = 1}^N D_{p^np^q} a^{N,l} u^{q,j,q} u^{n,i,n} u^{i,l} \Big| \leq \frac{C}{N} \sum_{l \neq i} \sum_{q,n = 1}^N |D_{p^np^q} a^{N,l}| |u^{q,j,q} |u^{n,i,n}|
    \\
    &\qquad \leq \frac{C}{N} \sum_{q,n = 1}^N \left(\frac{1}{N} + 1_{q = n} \right) |u^{q,j,q}| |u^{n,i,n}| 
     \leq \frac{C}{N} \big( \sum_{q = 1}^N |u^{q,j,q}|^2 \big)^{1/2} \big( \sum_{n = 1}^N |u^{n,i,n}|^2 \big)^{1/2} 
    \\
    &\qquad \leq \frac{C}{N} \sum_{q = 1}^N |u^{q,j,q}|^2 + \frac{C}{N} \sum_{n = 1}^N |u^{n,i,n}|^2,
    \\
    & \Big|\sum_{l,q=1}^{N} D_{p^q} \hat{H}^{N,i,l} u^{q,i,q} u^{l,j,l} \Big| \leq  C\sum_{l=1}^{N} |D_{p^i} \hat{H}^{N,i,l}| |u^{i,i,i}| |u^{l,j,l}| + C \sum_{l=1}^{N} \sum_{q \neq i} |D_{p^q}\hat{H}^{N,i,l}| |u^{q,i,q}| |u^{l,j,l}|
    \\
    &\qquad \leq \frac{C}{N} |u^{i,i,i}| |u^{j,j,j}| + \frac{C}{\sqrt{N}} |u^{i,i,i}| \left(\sum_{l \neq j} |u^{l,j,l}|^2\right)^{1/2} + \frac{C}{N} \left( \sum_{q=1}^{N} |u^{q,i,q}|^2 \right)^{1/2} \left( \sum_{l=1}^{N} |u^{l,j,l}|^2 \right)^{1/2}
    \\
    &\qquad \leq \frac{C}{N} |u^{i,i,i}| |u^{j,j,j}| + \frac{C}{N} |u^{i,i,i}|^2 + C \sum_{l \neq i} |u^{l,j,l}|^2 + \frac{C}{N} \left(\sum_{q=1}^{N} |u^{q,j,q}|^2 + \sum_{l=1}^N |u^{l,j,l}|^2 \right), 
    \\
    &\Big|\sum_{l \neq i} \sum_{q=1}^N D_{x^i p^q} a^{N,l} u^{q,j,q} u^{i,l} \Big| \leq \frac{C}{N} \sum_{q=1}^N |u^{q,j,q}| \leq \frac{C}{N} |u^{j,j,j}| + \frac{C}{\sqrt{N}}  \left( \sum_{q \neq j} |u^{q,j,q}|^2 \right)^{1/2}
    \\
    &\qquad \leq \frac{C}{N} \big(1 + |u^{j,j,j}|\big) + \sum_{q \neq j} |u^{q,j,q}|^2, 
    \\
   & \Big| \sum_{l \neq i} \sum_{q=1}^ND_{p^qx^j} a^{N,l} u^{q,i,q} u^{i,l} \Big| \leq \frac{C}{N} \sum_{l \neq i} \sum_{q=1}^N |D_{p^q x^j} a^{N,l}| |u^{q,i,q}| 
    \\
    &\qquad \leq \frac{C}{N} \sum_{q=1}^N |u^{q,i,q}| \leq \frac{C}{N} |u^{i,i,i}| + \frac{C}{N} + \frac{C}{N} \sum_{q \neq i} |u^{q,i,q}|^2, 
    \\
    &\Big| {\sum_{l=1}^N} D_{x^i} \hat{H}^{N,i,l} u^{l,j,l} \Big| \leq \frac{C}{N} \sum_{l=1}^N |u^{l,j,l}|  \leq \frac{C}{N} |u^{j,j,j}| + \frac{C}{N} + \sum_{l \neq j} |u^{l,j,l}|^2, 
    \\
    &\Big| \sum_{l =1}^N D_{p^l x^j} \hat{H}^{N,i} u^{l,i,l} \Big| \leq \frac{C}{N} |u^{i,i,i}| + \sum_{l \neq i} |D_{p^l x^j} \hat{H}^{N,i}| |u^{l,i,l}| \leq \frac{C}{N} |u^{i,i,i}| + \frac{C}{N} + C \sum_{l \neq i} |u^{l,i,l}|^2,
\end{align*}
which when combined with Proposition \ref{prop.uij} is enough to conclude that 
\begin{align} \label{Tiji3.bound}
    \norm{\cT^{i,j,i,3}}_{L^1(\scrL^N)} \leq C/N.
\end{align}
Finally, we have
\begin{align*}
    \Big| \sum_{l \neq i} D_{x^i x^j} a^{N,l} u^{i,l} \Big| \leq \frac{C}{N}, \quad |D_{x^i x^j} \hat{H}^{N,i}| \leq C/N,
\end{align*}
from which we deduce
\begin{align} \label{Tiji4.bound}
    \norm{\cT^{i,j,i,4}}_{\infty} \leq C/N. 
\end{align}
Combining \eqref{Tiji.2bound}, \eqref{Tiji3.bound} and \eqref{Tiji4.bound}, we confirm the last bound in \eqref{coeffbounds.iji}.

We now have established the bound \eqref{coeffbounds.iji}, and we wish to use this to estimate $u^{i,j,i}$. To do this, we will use a strategy very similar to the one used in the proof of Proposition \ref{prop.uij}. We fix $T_0 \in [0,T)$, then $\eps \in (0,T-T_0)$. Next, we choose $t_0 \in [T_0,T_0 + \eps)$, and $\bx_0 \in (\R^d)^N$. For $i,j,k = 1,\dots,N$, we use the notation 
\begin{align*}
    \bX = \bX^{t_0,\bx_0}, \quad Y^{i,j,i} = u^{i,j,i}(\cdot,\bX), \quad Z^{i,j,i,k} = \sqrt{2} D_k u^{i,j,i}(\cdot,\bX), \quad Z^{i,j,i,0} = \sqrt{2\sigma_0} \sum_{k = 1}^n D_{k} u^{i,j,i}(\cdot,\bX).
\end{align*}
By It\^o's formula, we have 
\begin{align*}
    dY_t^{i,j,i} = \Big( \sum_{k=1}^N \cB_t^{i,j,k} Z_t^{k,j,k,i} + \cC_t^{i,j} \Big) dt + \sum_{k = 0}^N Z_t^{i,j,i,k} dW_t^k, 
\end{align*}
where 
\begin{align*}
    \cB_t^{i,j,k} = \frac{1}{\sqrt{2}} \cB^{i,j,k}(t,\bX_t), \quad \cC_t^{i,j} = \cC^{i,j}(t,\bX_t).
\end{align*}
Again using It\^o's formula and the bounds in \eqref{coeffbounds.iji}, we deduce that for each $i \neq j$, 
\begin{align*}
    &\E\bigg[ |Y_{t_0}^{i,j,i}|^2 + \int_{t_0}^{T_0 + \eps} {\sum_{k=0}^N} |Z_t^{i,j,i,k}|^2 dt \bigg] \leq \E\bigg[ |Y_{T_0 + \eps}^{i,j}|^2 + \int_{t_0}^{T_0 + \eps} |Y_t^{i,j,i}| \Big|{\sum_{k=1}^N} \cB_t^{i,j,k} Z_t^{k,j,k,i} + \cC_t^{i,j} \Big| dt \bigg]
    \\
    &\qquad \leq \norm{u^{i,j,i}}_{\infty, T_0 + \eps}^2 + \|u^{i,j,i}\|_{\infty, T_0,T_0 + \eps} \E\bigg[ \int_{t_0}^{T_0 + \eps} \Big({\sum_{k=1}^N} |\cB_t^{i,j,k}| |Z_t^{k,j,k,i}| + |\cC_t^{i,j}| \Big)dt \bigg]
    \\
    &\qquad \leq \norm{u^{i,j,i}}_{\infty, T_0 + \eps}^2 + C \|u^{i,j,i}\|_{\infty, T_0,T_0 + \eps}  \E\bigg[\int_{t_0}^{T_0 + \eps} \Big(\frac{1}{N} {\sum_{k=1}^N} |Z_t^{k,j,k,i}|^2 \Big)^{1/2} dt  \bigg]
    \\
    &\qquad \qquad \qquad + \frac{C}{N} \|u^{i,j,i}\|_{\infty, T_0,T_0 + \eps} 
    \\
    &\qquad \leq \norm{u^{i,j,i}}_{\infty, T_0 + \eps}^2 + C \sqrt{\eps} \|u^{i,j,i}\|_{\infty, T_0,T_0 + \eps} \E\bigg[ \int_{t_0}^{T_0 + \eps} \frac{1}{N} {\sum_{k=1}^N} |Z_t^{k,j,k,i}|^2 dt \bigg]^{1/2}
    \\
    &\qquad \qquad \qquad  + \frac{C}{N} \|u^{i,j,i}\|_{\infty, T_0,T_0 + \eps} 
    \\
    &\qquad \leq \norm{u^{i,j,i}}_{\infty, T_0 + \eps}^2 + \Big(C \eps + \frac{1}{2} \Big) \|u^{i,j,i}\|^2_{\infty, T_0,T_0 + \eps} + \frac{1}{2} \max_{i' \neq j'} \E\bigg[ \int_{t_0}^{T_0 + \eps} {\sum_{k=1}^N} |Z_t^{i',j',i',k}|^2 dt \bigg] + \frac{C}{N^2},
\end{align*}
{where in the last inequality we have used the fact that
\begin{align*}
\E\bigg[ \int_{t_0}^{T_0 + \eps} \frac{1}{N} {\sum_{k=1}^N} |Z_t^{k,j,k,i}|^2 dt \bigg] &  = \frac{1}{N} \E\bigg[ \int_{t_0}^{T_0 + \eps}  |Z_t^{j,j,j,i}|^2 dt \bigg] + \frac{1}{N}\sum_{k\neq j} \E\bigg[ \int_{t_0}^{T_0 + \eps}  |Z_t^{k,j,k,i}|^2 dt \bigg]\\
& \le \frac{1}{N} \E\bigg[ \int_{t_0}^{T_0 + \eps}  |Z_t^{j,j,j,i}|^2 dt \bigg] + \max_{k\neq j} \E\bigg[ \int_{t_0}^{T_0 + \eps}  |Z_t^{k,j,k,i}|^2 dt \bigg]\\
& \le \frac{1}{N} \E\bigg[ \int_{t_0}^{T_0 + \eps}  |Z_t^{j,j,j,i}|^2 dt \bigg] + \max_{k\neq j} \E\bigg[ \int_{t_0}^{T_0 + \eps}  \sum_{l=1}^{N}|Z_t^{k,j,k,l}|^2 dt \bigg]
\end{align*}
}

Take a max over $i \neq j$ to find that 
\begin{align} \label{Zijikbound}
    \max_{i \neq j} &\E\bigg[ |Y_{t_0}^{i,j,i}|^2 + \int_{t_0}^{T_0 + \eps} {\sum_{k=0}^N} |Z_t^{i,j,i,k}|^2 dt \bigg]
    \nonumber \\
    &\leq \max_{i \neq j} \norm{u^{i,j,i}}_{\infty, T_0 + \eps}^2 + \big(C \eps + \frac{1}{2} \big) \max_{i \neq j} \|u^{i,j,i}\|^2_{\infty, T_0,T_0 + \eps} + \frac{1}{2} \max_{i \neq j} \E\bigg[ \int_{t_0}^{T_0  +\eps} {\sum_{k=1}^N} |Z_t^{i,j,i,k}|^2 dt \bigg] + \frac{C}{N^2}, 
\end{align}
from which we deduce that 
\begin{align*}
    \max_{i \neq j} \E\big[ |Y_{t_0}^{i,j,i}|^2 \big] \leq \max_{i \neq j} \norm{u^{i,j,i}}_{\infty, T_0 + \eps}^2 + \big(C \eps + \frac{1}{2} \big) \max_{i \neq j} \|u^{i,j,i}\|^2_{\infty, T_0,T_0 + \eps} { + \frac{C}{N^2}}. 
\end{align*}
Take a supremum over $t_0 \in [T_0,T_0 + \eps]$ and then $\bx_0 \in (\R^d)^N$, and recall the definition of $Y^{i,j,i}$ to obtain 
\begin{align*}
    \max_{i \neq j} \|u^{i,j,i}\|^2_{\infty, T_0,T_0 + \eps} \leq \max_{i \neq j} \norm{u^{i,j,i}}_{\infty, T_0 + \eps}^2 + \big(C \eps + \frac{1}{2} \big) \max_{i \neq j} \|u^{i,j,i}\|^2_{\infty, T_0,T_0 + \eps} { + \frac{C}{N^2}}.
\end{align*}
We deduce that there is a constant $\eps_0 > 0$ which is independent of $N$ such that for any $\eps < \eps_0$, we have 
\begin{align*}
     \max_{i \neq j} \|u^{i,j,i}\|^2_{\infty, T_0,T_0 + \eps} \leq C\max_{i \neq j} \norm{u^{i,j,i}}_{\infty, T_0 + \eps}^2 + C/N^2, 
\end{align*}
and so iterating this bound backwards in time we obtain the bound on $u^{i,j,i}$. The bound on $D_k u^{i,j,i}$ comes from \eqref{Zijikbound} after recalling the definition of $Z^{i,j,i,k}$. 
\end{proof}

\begin{corollary}\label{cor:kiji}
The bounds established in Proposition \ref{prop.uiji} imply that there exists a constant $C>0$ such that for $N\in\mathbb{N}$ large enough we have
\begin{align*}
 \norm{|D_k u^{i,j,i}|^2}_{L^1(\scrL^N)} \leq C/N^{2},\ \ \forall i,j,k\in\{1,\dots,N\},\ i\neq j.
\end{align*}
\end{corollary}

\begin{proposition} \label{prop.iii}
  Let Assumptions \ref{assump.regularity.disp} and \ref{assump.disp} hold. Then there is a constant $C$ such that for all $N\in\mathbb{N}$ large enough and all $i \neq j$, 
   \begin{align*}
       &\norm{u^{i,i,i}}_{\infty} + \norm{\Big(\sum_{k = 1}^N |D_k u^{i,i,i}|^2 \Big)^{1/2}}_{L^2(\scrL^N)} \leq C, 
       \\
        &\norm{u^{i,j,j}}_{\infty} + \norm{\Big(\sum_{k = 1}^N |D_k u^{i,j,j}|^2 \Big)^{1/2}}_{L^2(\scrL^N)} \leq C/N,
   \end{align*}
\end{proposition}

\begin{proof}
    We start with the bound on $u^{i,i,i}$. We can write the equation for $u^{i,i,i}$ as 
    \begin{align*}
        - \partial_t u^{i,i,i} - \scrL^N u^{i,i,i} + \cB^i D_i u^{i,i,i} + \cC^i = 0, 
    \end{align*}
    with 
    \begin{align*}
        \cB^i = - \sum_{l \neq i} D_{p^i} a^{N,l} u^{i,l} + \hat{H}^{N,i,i}, \quad \cC^i = - \sum_{l \neq i} \sum_{q \neq i} D_i u^{q,i,q} D_{p^q} a^{N,l} u^{i,l} + \sum_{l \neq i} \hat{H}^{N,i,l} D_i u^{l,i,l}  + \sum_{q = 2}^{4} \cT^{i,i,i,q},
    \end{align*}
{where the terms $\cT^{i,i,i,q}$ are defined in \eqref{nash.seconderiv}.} Our aim is now to establish the bounds 
    \begin{align} \label{coefficients.iii}
        \norm{\cB^i}_{\infty} \leq C/N, \quad \norm{\cC^i}_{L^1(\scrL^N)} \leq C.
    \end{align}
   
   The bound on $\cB^i$ follows {directly} from Propositions \ref{prop.aderivscaling} and \ref{prop.uij} {and from the definition of $\hat{H}^{N,i,i}$ in \eqref{def:hatHik}.} For the bound on $\cC^i$, we first note that 
 \begin{align*}
       \Big| \sum_{l \neq i} \sum_{q \neq i} D_i u^{q,i,q} D_{p^q} a^{N,l} u^{i,l}  \Big| & \leq \frac{C}{\sqrt{N}} \left(\sum_{q \neq i} |D_i u^{q,i,q}|^2 \right)^{1/2} \left(\sum_{l \neq i} \sum_{q \neq i} |D_{p^q} a^{N,l}|^{2}\right)^{1/2}\\ 
       & \leq C\sum_{q \neq i} |D_i u^{q,i,q}|^2 +\frac{C}{N} \left( \sum_{l,q, l\neq q} |D_{p^q} a^{N,l}|^{2} + \sum_{l\neq i} |D_{p^l} a^{N,l}|^{2} \right)\\
       & \leq C\sum_{q \neq i} |D_i u^{q,i,q}|^2 +C(1+1/N), 
  \end{align*}

\begin{align*}
       &\Big|\sum_{l\neq i} \hat{H}^{N,i,l} D_i u^{l,i,l} \Big| \leq \frac{C}{\sqrt{N}} \left( \sum_{l \neq i} |D_i u^{l,i,l}|^2  \right)^{1/2} \leq C/N + C\sum_{l \neq i} |D_i u^{l,i,l}|^2.
    \end{align*}
 We notice that by Corollary \ref{cor:kiji} we have
 \begin{align*}
 \norm{|D_k u^{i,j,i}|^2}_{L^1(\scrL^N)} \leq C/N^2,\ \forall\ i,j,k\in\{1,\dots,N\},i\neq j. 
 \end{align*}
 
 So,
 \begin{align}\label{estim:CN}
 \norm{\sum_{l \neq i} |D_i u^{l,i,l}|^2}_{L^1(\scrL^N)} \le  \sum_{l \neq i}\norm{ |D_i u^{l,i,l}|^2}_{L^1(\scrL^N)} \le C/N.
 \end{align}
 Therefore, putting together the previous arguments, we obtain
  \begin{align*}
        \norm{ - \sum_{l \neq i} \sum_{m \neq i} D_i u^{m,i,m} D_{p^m} a^{N,l} u^{i,l} + \sum_{l\neq i}  \hat{H}^{i,l} D_i u^{l,i,l}  }_{L^1(\scrL^N)} \leq C.
    \end{align*}

    It remains to shows that
    \begin{align} \label{iii.suff}
        \norm{\cT^{i,i,i,q}}_{L^1(\scrL^N)} \leq C, \quad \text{for } q = 2,3,4.
    \end{align}
    To bound $\cT^{i,i,i,2}$, we note that {$u^{i,l_{1},l_{2}}=u^{i,l_{2},l_{1}}$ for any $i,l_{1},l_{2}\in\{1,\dots,N\}$ and}
    \begin{align*}
        &\Big| \sum_{l,q=1}^{N} D_{p^q} a^{N,l} u^{q,i,q} u^{i,i,l} \Big| 
        \leq C \left(\sum_{q=1}^{N} |u^{q,i,q}|^2\right)^{1/2} \left(\sum_{l=1}^{N} |u^{i,l,i}|^2 \right)^{1/2} 
        \\
        &\qquad \qquad \qquad \leq C\sum_{q=1}^{N} |u^{q,i,q}|^2  + C \sum_{l=1}^{N} |u^{i,l,i}|^2 
        \\
        &\Big| \sum_{l \neq i} \sum_{q=1}^{N} D_{p^q} a^{N,l} u^{q,i,q} u^{i,l,i} \Big| \leq C \left( \sum_{q=1}^{N} |u^{q,i,q}|^2 \right)^{1/2} \left( \sum_{l\neq i} u^{i,l,i} \right)^{1/2} 
        \\
         &\qquad \qquad \qquad \leq C\sum_{q=1}^{N} |u^{q,i,q}|^2 + C \sum_{l\neq i} |u^{i,l,i}|^2
        \\
        &\Big| \sum_{l=1}^{N} D_{x^{i}} a^{N,l} u^{i,{i,l}} \Big| \leq C \left(\sum_{l = 1}^N |u^{i,l,i}|^2 \right)^{1/2} \leq C + C \sum_{l = 1}^N |u^{i,l,i}|^2,
        \\
        &\Big| \sum_{l \neq i} D_{x^i} a^{N,l} u^{i,{i,l}} \Big| \leq \frac{C}{\sqrt{N}} \left(\sum_{l \neq i} |u^{i,l,i}|^2\right)^{1/2} \leq \frac{C}{N} + \sum_{l \neq i} |u^{i,l,i}|^2, 
        \end{align*}
        which combined with Propositions \ref{prop.nash.preliminaries} and \ref{prop.uiji} gives the bund $\|\cT^{i,i,i,2}\|_{L^1(\scrL^N)} \leq C$. Turning to $\cT^{i,i,i,3}$, we have 
        \begin{align*}
        &\Big| \sum_{l \neq i} \sum_{q,n=1}^{N} D_{p^np^q} a^{N,l} u^{q,i,q} u^{n,i,n} u^{i,l}\Big|
        \\
        &\qquad \qquad \qquad \leq \frac{C}{N} \sum_{q,n=1}^{N} \Big( \sum_{l \neq i} |D_{p^n p^q} a^{N,l}|\Big) |u^{q,i,q}| |u^{n,i,n}| 
        \\
        &\qquad \qquad \qquad  \leq \frac{C}{N} \sum_{q,n = 1}^N \big(1/N + 1_{q = n} \big) |u^{q,i,q}| |u^{n,i,n}| \leq \frac{C}{N} \sum_{q=1}^{N} |u^{q,i,q}|^2, 
        \\
        &\Big| \sum_{l,q=1}^{N} D_{p^q} \hat{H}^{N,i,l} u^{q,i,q} u^{l,i,l} \Big| \leq C \sum_{q=1}^{N} |u^{q,i,q}|^2, 
        \\
        &\Big| \sum_{l \neq i} \sum_{q=1}^{N} D_{x^i p^q} a^{N,l} u^{q,i,q} u^{i,l} \Big| \leq \frac{C}{\sqrt{N}} \left(\sum_{q=1}^{N} |u^{q,j,q}|^2\right)^{1/2} \leq C/N + C\sum_{q=1}^{N} |u^{q,j,q}|^2, 
        \\
        &\Big| \sum_{l \neq i} \sum_{q=1}^{N} D_{p^qx^i} a^{N,l} u^{q,i,q} u^{i,l} \Big| \leq \frac{C}{\sqrt{N}} \left(\sum_{q=1}^{N} |u^{q,i,q}|^2\right)^{1/2} \leq \frac{C}{N} + C \sum_{q=1}^{N} |u^{q,i,q}|^2,
        \\
        &\Big| \sum_{l=1}^{N} D_{x^i} \hat{H}^{N,i,l} u^{l,i,l} \Big| \leq \frac{C}{\sqrt{N}} \left(\sum_{l=1}^{N} |u^{l,i,l}|^2\right)^{1/2} \leq \frac{C}{N} {+} \sum_l |u^{l,i,l}|^2, 
        \\
        &\Big| \sum_{l=1}^{N} D_{p^l x^i} \hat{H}^{N,i} u^{l,i,l} \Big| \leq C \left( \sum_{l=1}^{N} |u^{l,i,l}|^2 \right)^{1/2} \leq C + C \sum_{l=1}^{N} |u^{l,i,l}|^2, 
    \end{align*}
    which combined with Propositions \ref{prop.nash.preliminaries} and \ref{prop.uiji} give $\|\cT^{i,i,i,3}\|_{L^1(\scrL^N)} \leq C$. Finally, it is {straightforward} to check that $\|\cT^{i,i,i,4}\|_{\infty} \leq C$, 
  and so we deduce that \eqref{iii.suff} indeed holds. Now we proceed as in the proofs of Propositions \ref{prop.uij} and \ref{prop.uiji}. We fix $T_0 \in [0,T)$, then $\eps \in (0,T-T_0)$. Next, we choose $t_0 \in [T_0,T_0 + \eps)$, and $\bx_0 \in (\R^d)^N$. For $i,j = 1,\dots,N$, we use the notation 
\begin{align*}
    \bX = \bX^{t_0,\bx_0}, \quad Y^{i,i,i} = u^{i,i,i}(\cdot,\bX), \quad Z^{i,i,i,j} = \sqrt{2} D_j u^{i,i,i}(\cdot,\bX), \quad Z^{i,i,i,0} = \sqrt{2\sigma_0} D_k u^{i,i,i}(\cdot,\bX).
\end{align*}
By It\^o's formula, we have 
\begin{align*}
    dY_t^{i,i,i} = \Big( \cB^{i}_t Z_t^{i,i,i,i}  + \cC_t^i \Big) dt + \sum_{k = 0}^N Z_t^{i,j,i,k} dW_t^k, 
\end{align*}
where 
\begin{align*}
    \cB_t^i = \frac{1}{\sqrt{2}} \cB^i(t,\bX_t), \quad \cC_t^i = \cC^i(t,\bX_t). 
\end{align*}
We deduce that 
\begin{align*}
    &\E\bigg[ |Y_{t_0}^{i,i,i}|^2 + \int_{t_0}^{T_0 + \eps} {\sum_{k=0}^{N}} |Z_t^{i,i,i,k}|^2dt \bigg] \leq \E\bigg[ |Y_{T_0 + \eps}^{i,i,i}|^2 + \int_{t_0}^{T_0 + \eps} |Y_t^{i,i,i}| \Big| \cB_t^i Z_t^{i,i,i,i} + \cC_t^i \Big|dt \bigg] 
    \\
    &\qquad \leq \norm{u^{i,i,i}}_{\infty, T_0 + \eps} + C \E\bigg[ \int_{t_0}^{T_0 + \eps} |Y_t^{i,i,i}|^2 dt \bigg] + C \E\bigg[ \int_{t_0}^{T_0 + \eps} |Y_t^{i,i,i}| |\cC_t^i| dt \bigg] + \frac{1}{2} \E\bigg[ \int_{t_0}^{T_0 + \eps} |Z_t^{i,i,i,i}|^2 dt \bigg]
    \\
    &\qquad \leq \norm{u^{i,i,i}}_{\infty, T_0 + \eps} + C \eps \norm{u^{i,i,i}}^2_{\infty, T_0, T_0 + \eps} + C \norm{u^{i,i,i}}_{\infty, T_0, T_0 + \eps} + \frac{1}{2} \E\bigg[ \int_{t_0}^{T_0 + \eps} |Z_t^{i,i,i,i}|^2 dt \bigg], 
\end{align*}
and so absorbing the last term on the right-hand side we get 
\begin{align} \label{yiiibound}
    &\E\bigg[ |Y_{t_0}^{i,i,i}|^2 + \int_{t_0}^{T_0 + \eps} {\sum_{k=0}^{N}} |Z_t^{i,i,i,k}|^2 dt \bigg]
 \leq \norm{u^{i,i,i}}_{\infty, T_0 + \eps} + C \eps \norm{u^{i,i,i}}^2_{\infty, T_0, T_0 + \eps} + C \norm{u^{i,i,i}}_{\infty, T_0, T_0 + \eps}
 \\
 &\qquad \qquad \leq \norm{u^{i,i,i}}_{\infty, T_0 + \eps} + \left(C\eps + \frac{1}{2} \right) \norm{u^{i,i,i}}^2_{\infty, T_0, T_0 + \eps} + C.
\end{align}
Taking a supremum over $t_0$ and $\bx_0$, we find that 
\begin{align*}
    \norm{u^{i,i,i}}^2_{\infty, T_0, T_0 + \eps}  \leq \norm{u^{i,i,i}}_{\infty, T_0 + \eps} + \left(C\eps + \frac{1}{2} \right) \norm{u^{i,i,i}}^2_{\infty, T_0, T_0 + \eps} + C, 
\end{align*}
and the result of the proof follows from choosing $\eps$ small enough and iterating, exactly as in the proofs of Propositions \ref{prop.uij} and \ref{prop.uiji}. Having established the desired uniform bounds on $\|u^{i,i,i}\|_{\infty}$, we obtain the bound on $\norm{\Big(\sum_{k = 1}^N |D_k u^{i,i,i}|^2 \Big)^{1/2}}_{L^2(\scrL^N)}$ from \eqref{yiiibound} by recalling the definition of the definition of $Z^{i,i,i,j}$.

For the estimates on $u^{i,j,j}$ and $D_k u^{i,j,j}$, we write the equation for $u^{i,j,j}$ as 
    \begin{align*}
        - \partial_t u^{i,j,j} - \scrL^N u^{i,j,j} + \cC^{i,j} = 0, \quad  \cC^{i,j} = \sum_{q = 1}^4 \cT^{i,j,j,q},
    \end{align*}
  {where the $\cT^{i,j,j,q}$ terms have been defined in \eqref{nash.seconderiv}.}  
    We will now show that 
    \begin{align} \label{ijj.suff}
        \norm{\cT^{i,j,j,q}}_{L^1(\scrL^N)} \leq C/N, \quad q = 1\dots,N.
    \end{align}
    We start with $\cT^{i,j,j,1}$, and estimate
   \begin{align*}
    &\Big|\sum_{l \neq i} \sum_{q = 1}^N D_j u^{q,j,q} D_{p^q} a^{N,l} u^{i,l}\Big| \leq \frac{C}{N} \left|D_j u^{j,j,j}\right| \sum_{l \neq i} |D_{p^j} a^{N,l}| {+\frac CN \sum_{q \neq j} |D_j u^{q,j,q}|\sum_{l \neq i} |D_{p^q} a^{N,l}|}\\
     &\qquad \qquad \qquad \leq \frac{C}{N} |D_j u^{j,j,j}| {+ \frac{C}{\sqrt{N}}\left(\sum_{q \neq j} |D_j u^{q,j,q}|^{2}\right)^{1/2}}\\
     &\qquad \qquad \qquad \leq \frac{C}{N} |D_j u^{j,j,j}| {+ C\sum_{q \neq j} |D_j u^{q,j,q}|^{2} + \frac{C}{N}}, 
    \\
    &\Big|\sum_{l=1}^{N} \hat{H}^{N,i,l} D_j u^{l,j,l}\Big| \leq \frac{C}{N} |D_j u^{j,j,j}| + \frac{C}{\sqrt{N}} \left( \sum_{l \neq j} |D_j u^{l,j,l}|^2 \right)^{1/2} 
    \\
    &\qquad \qquad \qquad \leq \frac{C}{N}(1 + |D_j u^{j,j,j}|) + C \sum_{l \neq j} |D_j u^{l,j,l}|^2
    \end{align*}
    Using the bounds on $(D_k u^{i,i,i})$ obtained above and Proposition \ref{prop.uiji}, {and Corollary \ref{cor:kiji} and \eqref{estim:CN}} we deduce that 
    \begin{align*}
        \norm{\cT^{i,j,j,1}}_{L^1(\scrL^N)} \leq C/N.
    \end{align*}
    We now turn our attention to $\cT^{i,j,j,2}$, and note that
    \begin{align*}
        &\Big| \sum_{l = 1}^N \sum_{q = 1}^N D_{p^q} a^{N,l} u^{q,j,q} u^{i,j,l}\Big| \leq |u^{j,j,j}| \sum_{l=1}^{N} |D_{p^j} a^{N,l}| |u^{i,j,l}| + \sum_{l=1}^{N} \sum_{q \neq j} |D_{p^q} a^{N,l}| |u^{q,j,q}| |u^{i,j,l}|
    \\
    &\qquad \qquad \qquad \leq C |u^{i,j,j}| + \frac{C}{\sqrt{N}} \left( \sum_{l=1}^{N} |u^{i,j,l}|^2 \right)^{1/2} + C\left( \sum_{l=1}^{N} |u^{i,j,l}|^2 \right)^{1/2} \left( \sum_{q \neq j} |u^{q,j,q}|^2 \right)^{1/2}
    \\
    &\qquad \qquad \qquad \leq C |u^{i,j,j}| + \frac{C}{N} + C\sum_{l=1}^{N} |u^{i,j,l}|^2 {+C\sum_{q \neq j} |u^{q,j,q}|^2}, 
    \\
    &\Big| \sum_{l = 1}^N D_{x^j} a^{N,l} u^{i,j,l}  \Big| \leq C |u^{i,j,j}| + \frac{C}{\sqrt{N}} \left( \sum_{l=1}^{N} |u^{i,j,l}|^2 \right)^{1/2} \leq C |u^{i,j,j}| + \frac{C}{N} + C \sum_{l=1}^{N} |u^{i,j,l}|^2,
    \end{align*}
    and likewise 
    \begin{align*}
    &\Big| \sum_{l \neq i} \sum_{q=1}^{N} D_{p^q} a^{N,l} u^{q,j,q} u^{i,l,j} \Big| \leq C |u^{i,j,j}| + \frac{C}{N} {+C\sum_{q\neq j}|u^{q,j,q}|^{2}} + C\sum_{l=1}^{N} |u^{i,j,l}|^2, 
    \\
    &\Big|\sum_{l \neq i} D_{x^j} a^{N,l} u^{i,j,l} \Big| \leq C |u^{i,j,j}| + \frac{C}{\sqrt{N}} \left(\sum_{l=1}^{N} |u^{i,j,l}|^2\right)^{1/2} \leq {C |u^{i,j,j}|+} C/N + C \sum_{l=1}^{N} |u^{i,j,l}|^2, 
    \end{align*}
    and so using Proposition \ref{prop.uij} {and Corollary \ref{cor:kij}}, we deduce that
    \begin{align*}
        \norm{\cT^{i,j,j,2}}_{L^1(\scrL^N)} \leq C/N. 
    \end{align*}
    For $\cT^{i,j,j,3}$, we estimate
    \begin{align*}
    &\Big|\sum_{l \neq i} \sum_{q,n = 1}^N D_{p^np^q} a^{N,l} u^{q,j,q} u^{n,j,n} u^{i,l} \Big| \leq \frac{C}{N} \sum_{q,n = 1}^N \left(\sum_{l=1}^{N} |D_{p^n p^q} a^{N,l}|\right) |u^{q,j,q}| |u^{n,j,n}| 
    \\
    &\qquad \qquad \qquad \leq \frac{C}{N}{\sum_{q,n = 1}^N} \left(1/N + 1_{n = q} \right) |u^{q,j,q}| |u^{n,j,n}| \leq \frac{C}{N} \sum_{q=1}^{N} |u^{q,j,q}|^2 \leq C/N,
    \\
    &\Big|\sum_{l ,q=1}^{N} D_{p^q} \hat{H}^{N,i,l} u^{q,j,q} u^{l,j,l}\Big| \leq |u^{j,j,j}| \sum_{l=1}^{N} |D_{p^j} \hat{H}^{N,i,l}| |u^{l,j,l}| + |u^{j,j,j}| \sum_{q=1}^{N} |D_{p^q} \hat{H}^{N,i,j}| |u^{q,j,q}|
    \\
    &\qquad \qquad \qquad \qquad + \sum_{q \neq j} \sum_{l \neq j} |D_{p^q} \hat{H}^{N,i,l}| |u^{q,j,q}| |u^{l,j,l}|
    \\
    &\qquad \qquad \qquad \leq \frac{C}{N} {+\sum_{q \neq j} |u^{q,j,q}|^2 + \sum_{l \neq j} |u^{l,j,l}|^2 } + C\left( \sum_{q \neq j} |u^{q,j,q}|^2 \right)^{1/2} \left( \sum_{l \neq j} |u^{l,j,l}|^2 \right)^{1/2} \leq C/N, 
    \\
    &\Big|\sum_{l \neq i} \sum_{q = 1}^N D_{x^j p^q} a^{N,l} u^{q,j,q} u^{i,l}\Big| \leq |u^{j,j,j}| \sum_{l \neq i} |D_{x^j p^{j}} a^{N,l}| |u^{i,l}| + \sum_{l \neq i} \sum_{q \neq j} |D_{x^jp^q} a^{N,l}| |u^{q,j,q}| |u^{i,l}|
    \\
    & \qquad \qquad \qquad \leq \frac{C}{N} + \frac{C}{\sqrt{N}} \left(\sum_{q \neq j} |u^{q,j,q}|^2\right)^{1/2} \leq C/N + \sum_{q \neq j} |u^{q,j,q}|^2 \leq C/N, 
    \\
    &\Big|\sum_{l \neq i} \sum_{q = 1}^N D_{p^q x^j } a^{N,l} u^{q,j,q} u^{i,l}\Big| \leq \frac{C}{N} |u^{j,j,j}| + \sum_{l \neq i} \sum_{q \neq j} |D_{p^q x^j} a^{N,l}| |u^{q,j,q}| |u^{i,l}| \leq \frac{C}{N}, 
    \\
    &\Big| \sum_{l=1}^{N} D_{x^{j}} \hat{H}^{N,i,l} u^{l,j,l} \Big| \leq \frac{C}{N} |u^{j,j,j}| + \sum_{{l \neq j}} |D_{x^{j}} \hat{H}^{N,i,l}| |u^{l,j,l}| \leq C/N, 
    \\
    &\Big|\sum_{l = 1}^N D_{p^l x^j} \hat{H}^{N,i} u^{l,j,l} \Big| \leq \frac{C}{N} |u^{j,j,j}| + \sum_{l \neq {j}} |D_{p^l x^j} \hat{H}^{N,i}| |u^{l,{j},l}| \leq C/N, 
\end{align*}
from which we deduce that $\|\cT^{i,j,j,3}\|_{L^1(\scrL^N)} \leq C/N$. Finally, it is {straightforward} to check that $\norm{\cT^{i,j,j,4}}_{\infty} \leq C/N$, and so we have 
\begin{align} \label{fijjbound}
    \|\cC^{i,j}\|_{L^1(\scrL^N)} \leq C/N.
\end{align} The remainder of the proof is the same as the bounds above on $u^{i,i,i}$ and $(D_k u^{i,i,i})_{k = 1,\dots,N}$, we just use the bound \eqref{fijjbound} in place of \eqref{coefficients.iii}.
\end{proof}

\begin{proposition} \label{prop.ijk}
   Let Assumptions \ref{assump.regularity.disp} and \ref{assump.disp} hold. Then there is a constant $C>0$ such that for all $N$ large enough and all $i,j,k$ distinct, we have
   \begin{align*}
      \norm{u^{i,j,k}}_{\infty} + \norm{ \Big(\sum_{l = 1}^N |D_l u^{i,j,k}|^2 \Big)^{1/2}}_{L^2(\scrL^N)} \leq C/N^2.
   \end{align*}
\end{proposition}

\begin{proof}
   Using the bounds on (the derivatives of) $\ba^N$ and $\hat{H}^{N,i}$ from Proposition \ref{prop.aderivscaling} and Lemma \ref{lem.hatH}, as well as the bounds on $u^{i,j,i}, u^{i,i,i}$ and $u^{i,j,j}$ from Propositions \ref{prop.uiji} and \ref{prop.iii}, one can verify the bounds
    \begin{align*}
        &|\cT^{i,j,k,1}| \leq \frac{C}{N} \left(D_k u^{j,j,j}| + |D_k u^{k,j,k}|\right) + \frac{C}{\sqrt{N}} \left( \sum_{q \neq j,k} |D_k u^{q,j,q}|^2 \right)^{1/2} {+ \frac{C}{\sqrt{N}} \left( \sum_{l=1}^{N} |D_k u^{l,j,l}|^2 \right)^{1/2}}, 
        \\
        &|\cT^{i,j,k,2}| \leq C |u^{i,j,k}| + \frac{C}{N} \sum_{l \neq i,j,{k}} |u^{i,j,l}| + \frac{C}{N} \sum_{l \neq i,{j},k} |u^{i,l,k}| + \frac{C}{N^2}, 
        \\
        &|\cT^{i,j,k,3}| \leq C/N^2, \qquad |\cT^{i,j,k,4}| \leq C/N^2,
    \end{align*}
    for $i,j,k$ distinct, where we use the fact that $u^{i,l_{1},l_{2}}=u^{i,l_{2},l_{1}}$. We now fix $T_0 \in [0,T)$, $\eps \in (0,T-T_0]$, and then $(t_0,\bx_0) \in [T_0, T_0 + \eps] \times (\R^d)^N$. For $i,j,k,l = 1,\dots,N$, set 
\begin{align*}
    \bX = \bX^{t_0,\bx_0}, \quad Y_t^{i,j,k} = u^{i,j,k}(t,\bX_t), \quad Z_t^{i,j,k,l} = \sqrt{2} D_l u^{i,j,k}(t,\bX_{t}), \quad Z_t^{i,j,k,0} = \sqrt{2 \sigma_0} \sum_{l = 1}^N D_k u^{i,j,k}(t,\bX_t) dW_t^0.
\end{align*}
We have 
\begin{align*}
    dY_t^{i,j,k} = \sum_{q = 1}^{4} \cT_t^{i,j,k,q} dt + \sum_{l = 0}^N Z_t^{i,j,k,l} dW_t^l, 
\end{align*}
so computing $d|Y_t^{i,j,k}|^2$, integrating from $t_0$ to $T_0 + \eps$, and taking expectations gives 
\begin{align*}
    \E\bigg[ &|Y_{t_0}^{i,j,k}|^2 + \int_{t_0}^{T_0 + \eps} \sum_{{l=0}}^{N} |Z_t^{i,j,k,l}|^2 dt \bigg] \leq \E\bigg[ |Y_{T_0 + \eps}^{i,j,k}|^2 + \int_{t_0}^T |Y_t^{i,j,k}| \Big| \sum_{q = 1}^{4} \cT_t^{i,j,k,q} \Big| dt \bigg]
    \\
    &\leq \|u^{i,j,k}\|^2_{\infty,T_0 + \eps} + C\norm{u^{i,j,k}}_{\infty, T_0, T_0 + \eps} \E\bigg[ \int_{t_0}^T \Bigg\{ \frac{1}{N} \left(|Z_t^{{j,j,j,k}}| + |Z_t^{k,j,k,k}| \right) +  \left(\frac{1}{N} \sum_{q \neq j,k} |Z_t^{q,{j,q,k}}|^2 \right)^{1/2}
    \\
    &\qquad {+\left(\frac{1}{N} \sum_{q=1}^{N} |Z_t^{q,j,q,k}|^2 \right)^{1/2}} +  |Y_t^{i,j,k}| + \frac{1}{N} \sum_{q \neq i,j,{k}} |Y^{i,j,q}| +  \frac{1}{N} \sum_{q \neq {i,j,k}} |Y_t^{i,q,k}| + \frac{1}{N^2} \Bigg\} dt \bigg]
    \\
    &\leq \|u^{i,j,k}\|_{\infty,T_0 + \eps}^2 + C \norm{u^{i,j,k}}_{\infty, T_0,T_0 + \eps} \bigg( \frac{1}{N^2} + \sqrt{\eps} \E\bigg[ \int_{t_0}^T \frac{1}{N} \sum_{{q=1}}^{N} |Z_t^{q,j,k,q}|^2 \bigg]^{1/2}
    \\
    &\qquad \qquad  + \eps \norm{u^{i,j,k}}_{\infty, T_0,T_0 + \eps}
    + \eps \max_{i',j',k' \text{ distinct}}  \norm{u^{i',j',k'}}_{\infty, T_0,T_0 + \eps}  \bigg), 
\end{align*}
where we used Proposition \ref{prop.uiji} to deduce that 
\begin{align*}
   { \E\bigg[ \int_{t_0}^T \big( |Z_t^{j,j,j,k}| + |Z_t^{k,j,k,k}| \big)dt \bigg]} &=  \E\bigg[ \int_{t_0}^T \big( |Z_t^{j,k,j,j}| + |Z_t^{k,j,k,k}| \big)dt \bigg]\\
   & \leq C \Big( \norm{Z^{j,k,j,j}}_{L^2(\scrL^N)} + \norm{Z^{k,j,k,k}}_{L^2(\scrL^N)} \Big) \leq C/N. 
\end{align*}
Taking a maximum over $i,j,k$ distinct, we get 
\begin{align*}
    \max_{i,j,k \text{ distinct}} \E\bigg[ &|Y_{t_0}^{i,j,k}|^2 + \int_{t_0}^{T_0 + \eps} \sum_{l=0}^{N} |Z_t^{i,j,k,l}|^2 dt \bigg] \leq \max_{i,j,k \text{ distinct}} \|u^{i,j,k}\|^2_{\infty, T_0 + \eps}
    \\
    & + \Big(C\eps + \frac{1}{2} \Big) \max_{i,j,k \text{ distinct}} \norm{u^{i,j,k}}_{\infty,T_0,T_0 + \eps}^2 + \frac{C}{N^4} 
    \\
    &+ \sqrt{\eps} \max_{i,j,k \text{ distinct}} \norm{u^{i,j,k}}_{\infty, T_0,T_0 + \eps} \times \max_{i,j,k \text{ distinct}} \E\bigg[ \int_{t_0}^{T_0 + \eps} \frac{1}{N} \sum_{{l=1}}^{N} |Z_t^{i,j,k,l}|^2 dt \bigg]^{1/2}.
\end{align*}
After an application of Young's inequality, we arrive at 
\begin{align} \label{zijkl}
    \max_{i,j,k \text{ distinct}} \E\bigg[ &|Y_{t_0}^{i,j,k}|^2 + \int_{t_0}^{T_0 + \eps} \sum_{l=0}^{N} |Z_t^{i,j,k,l}|^2 dt \bigg] \leq C\max_{i,j,k \text{ distinct}} \|u^{i,j,k}\|_{\infty,T_0 + \eps}^2
    \nonumber \\
    & + C\eps \max_{i,j,k \text{ distinct}} \norm{u^{i,j,k}}_{\infty, T_0,T_0 + \eps}^2 + C/N^4. 
\end{align}
Taking a supremum over $t_0 \in [T_0, T_0 + \eps)$, $\bx_0 \in (\R^d)^N$, and recalling the definition of $Y^{i,j,k}$, we obtain 
    \begin{align*}
    \max_{i,j,k \text{ distinct}}& \norm{u^{i,j,k}}_{\infty, T_0,T_0 + \eps}^2 \leq C\max_{i,j,k \text{ distinct}} \|u^{i,j,k}\|_{\infty,T_0+ \eps}^2
    \\
    & + C\eps \max_{i,j,k \text{ distinct}} \norm{u^{i,j,k}}_{\infty, T_0,T_0 + \eps}^2 + C/N^4, 
\end{align*}
and so we deduce that there is some $\eps_0 > 0$ such that for all $\eps < \eps_0$, we have 
 \begin{align*}
    \max_{i,j,k \text{ distinct}} \norm{u^{i,j,k}}_{\infty, T_0,T_0 + \eps}^2 \leq \max_{i,j,k \text{ distinct}} \|u^{i,j,k}\|_{\infty,T_0 + \eps} + C/N^4. 
\end{align*}
Iterating this estimate gives the desired bound on $u^{i,j,k}$, and then returning to \eqref{zijkl} gives the corresponding bound on $D_l u^{i,j,k}$. 
\end{proof}

\subsection{Estimates on the third derivatives}

The aim of this subsection is to prove the bounds on third spatial derivatives of the solutions to the Nash system. To do this, we will need to differentiate once again the Nash system. Explicit computation shows that 
\begin{align*} 
    &D_l \big( - \scrL^N u^{i,j,k} \big) = - \scrL^N u^{i,j,k,l} - \sum_{q=1}^{N} \Big( D_{x^l} a^{N,q} + \sum_{n=1}^{N} D_{p^n} a^{N,q} u^{n,l,n} \Big) u^{i,j,k,q}, 
    \\
    &D_l \big(\cT^{i,j,k,1}\big) = - \sum_{q \neq i} \sum_{n=1}^{N} D_{p^n} a^{N,q} u^{i,q} D_n u^{n,j,k,l} + \sum_{q=1}^{N} \hat{H}^{N,i,q} D_q u^{q,j,k,l}
    \\
    &\quad - \sum_{q \neq i} \sum_{n=1}^{N} D_{p^n} a^{N,q} u^{i,q,l} u^{n,j,n,k} - \sum_{q \neq i} \sum_{n=1}^{N} \Big( D_{x^lp^n} a^{N,q} + \sum_{\hat q=1}^{N} D_{p^{\hat q} p^n} a^{N,q} u^{\hat q,l,\hat q} \Big) u^{i,q} u^{n,j,n,k}
    \\
    &\quad + \sum_{q=1}^{N} \Big( D_{x^l} \hat{H}^{N,i,q} + \sum_{n=1}^{N} D_{p^n} \hat{H}^{N,i,q} u^{n,l,n} \Big) u^{q,j,q,k}, 
    \\
    &D_l \big(\cT^{i,j,k,2}\big) = - \sum_{q,n=1}^{N} D_{p^n} a^{N,q} u^{n,k,n} u^{i,j,q,l} - \sum_{q=1}^{N} D_{x^k} a^{N,q} u^{i,j,q,l} 
    \\
    &\quad - \sum_{q \neq i} \sum_{n=1}^{N} D_{p^n} a^{N,q} u^{n,j,n} u^{i,q,k,l} - \sum_{q \neq i} D_{x^j} a^{N,q} u^{i,q,k,l}
    \\
    &\quad - \sum_{q,n=1}^{N} D_{p^n} a^{N,q} u^{i,j,q} u^{n,k,n,l} - \sum_{q,n=1}^{N} \Big( D_{x^l p^n} a^{N,q} + \sum_{\hat q=1}^{N} D_{p^{\hat q} p^n} a^{N,q} u^{\hat q,l,\hat q} \Big) u^{n,k,n} u^{i,j,q}
    \\
    &\quad - \sum_{q=1}^{N} \Big( D_{x^lx^k} a^{N,q} + \sum_{n=1}^{N} D_{p^n x^k} a^{N,q} u^{n,l,n} \Big) u^{i,j,q} - \sum_{q \neq i} \sum_n D_{p^n} a^{N,q} u^{n,j,n,l} u^{i,q,k}
    \\
    &\quad - \sum_{q \neq i} \sum_{n=1}^{N} \Big( D_{x^l p^n} a^{N,q} + \sum_{\hat q=1}^{N} D_{p^{\hat q} p^n} a^{N,q} u^{\hat q,l,\hat q} \Big) u^{n,j,n} u^{i,q,k} - \sum_{q \neq i} \Big( D_{x^l x^j} a^{N,q} + \sum_{n=1} D_{p^n x^j} a^{N,q} u^{n,l,n} \Big) u^{i,k,q},
    \\
    &D_l \big( \cT^{i,j,k,3} \big) = - \sum_{r \neq i} \sum_{n,q=1}^{N} D_{p^n p^q} a^{N,r} u^{n,j,n,l} u^{q,k,q} u^{i,r} - \sum_{r \neq i} \sum_{n,q=1}^{N} D_{p^np^q} a^{N,r} u^{n,j,n} u^{q,k,q,l} u^{i,r}
    \\
    &\quad + \sum_{r,n=1}^{N} D_{p^n} \hat{H}^{N,i,r} u^{n,k,n} u^{r,j,r,l} + \sum_{r,n=1}^{N} D_{p^n} \hat{H}^{N,i,r} u^{n,k,n,l} u^{r,j,r}
    \\
    &\quad - \sum_{r \neq i} \sum_{n=1} D_{x^k p^n} a^{N,r} u^{n,j,n,l} u^{i,r} - \sum_{r \neq i} \sum_{n=1}^{N} D_{p^n x^j} a^{N,r} u^{n,k,n,l} u^{i,r}
    \\
    &\quad + \sum_{r=1}^{N} D_{x^k} \hat{H}^{N,i,r} u^{r,j,r,l} + \sum_{r=1}^{N} D_{p^r x^j} \hat{H}^{N,i} u^{r,k,r,l}
    \\
    &\quad - \sum_{r \neq i} \sum_{n,q=1}^{N} D_{p^np^q} a^{N,r} u^{n,j,n} u^{q,k,q} u^{i,r,l} - \sum_{r \neq i} \sum_{n,q=1}^{N} \Big( D_{x^l p^np^q} a^{N,r} + \sum_{\hat r=1} D_{p^{\hat r} p^n p^q} a^{N,r} u^{\hat r,l,\hat r} \Big) u^{n,j,n} u^{q,k,q} u^{i,r}
    \\
    &\quad + \sum_{r,n=1}^{N} \Big( D_{x^l p^n} \hat{H}^{N,i,r} + \sum_{q=1}^{N} D_{p^qp^n} \hat{H}^{N,i,r} u^{q,l,q} \Big) u^{n,k,n} u^{r,j,r}
    \\
    &\quad - \sum_{r \neq i} \sum_{n=1}^{N} D_{x^k p^n} a^{N,r} u^{n,j,n} u^{i,r,l} - \sum_{r \neq i} \sum_{n=1} \Big( D_{x^lx^k p^n} a^{N,r} + \sum_{q=1}^{N} D_{p^q x^k p^n} a^{N,r} u^{q,l,q} \Big) u^{n,j,n} u^{i,r}
    \\
    &\quad - \sum_{r \neq i} \sum_{n=1}^{N} D_{p^n x^j} a^{N,r} u^{n,k,n} u^{i,r,l} - \sum_{r \neq i} \sum_{n=1}^{N} \Big( D_{x^l p^n x^j} a^{N,r} + \sum_{q=1}^{N} D_{p^q p^n x^j} a^{N,r} u^{q,l,q} \Big) u^{n,k,n} u^{i,r}
    \\
    &\quad + \sum_{r=1}^{N} \Big( D_{x^lx^k} \hat{H}^{N,i,r} + \sum_{n=1}^{N} D_{p^nx^k} \hat{H}^{N,i,r} u^{n,l,n} \Big)u^{r,j,r} + \sum_{r=1}^{N} \Big( D_{x^l p^r x^j} \hat{H}^{N,i} + \sum_{n=1}^{N} D_{p^np^r x^j} \hat{H}^{N,i} u^{n,l,n} \Big) u^{r,k,r}, 
    \\
    &D_l \big( \cT^{i,j,k,4} \big) = - \sum_{q \neq i} D_{x^kx^j} a^{N,q} u^{i,q,l} - \sum_{q \neq i} \Big( D_{x^lx^kx^j} a^{N,q} + \sum_{n=1}^{N} D_{p^n x^kx^j} a^{N,q} u^{n,l,n} \Big) u^{i,q}
    \\
    &\quad + D_{x^lx^kx^j} \hat{H}^{N,i} + \sum_{q=1}^{N} D_{p^q x^kx^j} \hat{H}^{N,i} u^{q,l,q}.
 \end{align*}
Collecting terms appropriately, we find that we can write the equation for $u^{i,j,k,l}$ as
 \begin{align*}
    &- \partial_t u^{i,j,k,l} - \scrL^N u^{i,j,k,l} + \cA^{i,j,k,l} u^{i,j,k,l} +  \sum_{r \neq l} \cA_1^{i,j,k,l;r} u^{i,j,k,r} 
    \\
    &\qquad +\sum_{r \neq k} \cA_2^{i,j,k,l;r} u^{i,j,r,l}  +  \sum_{r \neq j} \cA_3^{i,j,k,l;r} u^{i,r,k,l} + \sum_{r = 1}^N \cB^{i,j,k,l; r} D_r u^{r,j,k,l} + \cC^{i,j,k,{l}} = 0, 
\end{align*}
where the coefficients
\begin{align*}
    &\cA^{i,j,k,l}, \quad (\cA_1^{i,j,k,l; r})_{ r = 1,\dots,N, \, r \neq l}, \quad (\cA_2^{i,j,k,l; r})_{ r = 1,\dots,N, \, r \neq k}, \quad (\cA_3^{i,j,k,l; r})_{ r = 1,\dots,N, \, r \neq j}, 
    \\
    &(\cB^{i,j,k,l; r})_{r = 1,\dots,N}, \quad \cC^{i,j,k,{l}} 
\end{align*}
are functions on $[0,T] \times (\R^d)^N$ defined explicitly by 
\begin{align*}
    &\cA^{i,j,k,l} = -D_{x^l} a^{N,l} -\sum_{n=1}^N D_{p^n} a^{N,l} u^{n,l,n} - \sum_{n=1}^N D_{p^n} a^{N,k} u^{n,k,n} - D_{x^k} a^{N,k} 
    \\
    &\quad  - 1_{i \neq j} \sum_{n=1}^N D_{p^n} a^{N,j} u^{n,j,n} - 1_{i \neq j} D_{x^j} a^{N,j}, 
    \\
    &\cA^{i,j,k,l;r}_1 =  - \big( D_{x^l} a^{N,r} + \sum_{n=1}^N D_{p^n} a^{N,r} u^{n,l,n} \big)
    \\
    &\cA^{i,j,k,l;r}_2 = - \big( D_{x^k} a^{N,r} + \sum_{n=1}^N D_{p^n} a^{N,r} u^{n,k,n} \big), 
    \\
    &\cA^{i,j,k,l;r}_3 = - 1_{r \neq i} \big(D_{x^j} a^{N,r} + \sum_{n=1}^N D_{p^{n}} a^{N,r} u^{n,j,n}\big), 
    \\
    &\cB^{i,j,k,l; r} = \hat{H}^{N,i,r} - \sum_{n \neq i} D_{p^r} a^{N,n} u^{i,n}, 
    \\
    &\cC^{i,j,k,l} = - \sum_{r \neq i} \sum_{n=1}^N D_{p^n} a^{N,r} u^{i,r,l} u^{n,j,n,k} - \sum_{r \neq i} \sum_{n=1}^N \Big( D_{x^lp^n} a^{N,r} + \sum_{q=1}^N D_{p^q p^n} a^{N,r} u^{q,l,q} \Big) u^{i,r} u^{n,j,n,k}
    \\
    &\quad + \sum_{r=1}^N \Big( D_{x^l} \hat{H}^{N,i,r} + \sum_{n=1} D_{p^n} \hat{H}^{N,i,q} u^{n,l,n} \Big) u^{r,j,r,k}
    \\
    &\quad - \sum_{r,n=1}^N D_{p^n} a^{N,r} u^{i,j,r} u^{n,k,n,l} - \sum_{r,n=1}^N \Big( D_{x^l p^n} a^{N,r} + \sum_{q=1}^N D_{p^q p^n} a^{N,r} u^{q,l,q} \Big) u^{n,k,n} u^{i,j,r}
    \\
    &\quad - \sum_{r=1}^N \Big( D_{x^lx^k} a^{N,r} + \sum_{n=1}^N D_{p^n x^k a^{N,r}} u^{n,l,n} \Big) u^{i,j,r} - \sum_{r \neq i} \sum_{n=1}^N D_{p^n} a^{N,r} u^{n,j,n,l} u^{i,r,k}
    \\
    &\quad - \sum_{r \neq i} \sum_{n=1}^N \Big( D_{x^l p^n} a^{N,r} + \sum_{q=1}^N D_{p^q p^n} a^{N,r} u^{q,l,q} \Big) u^{n,j,n} u^{i,r,k} - \sum_{r \neq i} \Big( D_{x^l x^j} a^{N,r} + \sum_{n=1}^N D_{p^n x^j} a^{N,r} u^{n,l,n} \Big) u^{i,k,r}
    \\
    &\quad - \sum_{r \neq i} \sum_{n,q=1}^N D_{p^n p^q} a^{N,r} u^{n,j,n,l} u^{q,k,q} u^{i,r} - \sum_{r \neq i} \sum_{n,q=1}^N D_{p^np^q} a^{N,r} u^{n,j,n} u^{q,k,q,l} u^{i,r}
    \\
    &\quad + \sum_{r,n=1}^N D_{p^n} \hat{H}^{N,i,r} u^{n,k,n} u^{r,j,r,l} + \sum_{r,n=1}^N D_{p^n} \hat{H}^{N,i,r} u^{n,k,n,l} u^{r,j,r}
    \\
    &\quad - \sum_{r \neq i} \sum_{n=1}^N D_{x^k p^n} a^{N,r} u^{n,j,n,l} u^{i,r} - \sum_{r \neq i} \sum_{n=1} D_{p^n x^j} a^{N,r} u^{n,k,n,l} u^{i,r}
    \\
    &\quad + \sum_{r=1}^N D_{x^k} \hat{H}^{N,i,r} u^{r,j,r,l} + \sum_{r=1}^N D_{p^r x^j} \hat{H}^{N,i} u^{r,k,r,l}
    \\
    &\quad - \sum_{r \neq i} \sum_{n,q=1}^N D_{p^np^q} a^{N,r} u^{n,j,n} u^{q,k,q} u^{i,r,l} - \sum_{r \neq i} \sum_{n,q=1}^N \Big( D_{x^l p^np^q} a^{N,r} + \sum_{\hat r=1}^N D_{p^{\hat r} p^n p^q} a^{N,r} u^{\hat r,l,\hat r} \Big) u^{n,j,n} u^{q,k,q} u^{i,r}
    \\
    &\quad + \sum_{r,n=1}^N \Big( D_{x^l p^n} \hat{H}^{N,i,r} + \sum_{q=1}^N D_{p^qp^n} \hat{H}^{N,i,r} u^{q,l,q} \Big) u^{n,k,n} u^{r,j,r}
    \\
    &\quad - \sum_{r \neq i} \sum_{n=1}^N D_{x^k p^n} a^{N,r} u^{n,j,n} u^{i,r,l} - \sum_{r \neq i} \sum_{n=1}^N \Big( D_{x^lx^k p^n} a^{N,r} + \sum_{q=1}^N D_{p^q x^k p^n} a^r u^{q,l,q} \Big) u^{n,j,n} u^{i,r}
    \\
    &\quad - \sum_{r \neq i} \sum_{n=1}^N D_{p^n x^j} a^{N,r} u^{n,k,n} u^{i,r,l} - \sum_{r \neq i} \sum_{n=1}^N \Big( D_{x^l p^n x^j} a^{N,r} + \sum_{q=1}^N D_{p^q p^n x^j} a^{N,r} u^{q,l,q} \Big) u^{n,k,n} u^{i,r}
    \\
    &\quad + \sum_{r=1}^N \Big( D_{x^lx^k} \hat{H}^{N,i,r} + \sum_{n=1}^N D_{p^nx^k} \hat{H}^{N,i,r} u^{n,l,n} \Big)u^{r,j,r} + \sum_{r=1}^N \Big( D_{x^l p^r x^j} \hat{H}^{N,i} + \sum_{n=1}^N D_{p^np^r x^j} \hat{H}^{N,i} u^{n,l,n} \Big) u^{r,k,r}
    \\
    &\quad - \sum_{r \neq i} D_{x^kx^j} a^{N,r} u^{i,r,l} - \sum_{r \neq i} \Big( D_{x^lx^kx^j} a^{N,r} + \sum_{n=1}^N D_{p^n x^kx^j} a^{N,r} u^{n,l,n} \Big) u^{i,r}
    \\
    &\quad + D_{x^lx^kx^j} \hat{H}^{N,i} + \sum_{r=1}^N D_{p^r x^kx^j} \hat{H}^{N,i} u^{r,l,r}
\end{align*}

The key point about these coefficients is the following lemma, which can be proved using the bounds already obtained in the previous subsection, together with the bounds on the derivatives of $\ba^N$, $\hat{H}^{N,i}$, $\hat{H}^{N,i,j}$. For completeness, we provide its proof in Appendix \ref{app:B}.

\begin{lemma} \label{lem.third.order.coeff}
 Let Assumptions \ref{assump.regularity.disp} and \ref{assump.disp} hold. Then there is a constant $C>0$ such that for all $N\in\mathbb{N}$ large enough, we have 
 \begin{align*}
     \|\cA^{i,j,k,l}\|_{\infty} \leq C, \quad \|\cA^{i,j,k,l;r}_q\|_{\infty} \leq C/N, \quad \|\cB^{i,j,k,l; r}\|_{\infty} \leq C/N, \quad \|\cC^{i,j,k,l}\|_{L^1(\scrL^N)} \leq C \omega_{i,j,k,l},
 \end{align*}
 for all $i,j,k,l,r\in\{1,\dots,N\}$ and for $q\in\{1,2,3\}$.
\end{lemma}

The uniform estimates on the third order derivatives can be formulated as follows.

\begin{proposition} \label{prop.uijkl}
   Let Assumptions \ref{assump.regularity.disp} and \ref{assump.disp} hold. Then there is a constant $C>0$ such that for all $N\in\mathbb{N}$ large enough, and all $i,j,k,l \in \{1,\dots,N\}$, we have 
    \begin{align*}
        \|u^{i,j,k,l} \|_{\infty} \leq C \omega^N_{i,j,k,l}.
    \end{align*}
\end{proposition}

\begin{proof}
Fix $T_0 \in [0,T)$, $\eps \in (0,T - T_0]$, and then $(t_0,\bx_0) \in [{T}_0, T_0 + \eps] \times (\R^d)^N$. Now for $i,j,k,l,r \in \{1,\dots,N\}$, set 
\begin{align*}
    &\bX = \bX^{t_0,\bx_0}, \quad Y_t^{i,j,k,l} = u^{i,j,k,l}(t,\bX_t), 
    \\
    &Z_t^{i,j,k,l,r} = \sqrt{2} D_r u^{i,j,k,l}(t,\bX_t), \quad Z_t^{i,j,k,l,0} = \sqrt{2\sigma_0} \sum_{r = 1}^N D_r u^{i,j,k,l}(t,\bX_t). 
\end{align*}
Then, we have 
\begin{align*}
    dY_t^{i,j,k,l} &= \bigg( \cA^{i,j,k,l}_t Y_t^{i,j,k,l} + \sum_{r \neq l} \cA_{1,t}^{i,j,k,l;r} Y_t^{i,j,k,r} + \sum_{r \neq k} \cA_{2,t}^{i,j,k,l;r} Y_t^{i,j,r,l} + \sum_{r \neq j} \cA_{3,t}^{i,j,k,l;r} Y_t^{i,r,k,l} 
    \\
    &\qquad + \sum_{r = 1}^N  \cB^{i,j,k,l; r}_t Z_t^{r,j,k,l,r} + \cC_t^{i,j,k,l} \bigg)dt + \sum_{r = 0}^N Z_t^{i,j,k,l,r} dW_t^r, 
\end{align*}
where 
\begin{align*}
    \cA_t^{i,j,k,l} = \cA(t,\bX_t), \quad \cA_{q,t}^{i,j,k,l;r} = \cC_q^{i,j,k,l;r}(t,\bX_t), \quad \cB_t^{i,j,k,l; r} = \frac{1}{\sqrt{2}} \cB^{i,j,k,l; r}(t,\bX_t),
\end{align*}
$i,j,k,l,r\in\{1,\dots,N\}, q\in\{1,2,3\}.$ Define 
\begin{align*}
    \phi(t) := \max_{i,j,k,l} \frac{\norm{u^{i,j,k,l}}_{\infty,t}}{\omega^N_{i,j,k,l}}.
\end{align*}
Our goal is to show that $\phi$ is bounded. We compute $d |Y_t^{i,j,k,l}|^2$, integrate from $t_0$ to $T_0 + \eps$, and use Lemma \ref{lem.third.order.coeff} to deduce that 
\begin{align*}
    &\E\bigg[ |Y_{t_0}^{i,j,k,l}|^2 + \int_{t_0}^T \sum_{{r=0}}^N |Z_t^{i,j,k,l,r}|^2 dt \bigg] \leq \E\bigg[ |Y_{T_0 + \eps}^{i,j,k,l}|^2 + {2}\int_{t_0}^{T_0 + \eps} |Y_t^{i,j,k,l}| \bigg|  \cA^{i,j,k,l}_t Y_t^{i,j,k,l}
    \\
    &\qquad +  \sum_{r \neq l} \cA_{1,t}^{i,j,k,l;r} Y_t^{i,j,k,r} +\sum_{r \neq k} \cA_{2,t}^{i,j,k,l;r} Y_t^{i,j,r,l} + \sum_{r \neq j} \cA_{3,t}^{i,j,k,l;r} Y_t^{i,r,k,l} 
    \\
    &\qquad +  \sum_{r = 1}^N \cB^{i,j,k,l; r}_t Z_t^{r,j,k,l,r} + \cC_t^{i,j,k,l} \bigg| dt\bigg]
    \\
    &\leq \phi(T_0 + \eps)^2 (\omega_{i,j,k,l}^N)^2 + C \omega_{i,j,k,l}^N \left(\sup_{T_0 \leq t \leq T_0 + \eps} \phi(t) \right) \E\bigg[ \int_{t_0}^T \bigg( |Y_t^{i,j,k,l}| + \frac{1}{N} \sum_{r \neq i,j,k} |Y_t^{i,j,k,r}| 
    \\
    &\qquad + \frac{1}{N} \sum_{r=1}^N |Y_t^{i,j,r,l}| + \frac{1}{N} \sum_{r=1}^N |Y_t^{i,r,k,l}|  \bigg) dt \bigg] 
    \\
    &\qquad +  C \omega_{i,j,k,l}^N\left(\sup_{T_0 \leq t \leq T_0 + \eps} \phi(t) \right) \sqrt{\eps} \E\bigg[ \int_{t_0}^T \frac{1}{N} \sum_{r=1}^N |Z_t^{r,j,k,l,r}|^2 \bigg]^{1/2} + C\left(\sup_{T_0 \leq t \leq T_0 + \eps} \phi(t) \right)  \big(\omega_{i,j,k,l}^N\big)^2
    \\
    &\leq \phi(T_0 + \eps)^2 (\omega_{i,j,k,l}^N)^2 + C (\omega_{i,j,k,l}^N)^2 \left( \sup_{T_0 \leq t \leq T_0 + \eps} \phi(t) \right)^2 \eps + \frac{1}{2} \left(\sup_{T_0 \leq t \leq T} \phi(t) \right)^2 + C \left(\omega_{i,j,k,l}^N\right)^{2}
    \\
    &\qquad + \frac{1}{8 N} \E\bigg[ \int_{t_0}^T \sum_{r \neq j,k,l} |Z_t^{r,j,k,l,r}|^2 dt \bigg] + \frac{1}{8 N} \E\bigg[ \int_{t_0}^T \Big( |Z_t^{j,j,k,l,j}|^2 + |Z_t^{k,j,k,l,k}|^2 + |Z_t^{{l,j,k,l,l}}|^2 \Big) dt \bigg]
    \\
    &\leq \phi(T_0 + \eps)^2 (\omega_{i,j,k,l}^N)^2 + C (\omega_{i,j,k,l}^N)^2 \big( \sup_{T_0 \leq t \leq T_0 + \eps} \phi(t) \big)^2 \eps + \frac{1}{2} \big(\sup_{T_0 \leq t \leq T} \phi(t) \big)^2 + C \omega_{i,j,k,l}^2
    \\
    &\qquad + \frac{1}{2}  (\omega_{i,j,k,l}^{N})^{2} \max_{i',j',k',l'} \bigg(\big(\omega_{i',j',k',l'}^N\big)^{-2} \E\bigg[\int_{t_0}^T \sum_{r=1}^N |Z_t^{i',j',k',l',m'}|^2 dt \bigg] \bigg). 
\end{align*}
Dividing by $(\omega_{i,j,k,l}^N)^2$ and then taking a supremum in $i,j,k,l$, we obtain 
\begin{align*}
    &\sup_{i,j,k,l} (\omega_{i,j,k,l}^N)^{-2} \E\big[ |Y_{t_0}^{i,j,k,l}|^2 \big]
    \\
    & \quad \leq \phi(T_0 + \eps)^2 + C \left( \sup_{T_0 \leq t \leq T_0 + \eps} \phi(t) \right)^2 \eps + \frac{1}{2} \left(\sup_{T_0 \leq t \leq T} \phi(t) \right)^2 + C.
\end{align*}
Taking a supremum over $t_0 \in [T_0, T_0 + \eps]$ and $\bx_0 \in (\R^d)^N$, we find that 
\begin{align*}
    \left(\sup_{T_0 \leq t \leq T_0 + \eps} \phi(t)\right)^2 \leq \phi(T_0 + \eps)^2 + C \eps \left(\sup_{T_0 \leq t \leq T_0 + \eps} \phi(t)\right)^2 + \frac{1}{2} \left(\sup_{T_0 \leq t \leq T} \phi(t)\right)^2 + C, 
\end{align*}
so there is a constant $\eps_0 > 0$ such that if $\eps < \eps_0$, then 
\begin{align*}
    \left(\sup_{T_0 \leq t \leq T_0 + \eps} \phi(t)\right)^2 \leq C\phi(T_0 + \eps)^2 + C.
\end{align*}
Iterating this inequality gives the desired bound.
\end{proof}

\subsection{Completing the proof of Theorem \ref{thm.uniform.nash}}

We have already obtained estimates on the spatial derivatives of $u^{N,i}$ up to order three. To complete the proof of Theorem \ref{thm.uniform.nash}, we need to understand the growth of $u^{N,i}$ and the time regularity. We proceed with a sequence of lemmas. 

\begin{lemma}
    \label{lem.ani.growth}
    Let Assumptions \ref{assump.regularity.disp} and \ref{assump.disp} hold. Then there is a constant $C$ such that for all $N\in\mathbb{N}$ large enough, we have 
    \begin{align*}
        |a^{N,i}(\bx, \bp)| \leq C\big(1 + |x^i| + |p^i| + M_1(m_{\bx}^N) + M_1(m_{\bp}^N) \big).
    \end{align*}
\end{lemma}

\begin{proof}
    Note that the convergence of $\ba^N$ to $\Phi$ in Proposition \ref{prop.fixedpoint.prelims} guarantees that there is a constant $C>0$ such that for all $N\in\mathbb{N}$ large enough $|a^{N,i}(\bm{0}, \bm{0})| \leq C$ {(this fact is also detailed in \cite[Lemma 3.6]{JacMes})}. Thus the the result follows form the Lipschitz bound in Proposition \ref{prop.aderivscaling}, since 
    \begin{align*}
        |a^{N,i}(\bx,\bp)| &\leq |a^{N,i}(\bm{0}, \bm{0})| + \sum_{j = 1}^N \|D_{x^j} a^{N,i}\|_{\infty} |x^j| + \sum_{j = 1}^N \|D_{p^j} a^{N,i}\|_{\infty} |p^j| 
        \\
        &\leq C \Big( |x^i| + |p^i| + \frac{1}{N} \sum_{j = 1}^N |x^j| + \frac{1}{N} \sum_{j = 1}^N |p^j| \Big)
        \\
        &= C\big(1 + |x^i| + |p^i| + M_1(m_{\bx}^N) + M_1(m_{\bp}^N) \big)
    \end{align*}
\end{proof}

\begin{lemma} \label{lem.uni.growth}
    Let Assumptions \ref{assump.regularity.disp} and \ref{assump.disp} hold. Then there is a constant $C>0$ such that for all $N\in\mathbb{N}$ large enough, we have 
    \begin{align} \label{ui.growth}
       |u^{N,i}(t,\bx)| \leq C\big( 1 + |x^i|^2 + M_1(m_{\bx}^N) \big), 
    \end{align}
    and 
    \begin{align} \label{diui.growth}
        |D_i u^{N,i}(t,\bx)| \leq C \big(1 + |x^i| + M_1(m_{\bx}^N) \big)
    \end{align}
    for each $t \in [0,T]$, $\bx \in (\R^d)^N$.
\end{lemma}

\begin{proof}
    First, note that by It\^o's formula, we have the representation 
    \begin{align*}
        u^{N,i}(t_0,\bx_0) = \E\bigg[ \int_{t_0}^T L\big( X_t^i, \alpha_t^i, m_{\bX_t, \bm{\alpha}_t}^{N,-i}\big) dt + G(X_t^i,m_{\bX_T}^{N,i}) \bigg], 
    \end{align*}
    with 
    \begin{align*}
        \bX_t = \bX_t^{t_0,\bx_0}, \quad \balpha_t = \ba^N\big( \bX_t, \diag \bu^N(t,\bX_t) \big).
    \end{align*}
    In view of the lower bounds on $L$ and $G$ appearing in Assumption \ref{assump.regularity.disp}, we deduce that $u^{N,i}$ is bounded from below, i.e. there is a constant $C$ such that $u^{N,i}(t,\bx) \geq -C$ for all $t,\bx$. To obtain an upper bound, notice that because $L(x,a,\mu) \leq C\big( 1 + |x|^2 + |a|^2 + M_2(\mu)^{1/2} \big)$, we have 
    \begin{align*}
        H(x,a,\mu) \geq \frac{1}{C} |a|^2 - C\Big(1 + |x|^2 + M_2(\mu)^{1/2}\Big). 
    \end{align*}
    Combining this with Lemma \ref{lem.ani.growth} and Propositions \ref{prop.iii} and \ref{prop.uij}, we obtain
    \begin{align*}
        \partial_t u^{N,i} &= - \sum_{j = 1}^N \Delta_j u^{N,i} - \sigma_0 \sum_{j,k = 1}^N \text{tr}\big(D_{jk} u^{N,i}\big) + H\big(x^i, D_i u^{N,i}, \ba^N(\bx, \diag \bu^N) \big) 
        \\
        &\qquad \qquad + \sum_{j \neq i} a^{N,j}(\bx, \diag \bu^N) \cdot D_ju^{N,i}
        \\
        & \geq \frac{1}{C} |D_iu^{N,i}|^2 - C\Big( 1 + |x^i|^2 + M_2^{1/2}(m_{\bx}^N) + M_2^{1/2}( \diag \bu^N) \Big).
    \end{align*}
    As a consequence, for each fixed $(t_0,\bx_0)$, we have 
    \begin{align} \label{timederiv.est}
        u^{N,i}(t_0,\bx_0) &+ \frac{1}{C} \int_{t_0}^T |D_i u^{N,i}(t,\bx_0)|^2 dt \leq C\Big(1 + |x^i|^2 + M_2^{1/2}(m_{\bx}^N) \Big) + C \Big(\int_{t_0}^T \frac{1}{N} \sum_{j = 1}^N |D_j u^{N,j}(t,\bx_0)|^2 dt \Big)^{1/2}.
    \end{align}
    Recalling that $u^{N,i}$ is bounded from below, we can sum over $i$ to get 
    \begin{align*}
       \frac{1}{N} \sum_{j = 1}^N \int_{t_0}^T |D_i u^{N,i}(t,\bx_0)|^2 dt \leq C\Big(1 + M_2(m_{\bx}^N)\Big) + C \Big(\int_{t_0}^T \frac{1}{N} \sum_{j = 1}^N |D_j u^{N,j}(t,\bx_0)|^2 dt \Big)^{1/2}, 
    \end{align*}
    then apply Young's inequality to get
    \begin{align*}
        \frac{1}{N} \sum_{j = 1}^N \int_{t_0}^T |D_i u^{N,i}(t,\bx_0)|^2 dt \leq  C\Big(1 + M_2(m_{\bx}^N)\Big).
    \end{align*}
    Plugging this back into \eqref{timederiv.est} gives the upper bound 
    \begin{align*}
        u^{N,i}(t_0,\bx_0) \leq C\Big( 1 + |x^i|^2 + M_2^{1/2}(m_{\bx}^n) \Big). 
    \end{align*}
    Next, we fix $t_0 \in [0,T]$, $i \in \{1,...,N\}$, and $q \in \{1,...,d\}$. We define $\bm{e}^{i,q} = (0,...,0, e_q,0,...,0)$, where $e_q$ is  the $q$-th standard basis vector in $\R^d$ and is appears in the $i$-th slot of $\bm{e}^{i,q}$. Using the above bound on $|u^{N,i}(t_0,\cdot)|$ and applying the mean value theorem, we find that there exists $h \in [0,1]$ such that 
    \begin{align*}
       | D_{x^i_q} u^{N,i}(t_0,h\bm{e}^{i,q}) | = | u^{N,i}(t_0, \bm{e}^{i,q}) - u^{N,i}(t_0,\bm{0}) | \leq C, 
    \end{align*}
    with $C$ independent of $N$ and $t_0$. Together with the upper bounds on $D_{ij} u^{N,i}$ from the previous subsection, this gives the desired estimate on $D_i u^{N,i}$, and then this in turn implies the desired estimate on $u^{N,i}$. 
\end{proof}

\begin{lemma} \label{lem.nash.timereg}
     Let Assumptions \ref{assump.regularity.disp} and \ref{assump.disp} hold. Then there is a constant $C>0$ such that for all $N\in\mathbb{N}$ large enough, we have 
     \begin{align*}
         |\partial_t u^{N,i}(t,\bx)| \leq C\big(1 + |x^i|^2 + M_1^2(m_{\bx}^N) \big), \quad |\partial_t D_j u^{N,i}| \leq C\big(1 + |x^i| + |x^j| + M_1(m_{\bx}^N) \big).
     \end{align*}
\end{lemma}

\begin{proof}
    Directly from the Nash system \eqref{nashsystem}, we see that 
    \begin{align*}
        |\partial_t u^{N,i}| &= \Bigg|- \sum_{j=1}^{N} \Delta_j u^{N,i} - {\sigma_{0}}\sum_{j,k=1}^{N} \tr\big(D_{kj} u^{N,i}\big) + H\big(x^i, D_i u^{N,i}, m_{\bx, \ba^{N}(\bx,\diag \bu^N)}^{N,-i} \big) 
        \\
        &\qquad  + \sum_{j \neq i} a^{N,j}(\bx, \diag \bu^N) \cdot D_j u^{N,i} \Bigg|.
    \end{align*}
   The bound on $\partial_t u^{N,i}$ thus follows from applying the bounds on $a^{N,i}$, $D_j u^{N,i}$ and $D_{kj} u^{N,i}$ from Propositions \ref{prop.uij}, \ref{prop.uiji}, \ref{prop.iii}, and \ref{prop.ijk}, as well as Lemmas \ref{lem.ani.growth} and \ref{lem.uni.growth}.
   
    Next, for $i,j = 1,\dots,N$, we use \eqref{nashfirstder2} to write 
    \begin{align*}
        |\partial_t u^{i,j}| &= \Bigg| - \sum_{k=1}^{N} \Delta_{k} u^{i,j} - \sigma_{0} \sum_{k,l=1}^{N} \tr\big( D_{kl} u^{i,j} \big) - \sum_{k=1}^{N} D_k u^{{i,j}} a^{N,k}
        \\
        &\quad - \sum_{k \neq i} \sum_{l=1}^{N} D_{jl} u^l D_{p^l} a^{N,k} u^{i,k} - \sum_{k \neq i} D_{x^j} a^{N,k} u^{i,k} + \sum_{k \neq i} D_{jk} u^{N,k}  \hat{H}^{N,i,k} + D_{x^j} \hat{H}^{N,i} \Bigg|.
    \end{align*}
    Again the desired bound follows from plugging in the bounds we have already obtained on the derivatives of $u^{N,i}$ from Propositions \ref{prop.uiji}, \ref{prop.iii}, \ref{prop.ijk}, \ref{prop.uijkl} and the bounds on $\hat{H}^{N,i,j}$ and $\hat{H}^{N,i}$ from Lemma \ref{lem.hatH}.
\end{proof}

\begin{lemma}
    \label{lem.moments}  Let Assumptions \ref{assump.regularity.disp} and \ref{assump.disp} hold. Then there is a constant $C>0$ such that for all $N\in\mathbb{N}$ large enough, and each $(t_0,\bx_0) \in [0,T) \times (\R^d)^N$ and $i \in \{1,\dots,N\}$, we have 
     \begin{align*}
        \sup_{t_0 \leq t \leq T} \E\big[ |X_{t}^{t_0,\bx_0,i}| \big] \leq C\big(1 + |x_0^i| + M_1(m_{\bx_{0}}^N) \big).
     \end{align*}
\end{lemma}

\begin{proof}
   Recall that 
   \begin{align*}
       X^{t_0,\bx_0,i}_t = x_0^i + \int_{t_0}^t a^{N,i}\big(\bX_t^{t_0,\bx_0}, D_iu^{N,i}(t,\bX_t^{t_0,\bx_0}) \big) dt + \sqrt{2} (W_t^i - W_{t_0}^i) + \sqrt{2\sigma_0} (W_t^0 - W_{t_0}^0). 
   \end{align*}
   From Lemmas \ref{lem.ani.growth} and \ref{lem.uni.growth}, there is a constant $C$ independent of $N$, $t_0$, $t$ and $\bx_0$ such that
   \begin{align} \label{xibound}
       \E[ |X_t^{t_0,\bx_0,i}| ] \leq  |x_0^i| + C \E\bigg[ \int_{t_0}^t \Big(1 + |X_s^{t_0,\bx_0,i}| + M_1(\bX^{t_0,\bx_0}_s) \Big) ds \bigg]. 
   \end{align}
   Taking an average over $i$ and applying Gronwall's lemma, we obtain 
   \begin{align*}
       \E\big[ M_1(\bX_s^{t_0,\bx_0}) ] \leq C \big( 1 + M_1(\bx_0) \big), 
   \end{align*}
   and then plugging this into \eqref{xibound} completes the proof.
\end{proof}

\begin{lemma}
    \label{lem.ijk.time}  Let Assumptions \ref{assump.regularity.disp} and \ref{assump.disp} hold. Then there is a constant $C>0$ such that for all $N\in\N$ large enough {and $t_{0}\in[0,T)$, $\bx_{0}\in\R^{dN}$ $h\in[0,T-t_{0}),$} we have 
     \begin{align*}
        |D_{kj} u^{N,i}(t_0,\bx_0) - D_{kj} u^{N,i}(t_0 + h,\bx_0)| \leq C \omega_{ijk}^N\big( 1 + |x_{{0}}^i| + |x_{0}^j| + |x_{0}^k| + M_1(m_{\bx_{0}}^N) \big)\sqrt{h},
     \end{align*}
   for all $i,j,k\in\{1,\dots,N\}.$
\end{lemma}

\begin{proof}
    We estimate 
    \begin{align*}
        |D_{kj} u^{N,i}(t_0+ h, \bx_0) -D_{kj} u^{N,i}(t_0,\bx_0)| &\leq \E\Big[|D_{kj} u^{N,i}(t_0 + h,\bx_0) - D_{kj} u^{N,i}(t_0 + h,\bX_{t_0 + h}^{t_0,\bx_0})|
        \\
        &\quad + |D_{kj} u^{N,i}(t_0 + h,\bX_{t_0 + h}^{t_0,\bx_0}) - D_{kj} u^{N,i}(t_0,\bx_0)|\Big].
    \end{align*}
    Using Proposition \ref{prop.uijkl} and Lemma \ref{lem.moments}, we have 
    \begin{align} \label{uijk.time.1}
        \E\Big[|D_{kj} &u^{N,i}(t_0 + h,\bx_0) - D_{kj} u^{N,i}(t_0 + h,\bX_{t_0+h}^{t_0,\bx_0})| \Big] \leq C \sum_{l=1}^{N} \omega_{i,j,k,l}^N \E\big[ |x_0^l - X_{t_0 + h}^{t_0,\bx_0,l}| \big]
       \nonumber \\
        &\leq C \sum_{l=1}^{N} \omega_{i,j,k,l}^N \bigg( \sqrt{h} + \E\bigg[\int_{t_0}^{t+h} |a^{N,l}(t,\bX_t^{t_0,\bx_0})| dt \bigg] \bigg)
      \nonumber  \\
        &\leq C \sum_{l=1}^{N} \omega_{i,j,k,l}^N \bigg( \sqrt{h} + \E\bigg[\int_{t_0}^{t_0+h} \Big( |X_t^{t_0,\bx_0,l}| + \frac{1}{N} \sum_{r=1}^{N} |X_t^{t_0,\bx_0,r}|\Big) dt \bigg] \bigg)
      \nonumber  \\
        &\leq C \sum_{l=1}^{N} \omega^N_{i,j,k,l} \Big( \sqrt{h} + h \big(1 + |x_0^l| + M_1(m_{\bx_0}^N) \big) \Big)
      \nonumber  \\
        &\leq C \sqrt{h} \omega^N_{i,j,k} \Big(1 + |x_0^i| + |x_0^j| + |x_0^k| + M_1(m_{\bx}^N) \Big).
    \end{align}
    For the other term, we use the equation \eqref{nash.seconderiv} and It\^o's formula to deduce that 
    \begin{align*}
       &\E\Big[ |D_{kj} u^{N,i}(t_0 + h,\bX_{t_0 + h}^{t_0,\bx_0}) - D_{kj} u^{N,i}(t_0,\bx_0)| \Big] 
       \\
       &\quad = \E\bigg[ \bigg| \int_{t_0}^{t_0 + h} \sum_{q = 1}^4 \cT^{i,j,k}(t, \bX_t^{t_0,\bx_0} ) dt +  \sqrt{2} \int_{t_0}^{t_0 + h} \sum_{l = 1}^N D_{lkj} u^{N,i} dW_t^k + \sqrt{2\sigma_0} \int_{t_0}^{t_0 + h} \big(\sum_{l = 1}^N D_{lkj} u^{N,i}\big) dW_t^0 \bigg| \bigg].
    \end{align*}
    Using the estimates obtained on $u^{i,j}$, $u^{i,j,k}$, $u^{i,j,k,l}$, one can check that 
    \begin{align*}
        \norm{\cT^{i,j,k,q}}_{\infty} \leq C \omega_{i,j,k}, \quad q = 1,\dots,4, 
    \end{align*}
    so that 
    \begin{align*}
       \E\bigg[ \bigg| \int_{t_0}^{t_0 + h} \sum_{q = 1}^4 \cT^{i,j,k,{q}}(t, \bX_t^{t_0,\bx_0} ) dt\bigg|\bigg] \leq Ch \omega_{i,j,k}^N. 
    \end{align*}
    Meanwhile, we have 
    \begin{align} \label{uijk.time.2}
      \E\bigg[& \bigg| \sqrt{2} \int_{t_0}^{t_0 + h} \sum_{l = 1}^N D_{lkj} u^{N,i} dW_t^k + \sqrt{2\sigma_0} \int_{t_0}^{t_0 + h} \big(\sum_{l = 1}^N D_{lkj} u^{N,i}\big) dW_t^0 \bigg| \bigg] 
     \nonumber \\
      &{\leq} C \bigg( \E\bigg[ \int_{t_0}^{t_0 + h} \Big( \sum_{l = 1}^N |D_{lkj} u^{N,i}|^2 + \Big| \sum_{l = 1}^N D_{lkj} u^{N,i} \Big|^2 \Big) dt \bigg] \bigg)^{1/2}
      \nonumber \\
      &\leq C \sqrt{h} \omega_{i,j,k}^N, 
    \end{align}
 where in the penultimate inequality we have used the Burkholder--Davis--Gundy inequality and the last bound is coming from Proposition \ref{prop.uijkl}. Combining \eqref{uijk.time.1} and \eqref{uijk.time.2} completes the proof.
\end{proof}

\begin{lemma}
    \label{lem.dtu.timereg}
    Let assumptions \ref{assump.regularity.disp} and \ref{assump.disp} hold. Then there is a constant $C>0$ such that for all $N\in\N$ large enough, we have 
    \begin{align*}
        |\partial_t u^{N,i}(t,\bx) - \partial_t u^{N,i}(s,\bx)| \leq C \Big( |t-s|^{1/2} + \big(1 + |x^i|^2 + M_1^2(m_{\bx}^N) \big)|t-s| \Big),
    \end{align*}
for all $t,s\in[0,T],\bx\in\R^{dN}$ and $i\in\{1,\dots,N\}$.
\end{lemma}

\begin{proof}
    We use the fact that by \eqref{nashsystem}
    \begin{align*}
        \partial_t u^{N,i} &= - \sum_{j=1}^{N} \Delta_j u^{N,i} - \sigma_0 \sum_{j,k=1}^{N} \tr\big(D_{kj} u^{N,i}\big) + H\left(x^i, D_i u^{N,i}, m_{\bx, \ba^{N}(\bx,\diag \bu^N)}^{N,-i} \right) 
        \\
        &\qquad  + \sum_{j \neq i} a^{N,j}(\bx, \diag \bu^N) \cdot D_j u^{N,i},  
    \end{align*}
    together with the estimates on the time regularity we have obtained so far. In particular, by Lemma \ref{lem.ijk.time}, we have 
    \begin{align*}
        \Bigg| \sum_{j=1}^{N} \Delta_j u^{N,i}(t,\bx) - &\sum_{j=1}^{N} \Delta_j u^{N,i}(s,\bx) \Bigg| \leq C \sum_{j=1}^{N} \omega^N_{i,j} \big(1 + |x^i| + |x^j| + M_1(m_{\bx}^N) \big) |t-s|^{1/2} 
        \\
        &\leq C \big(1 + |x^i| + M_1(m_{\bx}^N) \big)|t-s|^{1/2}, 
    \end{align*}
    and similarly
    \begin{align*}
        \Bigg|\sum_{j,k=1}^{N} \tr\big(D_{kj} u^{N,i}(t,\bx)\big) & - \sum_{j,k=1}^{N} \tr\big(D_{kj} u^{N,i}(s,\bx)\big)\Bigg| \leq C \sum_{j,k=1}^{N} \omega^N_{i,j,k} \big(1 + |x^i| + |x^j| + |x^k| + M_1(m_{\bx}^N) \big) |t-s|^{1/2}
        \\
        &\leq C \big(1 + |x^i| + M_1(m_{\bx}^N) \big)|t-s|^{1/2}.
    \end{align*}
    Next, we estimate 
    \begin{align*}
        \Big| & H\big(x^i, D_iu^{N,i}(t,\bx),  m_{\bx, \ba^N(\bx, \diag \bu^N(t,\bx))}^{N,-i} \big) - H\big(x^i, D_iu^{N,i}({s},\bx), m_{\bx, \ba^N(\bx, \diag \bu^N({s},\bx))}^{N,-i} \big) \Big|
        \\
        &\leq C \big(1 + |D_iu^{N,i}(t,\bx)| + |D_iu^{N,i}(s,\bx)| \big) |D_iu^{N,i}(t,\bx)- D_iu^{N,i}(s,\bx)| 
        \\
        &\qquad + \frac{C}{N} \sum_{j \neq i} |D_j u^{N,j}(t,\bx) - D_ju^{N,j}(s,\bx) |
        \\
        &\leq C \big(1 + |x^i| \big) \big(1 + |x^i| + M_1(m_{\bx}^N) \big)|t-s| + \frac{C}{N} \sum_{j \neq i} \big(1 +|x^i| + |x^j| + M_1(m_{\bx}^N) \big)|t-s| 
        \\
        &\leq C\big(1 + |x^i|^2 + M_1^2(m_{\bx}^N) \big)|t-s|.
    \end{align*}
    Finally, we have 
    \begin{align*}
       \Big|\partial_t &\Big( \sum_{j \neq i} a^{N,j}\big(\bx, \diag \bu^N\big) D_j u^{N,i}\Big) \Big| = \sum_{j \neq i} a^{N,j}(\bx, \diag \bu^N) \partial_t D_j u^{N,i} 
       \\
       &\qquad  + \sum_{j \neq i} \sum_k D_{p^k} a^{N,j}(\bx, \diag \bu^N) \partial_t D_k u^{N,k} D_j u^{N,i}
       \\
       &\leq C\sum_{j \neq i} \big(1 + |x^j| + M_1(m_{\bx}^N) \big) \big(1 + |x^i| + |x^j| + M_1(m_{\bx}^N) \big) \omega_{i,j}^N 
       \\
       &\qquad + \frac{C}{N} \sum_{j \neq i} \sum_k \big(1/N + 1_{j = k} \big) \big(1 + |x^k| + M_1(m_{\bx}^N) \big)
       \\
       &\leq C\big(1 + |x^i|^2 + M_1^2(m_{\bx}^N) \big), 
    \end{align*}
    which completes the proof.
\end{proof}

\begin{proof}[Proof of Theorem \ref{thm.uniform.nash}]
    Combine Propositions \ref{prop.uij}, \ref{prop.uiji}, \ref{prop.iii}, \ref{prop.ijk}, \ref{prop.uijkl} with Lemmas \ref{lem.uni.growth},  \ref{lem.nash.timereg}, \ref{lem.ijk.time}, and \ref{lem.dtu.timereg}.
\end{proof}

\section{Existence via compactness}\label{sec:5}

Let us define for each $N \in \N$ the set of $N$-point empirical measures, 
\begin{align*}
    \cP^N \coloneqq \big\{ m_{\bx}^N : \bx \in (\R^d)^N \big\} \subset \cP_2(\R^d).
\end{align*}
Next, we define 
\begin{align*}
    &\cD^{1,N} := \Big\{ \big(t,x^1, m_{\bx}^{N,-1} \big) : (t,\bx) \in [0,T] \times (\R^d)^N \Big\} = [0,T] \times \R^d \times \cP^{N-1},
    \\
    &\cD^{2,N} := \Big\{ \big(t, x^1, m_{\bx}^{N,-1}, x^j \big) :  (t,\bx) \in [0,T] \times (\R^d)^N, \,\, j = 2,\dots,N \Big\} \subset [0,T] \times \R^d \times \cP^{N-1} \times \R^d, 
    \\
    &\cD^{3,N} := \Big\{ \big(t,x^1, m_{\bx}^{N,-1}, x^j,x^k \big) :  (t,\bx) \in [0,T] \times (\R^d)^N\,\, j,k = 2,\dots,N, \,\, j \neq k \Big\} 
    \\
    &\qquad \qquad \qquad \subset [0,T] \times \R^d \times \cP^{N-1} \times \R^d \times \R^d.
\end{align*}
With this notation, $\cD^{1,N}$ is just $[0,T] \times \R^d \times \cP^{N-1}$, $\cD^{2,N}$ is the set of $(t,x,m,y) \in \R^d \times \cP^{N-1}(\R^d) \times \R^d$ such that $ y \in \spt(m)$, and $\cD^{3,N}$ is the set of $(t,x,m,y,z) \in  \R^d \times \cP^{N-1}(\R^d) \times \R^d \times \R^d$ such that $m \in \cP^{N-1}$ and $y$ and $z$ are the locations of two distinct atoms appearing in the empirical measure $m$.
We also find it useful to define 
\begin{align*}
    \cD^1 := [0,T] \times \R^d \times \cP_2(\R^d), \quad \cD^2 := [0,T] \times \R^d \times  \cP_2(\R^d) \times \R^d, \quad \cD^3 := [0,T] \times \R^d \times \cP_2(\R^d) \times \R^d \times \R^d.
\end{align*}
We note that as $N \to \infty$, $\cD^{1,N}$ is becoming a denser and denser subset of $\cD^1$, and likewise for $\cD^{2,N}$ and $\cD^{3,N}$. These latter two cases seem to be a bit more subtle, but there one can rely on the fact that fully supported probability measures are dense in the set of probability measures.
 For each $N \in \N$, we are going to define eight functions 
\begin{align*}
    U^N, \,\, U^N_t, \,\, U^N_x, \,\, U^N_{xx}, \,\, U_m^N, \,\, U_{xm}^N, \,\, U_{ym}^N, \,\, U^N_{mm}
\end{align*}
with
\begin{align*}
    &U^N,U^N_t : \cD^{1,N} \to \R,  
    \qquad  U^N_x : \cD^{1,N} \to \R^d, \qquad  \qquad U_{xx}^N :  \cD^{1,N} \to \R^{d \times d}
    \\
    &U^N_m :  \cD^{2,N} \to \R^d,  
  \quad  \qquad \,\,\, U^N_{xm},  U^N_{ym} :  \cD^{2,N} \to \R^{d \times d},
    \\
   &U^N_{mm} : \cD^{3,N} \to \R^{d \times d}
\end{align*}
by fixing a symmetric admissible solution $\bu^N$ to \eqref{nashsystem2} (which exists for all large enough $N$ by Proposition \ref{prop.existence.N}) and setting, for each $(t,\bx) \in [0,T] \times (\R^d)^N$, and each $j,k = 2,\dots,N$ with $j \neq k$, 
\begin{align} \label{def.un}
    &U^N(t,x^1,m_{\bx}^{N,-1}) := u^{N,1}(t,\bx), \qquad \qquad \qquad \quad   \, \, \,\,
    U_t^N(t,x^1,m_{\bx}^{N,-1}) := \partial_t u^{N,1}(t,\bx),
   \nonumber  \\
    &U^N_x(t,x^1,m_{\bx}^{N,-1}) := D_1u^{N,1}(t,\bx), \qquad \qquad \quad \quad \,\,
    U^N_{xx}(t,x^1,m_{\bx}^{N,-1}) := D_{11}u^{N,1}(t,\bx),
    \nonumber\\
    & U^N_m(t,x^1,m_{\bx}^{N,-1}, x^j) := (N-1) D_j u^{N,1}(t,\bx),
   \quad\ \ 
   U_{xm}^N(t,x^1,m_{\bx}^{N,-1}, x^j) := (N-1) D_{1j} u^{N,1}(t,\bx),
  \nonumber\\
    &U_{ym}^N(t,x^1,m_{\bx}^{N,-1}, x^j) := (N-1) D_{jj} u^{N,1}(t,\bx),
     \quad U^N_{mm}(t,x^1,m_{\bx}^{N,-1},x^j,x^k) := (N-1)^2 D_{jk} u^{N,1}(t,\bx).
\end{align}
\begin{remark}
\begin{itemize}
\item We notice that the construction of the functions above needs only the existence of a symmetric admissible solution $\bu^N$, and in particular uniqueness of solutions of the Nash system \eqref{nashsystem2} is not needed. As discussed in Remark \ref{rmk.uniqueness}, uniqueness of the solutions of the limiting master equation is straightforward even if uniqueness for the $N$-player Nash system is not, and so the limiting master function $U$ we build will not depend on the choice of $\bu^N$. While the uniqueness of solutions to \eqref{nashsystem2} is expected, this remains a nontrivial open question due to the fact that we allow $u^{N,i}$ to grow quadratically in the $x^{i}$ variable (compared to the known results in the literature imposing linear growth at infinity).
\item The choice of the entry $x^{1}$ in the definition of the eight functions above is for convenience only. Due to the symmetry property one could have used the definition $U^N(t,x^i,m_{\bx}^{N,-i}) = u^{N,i}(t,\bx)$ (and similarly for all other functions) for any arbitrary $i\in\{1,\dots,N\}$.
\end{itemize} 
\end{remark}

\begin{lemma} \label{lem.compactness}
    Let Assumptions \ref{assump.regularity.disp} and \ref{assump.disp} hold. Then the functions $U^N, U^N_x, U^N_t,U^N_{xx},  U^N_m, U^N_{xm}, U^N_{ym}, U^N_{mm}$ are well-defined by the formulas appearing in \eqref{def.un}, and there is a constant $C>0$ such that for all $N\in\N$ large enough, we have the following estimates:
    \newline \newline 
    For each $(t,x,m), (t',x',m') \in \cD^{1,N}$, we have
   \begin{itemize} 
  \item  
        $|U^N(t,x,m)| \leq C\big(1 + |x|^2 + M_1(m)\big),$
  \item 
           $|U^N(t,x,m) - U^N(t',x',m')| \begin{array}{l}
           \\
           \\
           \\
             \leq C\big(1 + |x| + |x'| + M_1(m) + M_1(m') \big)|x-x'|  
            \nonumber \\[4pt]
            + C\big(1 + |x|^2 + |x'|^2 + M_1^2(m) + M_1^2(m') \big)|t-s|\\[5pt]
             + C \bd_1(m,m'),
            \end{array}
       $\vspace{5pt}
      \item $|U^N_t(t,x,m)| \leq C\big(1 + |x|^2 + M_1^2(m_{\bx}^N) \big),$
      \item
     $
       |U^N_t(t,x,m) -  U^N_t(t',x',m')| 
       \begin{array}{l} 
       \\
       \\
       \\
       \leq C \big(1 + |x|^2 + M_1^2(m) \big) |t-s|^{1/2} \\[4pt]
        + C\big(1 + |x| + |x'| + M_1(m) + M_1(m')\big) |x-x'|\\[5pt]       
        + C\big(1 + |x| + |x'| + M_2(m) + M_2(m') \big) \bd_2(m,m'),
     \end{array}
     $\vspace{5pt}
     \item $ |U_x^N(t,x,m)| \leq C \big(1 + |x| + M_1(m) \big),$
     \item  $|U^N_x(t,x,m) - U^N_x(t',x',m')| 
     \begin{array}{l}
     \\
     \\
     \leq C\big(1 + |x| + |x'| + M_1(m) + M_1(m')\big)|t - t'| \\[7pt]
       + C\big(|x -x'| + \bd_1(m,m') \big).
      \end{array}
       $
      \end{itemize}
    For each $(t,x,m),(t',x',m') \in \cD^{1,N}$, we have 
    \begin{itemize}
    \item $|U_{xx}^N(t,x,m)| \leq C,$
    \item $ |U_{xx}^N(t,x,m) - U_{xx}^N(t',x',m')| \begin{array}{l}
    \\
    \leq C\big(1 + |x| + |x'| + M_1(m) + M_1(m') \big) |t-t'|^{1/2} 
        \\[3pt]
        + C \big(|x-x'| + \bd_1(m,m') \big).
      \end{array}
    $
    \end{itemize}
    For each $(t,x,m,y)$, $(t',x',m',y') \in \cD^{2,N}$, we have
    \begin{itemize}
    \item $ |U^N_{xm}(t,x,m,y)| + |U^N_{ym}(t,x,m,y)| \leq C, $
    \item $   |U^N_{xm}(t,x,m,y) - U^N_{xm}(t',x',m',y')|
         \begin{array}{l}
         \\
         \leq C\big(1+ |x| + |x'| + |y| + |y'| \big) |t-t'|^{1/2}\\[4pt] 
        + C\big(|x-x'| + \bd_1(m,m') + |y-y'| \big). 
    \end{array}
    $
    \item $|U^N_{ym}(t,x,m,y) - U^N_{ym}(t',x',m',y')|  \begin{array}{l}
         \\
         \leq C\big(1+ |x| + |x'| + |y| + |y'| \big) |t-t'|^{1/2}\\[4pt] 
        + C\big(|x-x'| + \bd_1(m,m') + |y-y'| \big). 
    \end{array}$
    \end{itemize}
    For each $(t,x,m,y,z)$, $(t',x',m',y',z') \in \cD^{3,N}$, we have
    \begin{itemize}
    \item $ |U^N_{mm}(t,x,m,y,z)| \leq C,$
    \vspace{8pt}
    \item $ |U^N_{mm}(t,x,m,y,z) - U^N_{mm}(t',x',m',y',z')|\\[5pt]
\begin{array}{l} 
        \qquad\qquad\qquad\qquad \leq  C\big(1+ |x| + |x'| + |y| + |y'| + |z| + |z'| + M_1(m) + M_1(m') \big) |t-t'|^{1/2} 
        \\[4pt]
        \qquad\qquad\qquad\qquad + C\big(|x-x'| + \bd_1(m,m') + |y-y'| + |z-z'| \big).
  \end{array}
    $
    \end{itemize}
\end{lemma}

\begin{proof}
   The fact that these objects are well-defined comes from the assumed symmetry of the admissible solution $(u^{N,i})_{i = 1,\dots,N}$, and the estimates come from Theorem \ref{thm.uniform.nash}.
\end{proof}

We now use the estimates in Lemma \ref{lem.compactness} to extract a (locally uniformly) convergent subsequence. For this, we introduce the following notation: for $R > 0$, we set 
\begin{align*}
    &\cD^{1}_R = \big\{(t,x,m) \in \cD^{N} : |x| \leq R, \, M_2(m) \leq R \big\}, 
    \\
    &\cD^{2}_R = \big\{(t,x,m,y) \in \cD^{2} : |x| \leq R, \, M_2(m) \leq R, \, |y| \leq R \big\}, 
    \\
    &\cD^{3}_R = \big\{(t,x,m,y,z) \in \cD^{3} : |x| \leq R, \, M_2(m) \leq R, \, |y| \leq R, \, |z| \leq R\big\}. 
\end{align*}
as well as $\cD^{i,N}_R = \cD^{i,N} \cap \cD^{i}_R$, $i = 1,2,3$. 

\begin{corollary}
    \label{cor.subseq}
    There exists a subsequence $(N_k)_{k \in \N}$ and functions
    \begin{align*}
    &U, U_t : \cD^1 \to \R,  \,\,\, \qquad U_x : \cD^1 \to \R^d,
    \qquad \qquad U_{xx} :  \cD^1 \times \cP_2(\R^d) \to \R^{d \times d}
    \\
    &U_m : \cD^2  \to \R^d,  
    \qquad \,\,\,\,
    U_{xm},  U_{ym} :  \cD^2 \to \R^{d \times d},
    \\
   &U_{mm} : \cD^3 \to \R^{d \times d}
   \end{align*}
    such that for each $R > 0$, 
    \begin{align*}
       &\sup_{(t,x,m) \in \cD^{1,N}_R} |U(t,x,m) - U^{N_k}(t,x,m)| \xrightarrow{k \to \infty} 0, 
       \\
       &\sup_{(t,x,m) \in \cD^{1,N}_R} |U_t(t,x,m) - U^{N_k}_t(t,x,m)| \xrightarrow{k \to \infty} 0,
       \\
       &\sup_{(t,x,m) \in \cD^{1,N}_R} |U_x(t,x,m) - U^{N_k}_x(t,x,m)| \xrightarrow{k \to \infty} 0,
       \\
       &\sup_{(t,x,m,y) \in \cD^{2,N}_R} |U_m(t,x,m) - U^{N_k}_m(t,x,m,y)| \xrightarrow{k \to \infty} 0, 
       \\
       &\sup_{(t,x,m,y) \in \cD^{2,N}_R} |U_{xm}(t,x,m) - U^{N_k}_{xm}(t,x,m,y)| \xrightarrow{k \to \infty} 0, 
       \\
       &\sup_{(t,x,m,y) \in \cD^{2,N}_R} |U_{ym}(t,x,m) - U^{N_k}_{ym}(t,x,m,y)| \xrightarrow{k \to \infty} 0, 
       \\
       &\sup_{(t,x,m,y,z) \in \cD^{3,N}_R} |U_{mm}(t,x,m) - U^{N_k}_{mm}(t,x,m,y,z)| \xrightarrow{k \to \infty} 0. 
    \end{align*}
\end{corollary}

\begin{proof}
    For fixed $R > 0$, we can use Lemma \ref{lem.compactness} to find $C_R>0$ large enough that $U^{N,R} : \cD^1_R \to \R$ defined by 
    \begin{align*}
        U^{N,R}(t,x,m) = \inf_{\by \in (\R^d)^{N-1}, \, M_2(m_{\by}^N) \leq R} \Big\{ U^{N}(t,x,m_{\by}^N) + C_R \bd_1(m_{\by}^N, m) \Big\}
    \end{align*}
    is Lipchitz (uniformly in $N$) and satisfies 
    \begin{align*}
        U^{N,R}(t,x,m) = U^{N}(t,x,m), \quad \text{for  } (t,x,m) \in \cD^{1,N}_R.
    \end{align*}
    Applying the Arzel\`a--Ascoli compactness criterion, and noting that $\cD^{1}_R$ is compact when $m$ is endowed with the 1-Wasserstein metric $\bd_1$, we can find for each $R > 0$ a function $U^R : \cD^1_R \to \R$ and a subsequence $(N_k)_{k \in \N}$ such that  
    \begin{align*}
        \sup_{(t,x,m) \in \cD^{1,N}_R} |U^R(t,x,m) - U^{N_k}(t,x,m)| \leq   \sup_{(t,x,m) \in \cD^{1}_R} |U(t,x,m) - U^{N_k, R}(t,x,m)| \xrightarrow{k \to \infty} 0. 
    \end{align*}
    Applying a standard diagonalization argument, we find a function $U : \cD^1 \to \R$ and a single subsequence $(N_k)_{k \in \N}$ such that 
    \begin{align*}
         \sup_{(t,x,m) \in \cD^1_R} |U(t,x,m) - U^{N_k}(t,x,m)| \xrightarrow{k \to \infty} 0
    \end{align*}
    for each $R > 0$. Finally, applying the same reasoning to the functions {$U^N_t$}, $U^N_x$, $U^N_{xx}$, $U^N_{m}$, $U^{N}_{xm}$, $U^N_{ym}$, $U^N_{mm}$ (in the case of the functions $U^{N}_{xm},U^{N}_{ym},U^{N}_{mm}$ one needs to do the Lipschitz extension also in the $y,z$ variables) and passing to a further subsequence each time if necessary, we obtain the result.
\end{proof}

Our next aim is to show that the limit $U$ {is regular enough and} satisfies the master equation \eqref{ME}. For this, {we will show that we have in particular the natural connection between $U$ and the functions $U_t,U_x,U_m$, etc. constructed in Corollary \ref{cor.subseq}, i.e. $\partial_tU = U_t$, $D_x U = U_x$, $D_m U = U_m$ and so on.} We will need the following sequence of technical lemmas.

\begin{lemma} \label{lem.errorterms}
    There is a constant $C>0$ such that for each $N\in\N$ large enough and each $(t,\bx) \in [0,T] \times (\R^d)^N$, we have
    \begin{align*}
        \Big| H\big(x^1, D_1 u^{N,1}, m_{\bx, \ba^N(\bx,\diag \bu^N)}^{N,-1} \big) - \hat{H}\big(x^1,D_1 u^{N,1}, m_{\bx, \diag \bu^N}^{N,-1} \big) \Big| \leq \frac{C}{\sqrt{N}} \Big(|x^1| + M_2^{1/2}(m_{\bx}^N) \Big), 
    \end{align*}
    as well as 
    \begin{align*}
    \Big| D_p H \big(x^j,D_j u^{N,j},  m_{\bx, \ba^N(\bx,\diag \bu^N)}^{N,-j} \big) - D_p \hat{H}(x^j, D_j u^{N,j}, m_{\bx, \diag \bu^N}^{N,-j} \big) \Big| \leq \frac{C}{\sqrt{N}} \Big(|x^j| + M_2^{1/2}(m_{\bx}^N) \Big),
    \end{align*}
 {where $\hat H$ is defined in \eqref{HatHdef}.}   
\end{lemma}

\begin{proof}
    The goal is to reduce to the estimate 
    \begin{align} \label{aN.conv}
        \bd_2\big( m_{\bx, \ba^N(\bx,\bp)}^{N}, \Phi(m_{\bx, \bp}^N) \big) \leq \frac{C}{\sqrt{N}} M_2^{1/2}(m_{\bx,\bp}^N)
    \end{align}
    from Proposition \ref{prop.fixedpoint.prelims}. To this end, we estimate 
    \begin{align*}
       \Big| &H\big(x^1, D_1 u^{N,1}, m_{\bx, \ba^N(\bx,\diag \bu^N)}^{N,-1} \big) - \hat{H}\big(x^1,D_1 u^{N,1}, m_{\bx, \diag \bu^N}^{N,-1} \big) \Big| 
       \\
       &= \Big| H\big(x^1, D_1 u^{N,1}, m_{\bx, \ba^N(\bx,\diag \bu^N)}^{N,-1} \big) - H\big(x^1,D_1 u^{N,1}, \Phi\big(m_{\bx, \diag \bu^N}^{N,-1}\big) \big) \Big|
       \\
       &\leq C \bd_1 \big( m_{\bx, \ba^N(\bx,\diag \bu^N)}^{N,-1}, \Phi\big(m_{\bx, \diag \bu^N}^{N,-1}\big) \big)
       \\
       &\leq C \bd_1 \big( m_{\bx, \ba^N(\bx,\diag \bu^N)}^{N,-1}, m_{\bx, \ba^N(\bx,\diag \bu^N)}^{N} \big) + C \bd_1 \big( m_{\bx, \ba^N(\bx,\diag \bu^N)}^{N}, \Phi\big(m_{\bx, \diag \bu^N}^{N}\big) \big) 
       \\
       &\qquad + C \bd_1 \big( \Phi\big(m_{\bx, \diag \bu^N}^{N,-1}\big), \Phi\big(m_{\bx, \diag \bu^N}^{N}\big) \big)
       \\
       &\leq \frac{C}{N} \Big(|x^1| + |a^{N,1}| + M_1(m^N_{\bx, \ba^N}) + M_1(m^N_{\bx, \diag \bu^N}) \Big) + C \bd_1 \big( m_{\bx, \ba^N(\bx,\diag \bu^N)}^{N}, \Phi\big(m_{\bx, \diag \bu^N}^{N}\big) \big),  
    \end{align*}
    and then applying \eqref{aN.conv}, {Lemma} \ref{lem.ani.growth} {and the fact that the $\bd_1$ metric is dominated by the $\bd_2$ one}, we get 
    \begin{align*}
        \Big| &H\big(x^1, D_1 u^{N,1}, m_{\bx, \ba^N(\bx,\diag \bu^N)}^{N,-1} \big) - \hat{H}\big(x^1,D_1 u^{N,1}, m_{\bx, \diag \bu^N}^{N,-1} \big) \Big| 
        \\
        &\leq \frac{C}{N} \Big(|x^1| + |D_1 u^{N,1}| + M_1(m_{\bx}^N) + M_1(m^N_{\diag \bu^N}) \Big) + \frac{C}{\sqrt{N}} M_2^{1/2}(m_{\bx,\diag \bu^N}^N)^{1/2}, 
    \end{align*}
    and finally using the growth bounds on $D_i u^{N,i}$ from Lemma \ref{lem.uni.growth}, we obtain
    \begin{align*}
        \Big| &H\big(x^1, D_1 u^{N,1}, m_{\bx, \ba^N(\bx,\diag \bu^N)}^{N,-1} \big) - \hat{H}\big(x^1,D_1 u^{N,1}, m_{\bx, \diag \bu^N}^{N,-1} \big) \Big| 
        \\
        &\leq \frac{C}{N} \Big(|x^1| + M_1(m_{\bx}^N)\Big) + \frac{C}{\sqrt{N}} M_2^{1/2}(m_{\bx}^N) \leq \frac{C}{\sqrt{N}} \Big(|x^1| + M_2^{1/2}(m_{\bx}^N) \Big).
    \end{align*}
    The proof for the second bound is almost identical, and so is omitted. 
\end{proof}

\begin{lemma} \label{lem.equation}
    The functions $U, {U_t,} U_x, U_{xx}, U_m, U_{xm}, U_{ym}, U_{mm}$ from Corollary \ref{cor.subseq} satisfy 
    \begin{align} \label{ME.subseq}
    \begin{cases}
        \ds - U_t -
        (1 + \sigma_0) {\rm{tr}}\big(U_{xx}(t,x,m)\big) -  (1 + \sigma_0) \int_{\R^d} {\rm{tr}}\big(U_{ym}(t,x,m,y) \big)m(dy)    \vspace{.2cm}
        \\ \ds
        \qquad - \sigma_0 \int_{\R^d} \int_{\R^d} {\rm{tr}}\big(U_{mm}(t,x,m,y,y') \big) m(dy) m(dy') - 2\sigma_0 \int_{\R^d} {\rm{tr}}\big(U_{xm}(t,x,m,y) \big) m(dy)   \vspace{.2cm}
        \\ \ds \qquad + \hat{H}\Big(x, U_x(t,x,m), \big({\rm{Id}}, U_x(t,\cdot,m) \big)_{\#} m \Big) \vspace{.2cm}
        \\
        \ds \qquad + \int_{\R^d} D_p \hat{H}\Big(y,U_x(t,y,m), \big({\rm{Id}}, U_x(t,\cdot,m) \big)_{\#} m \Big) \cdot  U_m(t,x,m,y) m(dy) = 0 
        \vspace{.2cm} \\ 
        \qquad \qquad \qquad (t,x,m) \in (0,T) \times \R^d \times \cP_2(\R^d), 
        \vspace{.2cm} \\ \ds 
        U(T,x,m) = G(x,m), \quad (x,m) \in \R^d \times \cP_2(\R^d),
    \end{cases}
\end{align}
\end{lemma}

\begin{proof}
    We start by noting that for each $(t,\bx) \in [0,T] \times (\R^d)^N$, we have 
    \begin{align} \label{comp.idio}
        \sum_{j = 1}^N \Delta_j u^{N,1}(t,\bx) &= \tr\big(U^N_{xx}(t,x^1,m_{\bx}^{N,-1}) \big) + \frac{1}{N-1} \sum_{j = 2}^N \tr\big( U_{ym}^N(t,x^1,m_{\bx}^{N,-1}, x^j) \big) 
        \nonumber \\
        &= \tr\big(U^N_{xx}(t,x^1,m_{\bx}^{N,-1})\big) + \int_{\R^d} \tr\big({U}^N_{ym}(t,x^1,m_{\bx}^{N,-1},y) \big) m_{\bx}^{N,-1}(dy).
    \end{align}
    Next, we compute 
    \begin{align*}
        \sum_{j,k = 1}^N& \tr\big(D_{jk} u^{N,1} \big) = \sum_{j = 1}^N \Delta_j u^{N,1}(t,\bx) + 2 \sum_{j = 2}^N \tr\big( D_{1j} u^{N,1}(t,\bx) \big) + \sum_{j,k = 2,\dots,N, \, j \neq k} \tr\big(D_{jk} u^{N,1}(t,\bx) \big)
        \\
        &= \tr\big(U^N_{xx}(t,x^1,m_{\bx}^{N,-1}) \big) + \frac{1}{N-1} \sum_{j = 2}^N \tr\big( U_{ym}^N(t,x^1,m_{\bx}^{N,-1}, x^j) \big) 
        \\
        &\quad + \frac{2}{N-1} \sum_{j = 2}^N \tr\big( U_{xm}^N(t,x^1,m_{\bx}^{N,-1},x^j) \big) + \frac{1}{(N-1)^2} \sum_{j,k = 2,\dots,N, \, j \neq k} \tr\big( U_{mm}^N(t,x^1,m_{\bx}^{N,-1},x^j,x^k \big)
        \\
        &= \tr\big(U^N_{xx}(t,x^1,m_{\bx}^{N,-1})\big) + \int_{\R^d} \tr\big(U^N_{ym}(t,x^1,m_{\bx}^{N,-1},y) \big) m_{\bx}^{N,-1}(dy)
        \\
        &\quad + 2 \int_{\R^d} \tr\big( U_{ym}^N(t,x^1,m_{\bx}^{N,-1},y) \big) m_{\bx}^{N,-1}(dy) + \int_{\R^d} \int_{\R^d} \tr\big( U_{mm}^N(t,x^1,m_{\bx}^{N,-1},y,z) \big) (m_{\bx}^{N,-1})^{\otimes 2}(dy,dz) 
        \\
        &\quad - \frac{1}{{(N-1)^2}} \sum_{j=2}^N \tr\big( U_{mm}^N(t,x^1,m_{\bx}^{N,-1},x^j,x^j) \big).
    \end{align*}
    Since $U^N_{mm}$ is bounded uniformly in $N$, we deduce that there is a constant $C$ such that 
    \begin{align} \label{comp.common}
        &\bigg| \sum_{j,k = 1}^N \tr\big(D_{jk} u^{N,1} \big) - \bigg( \tr\big(U^N_{xx}(t,x^1,m_{\bx}^{N,-1})\big) + \int_{\R^d} \tr\big(U^N_{ym}(t,x^1,m_{\bx}^{N,-1},y) \big) m_{\bx}^{N,-1}(dy)
        \nonumber \\
        &\quad + 2 \int_{\R^d} \tr\big( U_{ym}^N(t,x^1,m_{\bx}^{N,-1},y) \big) m_{\bx}^{N,-1}(dy) + \int_{\R^d} \int_{\R^d} \tr\big( U_{mm}^N(t,x^1,m_{\bx}^{N,-1},y,z) \big) (m_{\bx}^{N,-1})^{\otimes 2}(dy,dz) \bigg) \bigg| \vspace{.5cm}
       \nonumber \\
        &\qquad \leq C/N.
    \end{align}
    Next, we note that 
    \begin{align*}
        &m_{\bx, \diag \bu^N}^{N,-1} =\frac{1}{N-1} \sum_{j = 2}^N \delta_{(x^j,D_ju^{N,j})} = \frac{1}{N-1} \sum_{j = 2}^N \delta_{(x^j,U_x^N(t,x^j,m_{\bx}^{N,-j}))}, 
        \\
        &\big( \text{Id}, U_x^N(t,\cdot,m_{\bx}^{N,-1})\big)_{\#} m_{\bx}^{N,-1}  = \frac{1}{N-1} \sum_{j = 2}^N \delta_{(x^j,U_x^N(t,x^j,m_{\bx}^{N,-1}))}, 
    \end{align*}
    so that 
    \begin{align} \label{empiricalbound}
        \bd_1\Big(m_{\bx, \diag \bu^N}^{N,-1},& \big( \text{Id}, U_x(t,\cdot,m_{\bx}^{N,-1})\big)_{\#} m_{\bx}^{N,-1} \Big) \leq \frac{1}{N-1} \sum_{j = 2}^N |U_x^N(t,x^j, m_{\bx}^{N,-1}) {-U_x^N(t,x^j, m_{\bx}^{N,-j})}| 
        \nonumber \\
        &\leq \frac{C}{N} \sum_{j = 2}^N \bd_1\big( m_{\bx}^{N,-1}, m_{\bx}^{N,-j} \big) \leq \frac{C}{N} \big(|x^1| + M_1(m_{\bx}^N) \big).
    \end{align}
    It follows that 
    \begin{align} \label{comp.hamiltonian}
        &\Big| \hat{H}\big(x^1,D_1u^{N,1}, m_{\bx, \diag \bu^N}^{N,-1} \big) - \hat{H}\Big(x^1,U^N_x(t,x^1,m_{\bx}^{N,-1}), \big( \text{Id}, U_x^N(t,\cdot,m_{\bx}^{N,-1})\big)_{\#} m_{\bx}^{N,-1} \Big) \Big| 
       \nonumber  \\
        &
        \qquad \leq \frac{C}{N} \big(|x^1| + M_1(m_{\bx}^N) \big), 
    \end{align}
    and thus combining \eqref{comp.hamiltonian} with Lemma \ref{lem.errorterms}, we find that
    \begin{align*}
        \Big|&H\big(x^1, D_1u^{N,1},m_{\bx, \ba^N(\bx, \diag \bu^N)}^{N,-1}\big)  - \hat{H}\Big(x^1,U^N_x(t,x^1,m_{\bx}^{N,-1}), \big( \text{Id}, U_x^N(t,\cdot,m_{\bx}^{N,-1})\big)_{\#} m_{\bx}^{N,-1} \Big)\Big|
        \\
        &\leq \Big|H\big(x^1, D_1u^{N,1},m_{\bx, \ba^N(\bx, \diag \bu^N)}^{N,-1}\big) - \hat{H}\big(x^1,D_1u^{N,1}, m_{\bx, \diag \bu^N}^{N,-1} \big)\Big|
        \\
        &\quad + \Big| \hat{H}\big(x^1,D_1u^{N,1}, m_{\bx, \diag \bu^N}^{N,-1} \big) - \hat{H}\Big(x^1,U^N_x(t,x^1,m_{\bx}^{N,-1}), \big( \text{Id}, U_x^N(t,\cdot,m_{\bx}^{N,-1})\big)_{\#} m_{\bx}^{N,-1} \Big)\Big|
        \\
        &\leq \frac{C}{\sqrt{N}} \big(|x^1| + M_2^{1/2}(m_{\bx}^N) \big) + \frac{C}{N} \big(|x^1| +M_1(m_{\bx}^N) \big) \leq \frac{C}{\sqrt{N}} \big(|x^1| + M_2^{1/2}(m_{\bx}^N) \big).
    \end{align*}
    By similar reasoning, and using \eqref{empiricalbound} again, we have
    \begin{align*}
    \Big|D_p &\hat{H}(x^j, D_j u^{N,j}, m_{\bx, \diag \bu^N}^{N,-j} \big) - D_p \hat{H}\Big( x^j, U^N_x(t,x^j,m_{\bx}^{N,-j}), \big(\text{Id},U_x^N(t,\cdot,m_{\bx}^{N,-1}) \big)_{\#} m_{\bx}^{N,-1} \Big) \Big|
        \\
        &\leq C |D_j u^{N,j} - U_x^N(t,x^j,m_{\bx}^{N,-1})| + C \bd_1\Big(m_{\bx, \diag \bu^N}^{N,-j}, \big(\text{Id},U_x^N(t,\cdot,m_{\bx}^{N,-1}) \big)_{\#} m_{\bx}^{N,-1} \Big) 
        \\
        &= C |U_x^N(t,x^j,m_{\bx}^{N,-j}) - U_x^N(t,x^j,m_{\bx}^{N,-1})| + C \bd_1\Big(m_{\bx, \diag \bu^N}^{N,-j}, \big(\text{Id},U_x^N(t,\cdot,m_{\bx}^{N,-1}) \big)_{\#} m_{\bx}^{N,-1} \Big) 
        \\
        &\leq \frac{C}{N} \big(|x^1| + |x^j| + M_1(m_{\bx}^N) \big) + C\bd_1\Big(m_{\bx, \diag \bu^N}^{N,-1}, \big( \text{Id}, U_x(t,\cdot,m_{\bx}^{N,-1})\big)_{\#} m_{\bx}^{N,-1} \Big) 
        \\
        &\qquad + \bd_1\Big(m_{\bx, \diag \bu^N}^{N,-1}, m_{\bx, \diag \bu^N}^{N,-j}\Big)
        \\
        &\leq \frac{C}{N} \big(|x^1| + |x^j| + M_1(m_{\bx}^N) \big), 
    \end{align*}
    so that 
    \begin{align} \label{comp.transport}
        &\Big| \int_{\R^d} D_p \hat{H}\Big(y, U_x^N(t,y,m_{\bx}^{N,-1}), \big(\text{Id}, U_x^N(t,\cdot,m_{\bx}^{N,-1}\big)_{\#} m_{\bx}^{N,-1} \Big) \cdot U_m^N(t,x^1,m_{\bx}^{N,-1},y) m_{\bx}^{N,-1}(dy) 
      \nonumber  \\
        &\quad - \sum_{j=2}^N D_pH \big(x^j,D_ju^{N,j},m_{\bx, \ba^N(\bx,\diag \bu^N)}^{N,-j} \big) \cdot D_j u^{N,1} \Big|
      \nonumber  \\
        &= \bigg| \sum_{j = 2}^N D_j u^{N,1} \cdot \bigg( D_p H(x^j, D_j u^{N,j}, m_{\bx, \diag \bu^N}^{N,-j} \big) - D_p \hat{H}\Big( x^j, U^N_x(t,x^j,m_{\bx}^{N,-j}), \big(\text{Id},U_x^N(t,\cdot,m_{\bx}^{N,-1}) \big)_{\#} m_{\bx}^{N,-1} \bigg) \bigg|
     \nonumber   \\
        &\leq \frac{C}{N} \sum_{j = 2}^N \Big| D_p H(x^j, D_j u^{N,j}, m_{\bx, \diag \bu^N}^{N,-j} \big) - D_p \hat{H}\Big( x^j, U^N_x(t,x^j,m_{\bx}^{N,-j}), \big(\text{Id},U_x^N(t,\cdot,m_{\bx}^{N,-1}) \big)_{\#} m_{\bx}^{N,-1}\Big| 
       \nonumber  \\
        &\leq \frac{C}{N^2} \sum_{j = 2}^N \big(|x^1| + |x^j| + M_1(m_{\bx}^N) \big) \leq \frac{C}{N} \big(|x^1| + M_1(m_{\bx}^N) \big).
    \end{align}
    Putting together \eqref{comp.idio}, \eqref{comp.common}, \eqref{comp.hamiltonian} and \eqref{comp.transport}, and using the Nash system \eqref{nashsystem}, we deduce that 
    \begin{align*}
        \ds &\bigg| - U_t^N -
        (1 + \sigma_0) \tr\big(U^N_{xx}(t,x,m)\big) -  (1 + \sigma_0) \int_{\R^d} \tr\big(U^N_{ym}(t,x,m,y) \big)m(dy)    \vspace{.2cm}
        \\ \ds
        &\qquad - \sigma_0 \int_{\R^d} \int_{\R^d} \tr\big(U_{mm}^N(t,x,m,y,y') \big) m(dy) m(dy') - 2\sigma_0 \int_{\R^d} \tr\big(U^N_{xm}(t,x,m,y) \big) m(dy)   \vspace{.2cm}
        \\ \ds &\qquad + \hat{H}\Big(x, U^N_x(t,x,m), \big(\text{Id}, U_x(t,\cdot,m) \big)_{\#} m \Big) \vspace{.2cm}
        \\
        \ds &\qquad + \int_{\R^d} D_p \hat{H}\Big(y,U^N_x(t,y,m), \big(\text{Id}, U^N_x(t,\cdot,m) \big)_{\#} m \Big) \cdot  U^N_m(t,x,m,y) m(dy)\bigg| \leq \frac{C}{\sqrt{N}} \big(|x| + M_2^{1/2}(m) \big)
        \vspace{.2cm}
\end{align*}
for each $(t,x,m) \in \cD^{1,N}$. We emphasize that the inequality above holds only for $(t,x,m) \in \cD^{1,N}$, in particular $m$ is assumed to be an empirical measure in this inequality. Passing to the limit, we get \eqref{ME.subseq}.
\end{proof}

\begin{proposition} \label{prop.consistency}
    The function $U$ from Corollary \ref{cor.subseq} is $C^{1,2}{(\cD^1)}$, and satisfies 
    \begin{align*}
       &\partial_t U = U_t, \quad  D_x U = U_x, \quad D_{xx} U = U_{xx}, \quad D_m U = U_m, 
       \\
       &D_{xm} U = U_{xm}, \quad D_{ym} U = U_{ym}, \quad D_{mm} U = U_{mm}.
    \end{align*}
\end{proposition}

\begin{proof}
    Let us show that $D_m U$ exists and $D_m U = U_m$, the other proofs being similar. First, we lift $U$ to a function $\hat{U} : [0,T] \times \R^d \times L^2 \to \R$, given by 
    \begin{align*}
        \hat{U}(t,x,Y) = U\big(t,x,\cL(Y) \big).
    \end{align*}
    Showing that $D_m U = U_m$ is equivalent {by \cite{GanTud:19}} to showing that 
    \begin{align} \label{lifted.grad}
        \nabla_{L^2} \hat{U}(t,x,Y) = U_m\big(t,x,\cL(Y), Y\big), 
    \end{align}
    and this is what we will now show. We fix an arbitrary $t \in [0,T]$, $x \in \R^d$, $\by = (y^2,\dots,y^N) \in (\R^d)^{N-1}$, and $\bz = (z^2,\dots,z^N) \in (\R^d)^{N-1}$. We let $\Omega^2,\dots,\Omega^N$ be a partition of $\Omega$ into $N-1$ sets of equal probability, and define random variables $Y$ and $Z$ by 
    \begin{align*}
        (X,Y) = (x^i,z^i) \text{ on } \Omega^i. 
    \end{align*}
    Now, we have, {for the subsequence $(N_k)_{k \in \N}$ given in Corollary \ref{cor.subseq},}
    \begin{align*}
        \hat{U}(t,x, Y) & - \hat{U}(t,x,Z) = U\Big(t,x, \frac{1}{N-1} \sum_{i = 2}^N \delta_{y^i} \Big) - U\Big(t,x, \frac{1}{N-1} \sum_{i = 2}^N \delta_{z^i} \Big)
        \\
        &= \lim_{k \to \infty} \bigg( U^{N_k}\Big(t,x, \frac{1}{N-1} \sum_{i = 2}^N \delta_{y^i} \Big) - U^{N_k}\Big(t,x, \frac{1}{N-1} \sum_{i = 2}^N \delta_{z^i} \Big) \bigg)
        \\
        &= \lim_{k \to \infty} \Big( u^{N_k,1}\big(t, (x,\by)\big) -u^{N_k,1}\big(t, (x,\bz)\big) \Big)
        \\
        &= \lim_{k \to \infty} \Big( \int_0^1 \sum_{j = 2}^N D_j u^{N_k,1}\big(t, (x, (s\by + (1-s) \bz) ) \big) \cdot (y^i - z^i) ds \Big) 
        \\
        &= \lim_{k \to \infty} \Big( \int_0^1 \frac{1}{N-1} \sum_{j = 2}^N U^{N_k}_m\big(t, x, m_{s\by + (1-s) \bz}^{N,-1}, sy^j + (1-s) z^j ) \big) \cdot (y^i - z^i) 
        ds\Big)
        \\
        &= \lim_{k \to \infty} \Big( \int_0^1\E\Big[ U_m^{N_k}\big(t,x, \cL(sY + (1-s) Z), sY + (1-s) Z \big) \cdot (Y-Z)  \Big]  ds \Big)
        \\
        &= \int_0^1\E\Big[ U_m\big(t,x, \cL(sY + (1-s) Z), sY + (1-s) Z \big) \cdot (Y-Z) \Big] ds, 
    \end{align*}
    which shows that \eqref{lifted.grad} indeed holds, so that $U_m = D_m U$. 
\end{proof}

\begin{proof}[Proof of Theorem \ref{thm.main.disp}]
    For existence, we combine Lemma \ref{lem.equation} with Proposition \ref{prop.consistency}. For uniqueness, we reduce the question to the uniqueness of Nash equilibria for the corresponding MFGC, as discussed in \cite{JacMes}. We fix a classical solution $U$ to the master equation and $(t_0,x_0,m_0) \in [0,T] \times \R^d \times \cP_2(\R^d)$, and define $X$ by
   \begin{align*}
      \begin{cases}
           dX_t = - D_p \hat{H}\Big(X_t, D_xU(t,X_t, m_t), \big(\text{Id}, D_xU(t,\cdot, m_t)\big)_{\#} m_t \Big) dt + \sqrt{2} dW_t + \sqrt{2\sigma_0} dW_t^0, 
       \\
       X_{t_0} \sim m_0, \quad m_t = \cL^0(X_t), 
      \end{cases}
   \end{align*}
   where $W$ and $W^0$ are independent Brownian motions on some filtered probability space satisfying the usual conditions, $\cL^0(X_t) = \cL(X_t | \cF_t^{W_0})$, $\cF_t^{W^0}$ being the filtration generated by $W^0$, and $X_{t_0}$ is independent of $W$ and $W^0$. Then, It\^o's formula shows that
   \begin{align*}
       \mu_t = \Phi\Big( \big(\text{Id}, D_xU(t, \cdot, m_t) \big)_{\#} m_t \Big)
   \end{align*}
   is a mean field equilibrium in the sense defined e.g. on page 6 of \cite{JacMes}, and that 
   \begin{align*}
       U(t_0,x_0,m_0) = \inf_{\alpha} \E\bigg[ \int_{t_0}^T L\big(X_t^{\alpha}, \alpha_t, \mu_t \big) dt + G(X_T^{x_0}, \mu_T) \bigg], 
   \end{align*}
   where the infimum is taken over all square-integrable, adapted processes $\alpha$ and $X^{\alpha}$ is determined from $\alpha$ by
   \begin{align*}
       dX_t^{\alpha} = \alpha_t dt + \sqrt{2} dW_t + \sqrt{2\sigma_0} dW_t^0, \qquad X_{t_0}^{\alpha} = x^0.
   \end{align*}
   Because this holds for any classical solution $U$, we see that uniqueness of classical solutions follows from uniqueness of mean field equilibria, and under the displacement semi-monotonicity conditions considered here (more precisely, Assumptions \ref{assump.regularity.disp} and \ref{assump.disp}), uniqueness of equilibria follows from e.g. Theorem A.5 in \cite{JacMes}. 
\end{proof}

\section{The Lasry--Lions monotone case}\label{sec:6}

Suppose now that Assumptions \ref{assump.regularity.LL} and \ref{assump.LL} hold. By Lemma \ref{lem.lagrangianLL}, this means that Assumption \ref{assump.fixedpoint} holds, as well. Then Theorem \ref{thm.main.disp} allows us to produce local in time solutions, with the length of the time horizon depending on $G$ {and $L$} only through {their} displacement semi-monotonicity constant. The proof of Theorem \ref{thm.main.LL} thus reduces to the following a-priori estimate.

\begin{proposition}
    \label{prop.ll.apriori}
    Suppose that Assumptions \ref{assump.regularity.LL} and \ref{assump.LL} hold. Then there is a constant $C>0$ with the following property. If $T_0 \in [0,T)$ and $U : [T_0,T] \times \R^d \times \cP_2(\R^d) \to \R$ is a classical solution to \eqref{ME} on the the time interval $[T_0,T]$, then we have 
    \begin{align} \label{disp.monotone.prop}
        \E\Big[ \big(D_x U(t,X,\cL(X)) - D_xU(t,X',\cL(X')) \big)\cdot (X-X') \Big] \geq -C\E\big[|X - X'|^2 \big]
    \end{align}
    for all $t \in [T_0,T]$ and all square-integrable $X,X'$. 
\end{proposition}

Before proving Proposition \ref{prop.ll.apriori}, we explain how it implies Theorem \ref{thm.main.LL}.

\begin{proof}[Proof of Theorem \ref{thm.main.LL}]
    Suppose that for some $T_0 \in [0,T)$, we have a classical solution $U : [T_0,T] \times \R^d \times \cP_2(\R^d) \to \R$. Then Proposition \ref{prop.ll.apriori} shows that there is a constant $C>0$ independent of $T_0$ such that 
     \begin{align*}
        \E\Big[ \big(D_x U(T_0,X,\cL(X)) - D_xU(T_0,X',\cL(X')) \big)\cdot (X-X') \Big] \geq -C\E\big[|X - X'|^2 \big].
    \end{align*}
    By Theorem \ref{thm.main.disp}, it follows that there is an $\eps > 0$ which is independent of $T_0$ {(but dependent on the semi-monotonicity constant $C$ and the semi-monotonicity constants $C_{L,a}$ and $C_{L,x}$ of $L$; i.e. it is small enough such that $C_{L,a}-\eps C - (\eps^2/2)C_{L,x}>0$)} such that if $U$ is a classical solution to \eqref{ME} on the time interval $[T_0,T]$, then the equation
    \begin{align} 
    \begin{cases}
        \ds - \partial_t V - \sigma \Delta_{\id} V - \sigma_0 \Delta_{\com } V + \hat{H}\Big(x, D_x V(t,x,m), \big(\text{Id}, D_xV(t,\cdot,m) \big)_{\#} m \Big) \vspace{.2cm}
        \\
        \ds \qquad + \int_{\R^d} D_p \hat{H}\Big(y,D_xV(t,y,m), \big(\text{Id}, D_xV(t,\cdot,m) \big)_{\#} m \Big) \cdot D_m V(t,x,m,y) m(dy) = 0 
        \vspace{.2cm} \\ 
        \qquad \qquad \qquad (t,x,m) \in ((T_0 - \eps) \vee 0, T_0) \times \R^d \times \cP_2(\R^d), 
        \vspace{.2cm} \\ \ds 
        V(T_0,x,m) = U(T_0,x,m), \quad (x,m) \in \R^d \times \cP_2(\R^d)
    \end{cases}
\end{align}
admits a classical solution. By concatenating $V$ and $U$, we find that there is a constant $\eps > 0$, independent of $T_0$, with the following property: if there exists a classical solution to \eqref{ME} on the time interval $[T_0,T]$, then there also exists a classical solution on the larger time interval $[(T_0 - \eps) \vee T_0, T]$. By induction, this completes the proof.

\end{proof}

Here is a {well-known} lemma {(see for instance \cite[Remark 2.8]{GanMesMouZha})} which reduces the a-priori estimate \eqref{disp.monotone.prop} to Lasry--Lions monotonicity and an estimate on $D_{xx} U$. 

\begin{lemma} \label{lem.llmonotone.comp}
    Suppose that for some $C > 0$, a function $V : \R^d \times \cP_2(\R^d)\to \R$ is $C^2$, Lasry--Lions monotone, and $\|D_{xx} V\|_{\infty} \leq C$. Then we have 
    \begin{align*} 
        \E\Big[ \big(D_x V(X,\cL(X)) - {D_x}V(X',\cL(X')) \big)\cdot (X-X') \Big] \geq - C \E\big[|X - X'|^2 \big]
    \end{align*}
    for all square-integrable random variables $X,X'$. 
\end{lemma}

The following proposition contains an a-priori estimate on $D_{xx} U$ which, in light of Lemma \ref{lem.llmonotone.comp}, will be useful in proving Proposition \ref{prop.ll.apriori}.

\begin{proposition} \label{prop.dxx}
    Let Assumption \ref{assump.regularity.disp} and \ref{assump.LL} hold. Then there is a constant $C>0$ with the following property: for any $T_0 \in [0,T)$ and any $U : [T_0,T] \times \R^d \times \cP_2(\R^d) \to \R$ satisfying \eqref{ME}, we have 
    \begin{align*}
        \sup_{(t,x,m) \in [T_0,T] \times \R^d \times \cP_2(\R^d)} \Big( |D_x U(t,x,m)| + |D_{xx} U(t,x,m)| \Big) \leq C. 
    \end{align*}
\end{proposition}

\begin{proof}
   We start with the bound on $D_x U$. For this, we are going to use a stochastic representation of $D_x U$ in terms of the Pontryagin maximum principle. This will be an extension of the same result for standard mean field games (see e.g. Theorem 4.10 of \cite{CarDel:vol2}), so we only sketch the proof. We fix $(t_0,x_0,m_0) \in [0,T] \times \R^d \times \cP_2(\R^d)$, and define $X$ by
   \begin{align*}
      \begin{cases}
           dX_t = - D_p \hat{H}\Big(X_t, D_xU(t,X_t, m_t), \big(\text{Id}, D_xU(t,\cdot, m_t)\big)_{\#} m_t \Big) dt + \sqrt{2} dW_t + \sqrt{2\sigma_0} dW_t^0, 
       \\
       X_{t_0} \sim m_0, \quad m_t = \cL^0(X_t), 
      \end{cases}
   \end{align*}
   where $W$ and $W^0$ are independent Brownian motions on some filtered probability space satisfying the usual conditions, and $\cL^0(X_t) = \cL(X_t | \cF_t^{W_0})$, $\cF_t^{W^0}$ being the filtration generated by $W^0$. 
   By a verification argument, $m$ is {the first marginal of a} MFE for the game started from $m_{t_0}$ at time $t_0$, and in light of the uniqueness result as a consequence of the Lasry--Lions monotonicity (discussed in \cite[page 306]{CD1}), it is {coming from} the unique MFE, {which has the representation} 
   \begin{align*}
       \mu_t = \big(\text{Id}, D_xU(t,\cdot, m_t)\big)_{\#} m_t. 
   \end{align*}
   For $x_0 \in \R^d$, we define $(X^{x_0}_{t})_{t\in[t_{0},T]}$ by 
   \begin{align*}
       dX_t^{x_0} = - D_p \hat{H}\Big(X_t^{x_0}, D_xU(t,X_t^{x_0}, m_t), \mu_t \Big) dt + \sqrt{2} dW_t + \sqrt{2\sigma_0} dW_t^0, \quad X_{t_0}^{x_{0}}= {x_0}. 
   \end{align*}
 On the one hand, a verification argument shows that the control
 \begin{align*}
     \alpha^*_t = - D_p \hat{H}\Big( X_t^{x_0}, D_x U(t, X_t^{x_0}, m_t), \mu_t \Big)
 \end{align*}
 is the unique optimizer for the control problem
   \begin{align*}
       \inf_{\alpha} \E\bigg[ \int_{t_0}^T L\Big(X_t, \alpha_t, \Phi(\mu_t) \Big) dt + G(X_T, m_T) \bigg]
   \end{align*}
   with dynamics 
   \begin{align*}
       d X_t = \alpha_t dt + \sqrt{2} dW_t + \sqrt{2\sigma_0} dW_t^0, \quad X_{t_0} = x_0.
   \end{align*}
   On the other hand, the stochastic maximum principle applied to this control problem shows that we must have
   \begin{align*}
    - D_p \hat{H}\Big( X_t^{x_0}, D_x U(t, X_t^{x_0}, m_t), \mu_t \Big) =   \alpha_t^* = - D_p \hat{H}\Big(X_t^{x_0}, Y_t^{x_0}, \mu_t \Big), 
   \end{align*}
   where $Y^x$ is the unique solution of the BSDE
   \begin{align*}
       dY_t^{x_0} = D_x \hat{H}\big(X_t^{x_0}, Y_t^{x_0}, \mu_t \big) dt + Z_t^{x_0} dW_t +  Z_t^{x_0,0} dW_t^0, \quad Y_T^{x_0} = D_x G(X_T^{x_0}, m_T).
   \end{align*}
   By the strict convexity of $\hat{H}$ in $p$, we deduce that 
   \begin{align*}
       D_x U(t,X_t^{x_0},m_t) = Y_t^{x_0}.
   \end{align*}
   Now $d |Y_t^{x_0}|^2$, taking expectations, and using Assumption \ref{assump.regularity.LL}, we get 
   \begin{align*}
       \E\big[ |Y_t^{x_0}|^2 \big] \leq C \Big(1 + \E\bigg[ \int_t^T |Y_s^{x_0}|^2 ds \bigg]\Big), 
   \end{align*}
   and so Gronwall's Lemma gives 
   \begin{align*}
       |D_x U(t_0,x_0, m_0)|^2 = |Y_{t_0}^{x_0}|^2 \leq \sup_{t_0 \leq t \leq T} \E\big[ |Y_t^{x_0}|^2 \big] \leq C, 
   \end{align*}
   which gives the desired bound on $D_x U$.
   
   For the bound on $D_{xx} U$, we are going to mimic an argument in {the proof of \cite[Proposition 5.53]{CarDel:vol2}}. We fix $(t_0,\bx_0,m_0)$, and we define $X^{x_0}_t, m_t, \mu_t$ as before. We also define $Y$ by 
   \begin{align*}
       Y_t = x_0 + \sqrt{2}W_t^{t_0} + \sqrt{2\sigma_0}W_t^{0,t_0}, \quad W_t^{t_0} = W_t - W_{t_0}, \quad W_t^{0,t_0} = W_t^0 - W_{t_0}^0. 
   \end{align*}
   Then we compute 
   \begin{align*}
       dU(t,Y_t, m_t) = \hat{H}\big(Y_t, D_xU(t,Y_t,m_t), \mu_t\big) dt + dM_t, 
   \end{align*}
   with $M$ a martingale whose form is unimportant. We now use the independence of $W^{t_0}$ from $(W^{0,t_0}, m, \mu)$ to deduce that
   \begin{align*}
       U(t_0&,x_0,m_0) = \E\bigg[ \int_{t_0}^T -\hat{H}\big(Y_s, D_x U(s,Y_s, m_s), \mu_s \big) ds + G(Y_T, \mu_T) \bigg]
       \\
       &= \E\bigg[ \int_{t_0}^T \int_{\R^d} -\hat{H}\big(x + y + \sqrt{2\sigma_0}W_t^{0,t_0}, D_x U(t, x + y + \sqrt{2\sigma_0}W_t^{0,t_0}, m_t\big) p_{T-t}(y) dy dt
       \\
       &\qquad + \int_{\R^d} G(x + y + \sqrt{2\sigma_0} W_t^{0,t_0} p_{T-t_0}(dy)\bigg], 
   \end{align*}
   with $(t,x) \mapsto p_t(x)$ being the fundamental solution to the heat equation. Differentiating in $x$, we find that $V \coloneqq D_x U$ satisfies
        \begin{align*}
       V(t_0&,x_0,m_0) 
       \\
       &= \E\bigg[ \int_{t_0}^T \int_{\R^d} -D_x \hat{H}\big(x + y + \sqrt{2\sigma_0}W_t^{0,t_0}, V(t, x + y + \sqrt{2\sigma_0}W_t^{0,t_0}, \mu_t\big) p_{T-t}(y)
       \\
       &\quad + \int_{t_0}^T \int_{\R^d} \bigg( -D_x V\big(t,x+y+ \sqrt{2\sigma_0}W_t^{0,t_0}, m_t \big)
       \\
       &\qquad \qquad \qquad  \times D_p \hat{H}\big(x + y + \sqrt{2\sigma_0} W_t^{0,t_0}, V(t, x + y + \sqrt{2\sigma_0}W_t^{0,t_0}, \mu_t\big)p_{T-t}(y) \bigg) dy dt
       \\
       &\quad  + \int_{\R^d} D_x G\big(x + y + \sqrt{2\sigma_0}W_t^{0,t_0}\big)  p_{T-t_0}(y)  dy\bigg]
       \\
       &= \E\bigg[ \int_{t_0}^T \int_{\R^d} - D_x H\big(y + \sqrt{2\sigma_0}W_t^{0,t_0}, V(t, y + \sqrt{2\sigma_0}W_t^{0,t_0}, m_t\big) p_{T-t}(x-y)
       \\
       &\quad + \int_{t_0}^T \int_{\R^d} \bigg( - D_x V\big(t,y+ \sqrt{2\sigma_0}W_t^{0,t_0}, m_t \big)
       \\
       &\qquad \qquad \qquad \times  D_p \hat{H}\big(y + \sqrt{2\sigma_0}W_t^{0,t_0}, V(t,  y + \sqrt{2\sigma_0}W_t^{0,t_0}, \mu_t\big) p_{T-t}(x-y) \bigg)  dy dt
       \\
       &\quad + \int_{\R^d} D_xG \big( y + \sqrt{2\sigma_0}W_t^{0,t_0} \big) p_{T-t_0}(x-y) dy\bigg].
   \end{align*}
  Now using the bound on $V = D_x U$, Assumption \ref{assump.regularity.LL}, and the classical smoothing estimate 
   \begin{align*}
       \norm{ D_x \int_{\R^d} \psi(y) p_{T-t}(x-y) dy }_{\infty} \leq \frac{C}{\sqrt{T-t}} \|\psi\|_{\infty}, 
   \end{align*}
   {(where $C>0$ is independent of $\psi \in L^\infty(\R^d)$)} 
   we obtain the bound 
   \begin{align*}
       \sup_{(x,\mu) \in \R^d \times \cP_2(\R^d)} \big| D_{xx}U(t_0,x,\mu) \big| \leq C + \int_{t_0}^T \frac{C}{\sqrt{T-t}} \Big(1 + \sup_{(x,\mu) \in \R^d \times \cP_2(\R^d)}  \big| D_{xx}U(t,x,\mu) \big| \Big) dt, 
   \end{align*}
   which as in {the proof of \cite[Proposition 5.53]{CarDel:vol2}} is enough to conclude the desired bound. {We emphasize that the constant $C>0$ produced to uniformly bound the second derivative might depend on $T$, but it is independent of $T_0$.}
\end{proof}

\begin{proposition} \label{prop.LLpropagation}
Let Assumption \ref{assump.regularity.disp} and \ref{assump.LL} hold. Then if $U : [T_0,T] \times \R^d \times \cP_2(\R^d) \to \R$ satisfies \eqref{ME}, then $U(t,\cdot, \cdot)$ is Lasry--Lions monotone for each $t \in [T_0,T]$. 
\end{proposition}

\begin{proof}
This follows from the computations appearing in the proof of Remark 4.4 in \cite{MouZha:2022}.
\end{proof}

\begin{proof}[Proof of Proposition \ref{prop.ll.apriori}]
    For existence, combine Propositions \ref{prop.dxx} with Proposition \ref{prop.LLpropagation}. Uniqueness follows from the uniqueness of mean field equilibria as in the proof of Theorem \ref{thm.main.disp}, and uniqueness of equilibria under the Lasry-Lions monotonicity condition can be obtained by mimicking the proof of Theorem 3.29 in \cite{CD1}, for instance.
\end{proof}

\appendix 

\section{Existence of solutions to the Nash system}\label{app:A}

The goal of this Appendix is to prove the following existence result.

\begin{proposition} \label{prop.existence.N}
   {Let Assumptions \ref{assump.regularity.disp} and \ref{assump.disp} hold.} Then for all $N\in\N$ large enough, there exists a symmetric admissible {classical} solution $\bu^N$ to the Nash system \eqref{nashsystem2}.
\end{proposition}

    To prove this, we will show that for any $(g^i)_{i = 1,\dots,N}$, $g^i \in C^5\big( (\R^d)^N \big)$ with bounded derivatives of order $2$, $3$, and $4$, the system
    \begin{align} \label{nashsystem3}
    \begin{cases}
       \ds  - \partial_t v^{i} - \sum_{j = 1}^N \Delta_{j} v^{i} - \sigma_0 \sum_{j,k = 1}^N \tr\big(D_{jk} v^{i} \big) + \hat{H}^{N,i}(\bx,\diag \bv)
       \\
       \ds \qquad - \sum_{j \neq i} a^{N,j}(\bx, \diag \bv^N) \cdot D_j v^{i} = 0, 
      \quad  (t,\bx) \in [0,T] \times (\R^d)^N, 
       \\
       \ds v^{i}(T,\bx) = g^i(\bx), \quad \bx \in (\R^d)^N.
    \end{cases}
\end{align}
has a solution if $T < \eps$, for some $\eps$ depending on $(g^i)_{i = 1,\dots,N}$ only through $\max_{j \neq i} \|D_j g^i\|_{\infty}$ and $\max_i \|D^2 g^i\|_{\infty}$. To this end, we will introduce a truncation procedure: we let $(\pi_R)_{R > 0}$ be a collection of maps such that all the derivatives of $\pi_R$ are bounded uniformly in $R$, and $\pi_R(x) = x$ for $|x| \leq R$. For $\bx \in (\R^d)^N$, set $\pi_R(\bx) = (\pi_R(x^1),\dots,\pi_R(x^N))$. Then define 
    \begin{align*}
        \hat{H}^{N,R,i}(\bx,\bp) = \hat{H}^{N,i}\big( \pi_R(\bx), \pi_R(\bp) \big), \quad a^{N,R,i}(\bx,\bp) = a^{N,i}\big( \pi_R(\bx), \pi_R(\bp) \big), \quad g^{R,i}(\bx) = g^i(\pi(\bx)).
    \end{align*}
    We consider the equation
    \begin{align} \label{nashsystem4}
    \begin{cases}
       \ds  - \partial_t v^{R,i} - \sum_{j = 1}^N \Delta_{j} v^{R,i} - \sigma_0 \sum_{j,k = 1}^N \tr\big(D_{jk} v^{R,i} \big) + \hat{H}^{N,R,i}(\bx,\diag \bv^R)
       \\
       \ds \qquad - \sum_{j \neq i} a^{N,R,j}(\bx, \diag \bv^N) \cdot D_j v^{R,i} = 0, 
      \quad  (t,\bx) \in [0,T] \times (\R^d)^N, 
       \\
       \ds v^{R,i}(T,\bx) = g^{R,i}(\bx), \quad \bx \in (\R^d)^N.
    \end{cases}
\end{align}
This equation has a unique classical solution $\bv^R = (v^{R,1},\dots,v^{R,N})$, see e.g. Proposition 3.3 of \cite{ma1994}.

\begin{lemma} \label{lem.apriori.shorttime}
    Let Assumption \ref{assump.regularity.disp} hold, and assume that $g^i \in C^5((\R^d)^N)$  with bounded derivatives of order $2$, $3$, $4$, and $5$. Then there is a constant $\eps$ which depends on $(g^i)_{i = 1,\dots,N}$ only through $\max_{j \neq i} \|D_j g^i\|_{\infty}$ and $\max_{i,j,k} \|D_{jk} g^i\|_{\infty}$ such that if $T < \eps$, then
    \begin{align*}
         \max_{j \neq i} \|D_j v^{R,i}\|_{\infty} + \max_{i,j,k} \|D_{kj} v^{R,i} \|_{\infty} \leq C, 
    \end{align*}
    for some constant $C>0$ which is independent of $R$. 
\end{lemma}

\begin{proof}
  We emphasize throughout this proof, $C$ denotes a constant which is independent of $R$ but which, crucially, can depend on $N$. Define 
  \begin{align*}
      m^R(t) = \max_{j \neq i} \|D_j v^{R,i}\|_{\infty,t} + \max_{i,j,k} \|D_{kj} v^{R,i}\|_{\infty,t}.
  \end{align*}
  For each $i,j,k = 1,\dots,N$, we have that $v^{R,i,j} \coloneqq D_j v^{R,i}$ and $v^{R,i,j,k} = D_{kj} v^{R,i}$ satisfy 
   \begin{align*} 
    \begin{cases} \ds - \partial_t v^{R, i,j} - \scrL^N v^{R,i,j}
   - \sum_{k \neq i} \sum_{l = 1}^N D_{jl} v^{R,l} D_{p^l} a^{N,R,k} v^{i,k} - \sum_{k \neq i} D_{x^j} a^{N,R,k} v^{R,i,k}   \\ \ds
    \quad  + \sum_{k \neq i}  D_{jk}v^{R,k}  D_{p^k} \hat{H}^{N,R,i}
    + D_{x^j} \hat{H}^{N,R,i} = 0, \quad (t,\bx) \in [0,T] \times (\R^d)^N, \vspace{.2cm} 
    \\ \ds 
    v^{R,i,j}(T,\bx) = D_{x^j} g^i(\bx), \quad \bx \in (\R^d)^N, 
    \end{cases}
\end{align*}
and
   \begin{align*} 
    \begin{cases} \ds - \partial_t v^{R,i,j,k} - \scrL^N v^{R,i,j,k} + \sum_{q = 1}^4 \cT^{R,i,j,k,q} = 0, \quad (t,\bx) \in [0,T] \times (\R^d)^N, \vspace{.2cm}
    \\ \ds 
    v^{R,i,j,k}(T,\bx) = D_{x^k x^j} g^i(\bx), \quad \bx \in (\R^d)^N, 
    \end{cases}
\end{align*}
where 
\begin{align*}
   &\cT^{i,j,k,1} = - \sum_{l \neq i} \sum_{q = 1}^N D_k v^{R,q,j,q} D_{p^q} a^{N,R,l} v^{R,i,l} + \sum_{l \neq i} D_{p^l} \hat{H}^{N,R,i} D_k v^{R,l,j,l}, 
   \\
   &\cT^{i,j,k,2} = - \sum_{l = 1}^N \sum_{q = 1}^N D_{p^q} a^{N,R,l} v^{R,q,k,q} v^{R,i,j,l} - \sum_{l = 1}^N D_{x^k} a^{N,R,l} v^{R,i,j,l} - \sum_{l \neq i} \sum_{q = 1}^N D_{p^q} a^{N,R,l} v^{R,q,j,q} v^{R,i,l,k} 
   \\
   &\qquad \qquad \qquad - \sum_{l \neq i} D_{x^j} a^{N,R,l} v^{R,i,k,l},
   \\
   &\cT^{i,j,k,3} = - \sum_{l \neq i} \sum_{q,n = 1}^N D_{p^n p^q} a^{N,R,l} v^{R,q,j,q} v^{R,n,k,n} v^{R,i,l} + \sum_{l \neq i} \sum_{q = 1}^N D_{p^q p^l} \hat{H}^{N,R,i} v^{R,q,k,q} v^{R,l,j,l} 
   \\
   &\qquad \qquad \qquad - \sum_{l \neq i} \sum_{q = 1}^N D_{x^kp^q} a^{N,R,l} v^{R,q,j,q} v^{R,i,l} - \sum_{l \neq i} \sum_{q = 1}^N D_{p^q x^j} a^{N,R,l} v^{R,q,k,q} v^{R,i,l} 
   \\
   &\qquad \qquad \qquad + \sum_{l \neq i} D_{x^k p^l} \hat{H}^{N,R,i} v^{R,l,j,l} + \sum_{l = 1}^N D_{p^lx^j} \hat{H}^{N,R,i} v^{R,l,k,l}
   \\
   &\cT^{i,j,k,4} = - \sum_{l \neq i} D_{x^kx^j} a^{N,R,l} v^{R,i,l} + D_{x^kx^j} \hat{H}^{N,R,i}, 
\end{align*}
and with $\scrL^N$ now denoting the the differential operator, {similar to the one defined in \eqref{def:L^N}, but with truncated coefficients,} which acts on $f : [0,T] \times (\R^d)^N \to \R$ via 
\begin{align*}
    \scrL^N f = \sum_{j = 1}^N \Delta_j f + \sum_{j,k = 1}^N \tr\big(D_{jk} f \big) + \sum_{j = 1}^N a^{N,R,j} (\bx, \diag \bu^N) \cdot D_j f.
\end{align*}
Fix $(t_0,\bx_0) \in [0,T] \times (\R^d)^N$, and set define a process $\bX$ by 
\begin{align*}
    dX_t^i = a^{N,R,i}(\bX_t, D_{i} v^{R,i}(t,\bX_t)) dt + \sqrt{2} dW_t^i + \sqrt{2\sigma_0} dW_t^0, 
\end{align*}
and then set 
\begin{align*}
    Y_t^{i,j} = v^{R,i,j}(t,\bX_t), \quad Z_t^{i,j,k} = \sqrt{2} v^{R,i,j,k}(t,\bX_t), \quad Z_t^{i,j,0} = \sqrt{2\sigma_0} \sum_{k = 1}^N v^{R,i,j,k}(t,\bX_t) 
\end{align*}
for $i,j,k = 1,\dots,N$. Then apply It\^o's formula as in the proof of Proposition \ref{prop.uij}, and use the estimates on $\hat{H}^{N,i}$ and $\ba^{N,i}$ from Proposition \ref{prop.aderivscaling} and Lemma \ref{lem.hatH}, together with the bounds on the derivatives of $\pi_R$, to deduce that for $i \neq j$, we have
\begin{align*}
    \E\bigg[ &|Y_{t_0}^{i,j}|^2 + \sum_{{k = 0}}^N \int_{t_0}^T |Z_t^{i,j,k}|^2 dt \bigg]
    \\
    &= \E\bigg[  |Y_{T}^{i,j}|^2 + {2}\int_{t_0}^T Y_t^{i,j} \Big( \sum_{k \neq i} \sum_{l=1}^{N} D_{jl}v^{R,i} D_{p^l} a^{N,R,k} v^{R,i,k} + \sum_{k \neq i} D_{x^j} a^{N,R,k} v^{R,i,k}
    \\
    &\qquad - \sum_{k \neq i} D_{jk} v^{N,R,k} D_{p^k} \hat{H}^{N,R,i} {-} D_{x^j} \hat{H}^{N,R,i}\Big)(t,\bX_t) dt \bigg]
     \\
     &\leq C\Big( |m^R(T)|^2 + 1 + \int_{t_0}^T |m^R(t)|^3 dt \Big),
\end{align*}
with $C$ independent of $R$. 
Taking a supremum over $\bx$, and then a max over $i \neq j$, we find that 
\begin{align} \label{shorttime.vij}
    \max_{j \neq i} \|D_j v^i\|^2_{\infty, t} \leq C \Big( |m^R(T)|^2 + 1 + \int_{t}^T |m^R(s)|^{3} ds \Big). 
\end{align}
Next, we define
\begin{align*}
    Y_t^{i,j,k} = v^{R,i,j,k}(t,\bX_t), \quad Z_t^{i,j,k,l} = \sqrt{2} D_{x^l}u^{R,i,j,k}(t,\bX_t), \quad Z_t^{i,j,0} = \sqrt{2\sigma_0} \sum_{l = 1}^N D_{x^l} u^{R,i,j,k}(t,\bX_t), 
\end{align*}
and apply It\^o's formula as in the proof of Proposition \ref{prop.uiji} to find that 
\begin{align*}
    \E\bigg[ &|Y_{t_0}^{i,j,k}|^2 + \sum_{{l = 0}}^N \int_{t_0}^T |Z_t^{i,j,k,l}|^2 dt \bigg]
    \\
    &= \E\bigg[  |Y_{T}^{i,j,k}|^2 - {2}\int_{t_0}^T Y_t^{i,j,k} \Big( \sum_{q = 1}^4 \cT^{i,j,k,q}(t,\bX_t) \Big) dt \bigg].
\end{align*}
One can directly check that 
\begin{align*}
    |\cT^{i,j,k,1}| \leq C \big(1 + m^{R}(t) \big) \sum_{{r=1}}^{N} |Z_t^{r,j,r,k}|, \quad |\cT^{i,j,k,2}| + |\cT^{i,j,k,3}| + |\cT^{i,j,k,4}| \leq C \big(1 + m^R(t)^3\big). 
\end{align*}
We thus obtain 
\begin{align*}
    \E\bigg[ &|Y_{t_0}^{i,j,k}|^2 + \sum_{{l = 0}}^N \int_{t_0}^T |Z_t^{i,j,k,l}|^2 dt \bigg]
    \\
    &\leq C \Big( |m^R(T)|^2 + \int_{t_0}^T (1 + |m^R(t)|^2) \sum_{{r=1}}^{N} |Z_t^{r,j,r,k}| dt +  \int_{t_0}^T \big(1 + |m^R(t)|^4 \big) dt \Big)
    \\
    &\leq C \Big( |m^R(T)|^2 + 1 + \int_{t_0}^T |m^R(t)|^4 dt \Big) + \frac{1}{2 N} \E\bigg[ \int_{t_0}^T \sum_{{r=1}}^{N} |Z_t^{r,j,r,k}|^2 dt \bigg]
\end{align*}

Summing over all $i,j,k$, we get that 
    \begin{align*}
    \sum_{i,j,k = 1}^N \E\bigg[ &|Y_{t_0}^{i,j,k}|^2 + \sum_{{l = 0}}^N \int_{t_0}^T |Z_t^{i,j,k,l}|^2 dt \bigg]
    \\
    &\leq C \Big( |m^R(T)|^2 + 1 + \int_{t_0}^T |m^R(t)|^4 dt \Big) + \frac{1}{2} \sum_{i = 1}^{N} \E\bigg[ \int_{t_0}^T \sum_{{j,k,r=1}}^{N} |Z_t^{r,j,r,k}|^2 dt \bigg]
    \\
    &\leq C \Big( |m^R(T)|^2 + 1 + \int_{t_0}^T |m^R(t)|^4 dt \Big) + \frac{1}{2} \sum_{{i,j,k,l=1}}^{N} \E\bigg[ \int_{t_0}^T \sum_{r=1}^{N} |Z_t^{i,j,k,l}|^2 dt \bigg]
\end{align*}
and so 
  \begin{align*}
    \sum_{i,j,k = 1}^N \E\bigg[ |Y_{t_0}^{i,j,k}|^2 + \sum_{l = 1}^N \int_{t_0}^T |Z_t^{i,j,k,l}|^2 dt \bigg]
   \leq  C \Big( |m^R(T)|^2 + 1 + \int_{t_0}^T |m^R(t)|^4 dt \Big).
    \end{align*}
Taking a supremum over $\bx_0$, we deduce that 
\begin{align*}
    \max_{i,j,k} \|D_{kj} v^i\|_{\infty,t} \leq C \Big( |m^R(T)|^2 + 1 + \int_{t_0}^T |m^R(t)|^4 dt \Big),
\end{align*}
which we combine with \eqref{shorttime.vij} to obtain
\begin{align*}
    |m^R(t)|^2 \leq C \Big( |m^R(T)|^2 + 1 + \int_{t_0}^T |m^R(t)|^4 dt \Big), 
\end{align*}
from which the result easily follows.
\end{proof}

Using Lemma \ref{lem.apriori.shorttime}, we get the following short time existence result.

\begin{proposition} \label{prop.existence.shorttime}
   Let Assumption \ref{assump.regularity.disp} hold, and assume that $g^i \in C^5((\R^d)^N)$  with bounded derivatives of order $2$, $3$, $4$, and $5$. Then there is a constant $\eps>0$ which depends on $(g^i)_{i = 1,\dots,N}$ only through
    \begin{align*}
        \max_{j \neq i} \|D_j g^i\|_{\infty} + \max_{i,j,k} \|D_{kj} g^i\|_{\infty}
    \end{align*}
    such that if $T < \eps$, then there exists a classical solution to \eqref{nashsystem3}, with $v^i, D_j v^i$, $D_{kj} v^i \in C^{1,2}_{\text{loc}}$, and with bounded derivatives of order $2,3$, and $4$. Moreover, if $(g^i)_{i = 1,\dots,N}$ satisfies the symmetry conditions \eqref{symmetry.1}, \eqref{symmetry.2}, then so does $(v^{N,i})_{i = 1,\dots,N}$.
\end{proposition}

\begin{proof}
    Lemma \ref{lem.apriori.shorttime} shows provides the existence of constants $\eps>0$ and $C>0$ (depending only on the constants indicated) such that if $T < \eps$, then the second spatial derivatives of $v^{R,i}$ are bounded, uniformly in $R$. From here, standard parabolic regularity together with the regularity of $\hat{H}^{N,i}$ and $a^{N,i}$ allows us to verify that $v^{R,i}$ is bounded in the parabolic H\"older space $C_{t,\bx,\text{loc}}^{2 + \alpha}$, and that the derivatives of order $2$, $3$, and $4$ are bounded. We can thus produce a solution $(v^i)_{i = 1,\dots,N}$ to \eqref{nashsystem3} with the desired regularity properties. If $(g^i)_{i = 1,\dots,N}$ satisfies the symmetry condition \eqref{symmetry}, then so does $(g^{R,i})_{i = 1,\dots,N}$, and hence $(v^{R,i})_{i = 1,\dots,N}$ by the uniqueness of classical solutions to \eqref{nashsystem4}, and so $(v^i)_{i = 1,\dots,N}$ satisfies \eqref{symmetry} also. 
\end{proof}

\begin{proof}[Proof of Proposition \ref{prop.existence.N}]
    Note that because of Assumption \ref{assump.regularity.disp}, the function $g^{N,i}$ defined as $g^{N,i}(\bx) := G(x^i,m_{\bx}^{N,-i})$ is of class $C^5$, with bounded derivatives of order $2$, $3$, $4$ and $5$. Proposition \ref{prop.existence.shorttime} thus allows us to find a solution with the desired properties on some time interval $[T-\eps,T]$, {with some $\eps>0$ small,} and then our main a-priori estimate Theorem \ref{thm.uniform.nash} ensures that $D_j u^{N,i}$ and $D_{kj} u^{N,i}$ are bounded independently of $\eps$. We can thus build a solution backwards in time by repeatedly applying Proposition \ref{prop.existence.shorttime}.
\end{proof}

\section{Proof of Lemmas \ref{lem.hatH} and \ref{lem.third.order.coeff}}\label{app:B}

Recall that for $N \in \N$ and $i \in \{1,...,N\}$, $\hat{H}^{N,i} : (\R^d)^N \times (\R^d)^N \to \R$ is defined by 
\begin{align} \label{hatHi.app}
    \hat{H}^{N,i}(\bx,\bp) = H\big(x^i, p^i, m_{\bx, \ba^N(\bx,\bp)}^{N,-i} \big).
\end{align}
In addition, for $i,j = 1,...,N$, $\hat{H}^{i,j} : (\R^d)^N \times (\R^d)^N \to \R$ is defined by 
\begin{align} \label{hatHij.app}
    \hat{H}^{i,j}(\bx,\bp) = \frac{1}{N-1} \sum_{q \neq i} D_{\mu}^a H\big( x^i,p^i,m_{\bx,\ba^N(\bx,\bp)}^{N,-i}, x^q,a^{N,q}(\bx,\bp) \big) D_{p^j} a^{N,q}(\bx,\bp).
\end{align}

\begin{proof}[Proof of Lemma \ref{lem.hatH}]
    As in Section \ref{sec.uniform}, we argue as if $d = 1$ for simplicity of notation, and we begin by differentiating \eqref{hatHi.app} to see that 
\begin{align*}
    &D_{p^j} \hat{H}^{N,i}(\bx,\bp) = D_p H\big( x^i,p^i,m_{\bx,\ba^N}^{N,-i}\big) 1_{ i = j}
    \\
    &\qquad + \frac{1}{N-1} \sum_{q \neq i} D_{\mu}^a H\big(x^i, p^i, m_{\bx,\ba^N}^{N,-i}, x^q, a^{N,q} \big) D_{p^j} a^{N,q}, 
    \\
    &D_{p^k p^j} \hat{H}^{N,i}(\bx,\bp) = D_{pp} H \big( x^i,p^i,m_{\bx,\ba^N}^{N,-i} \big) 1_{i = j = k}
    \\
    &\qquad + \frac{1}{N-1} \sum_{q \neq i} D_{\mu}^a D_p H\big( x^i,p^i,m_{\bx,\ba^N}^{N,-i}, x^q,a^{N,q} \big) \Big( D_{p^k} a^{N,q} 1_{i = j} + D_{p^j} a^{N,q} 1_{i = k} \Big)
    \\
    &\qquad + \frac{1}{(N-1)^2} \sum_{q \neq i} \sum_{r \neq i} D_{\mu}^a D_{\mu}^a H\big( x^i,p^i,m_{\bx,\ba^N}^{N,-i}, x^q,a^{N,q},x^r,a^{N,r} \big) D_{p^j} a^{N,q} D_{p^k} a^{N,r}
    \\
    &\qquad + \frac{1}{N-1} \sum_{q \neq i} D_a D_{\mu}^a H\big( x^i,p^i,m_{\bx,\ba^N}^{N,-i}, x^q,a^{N,q} \big) D_{p^k} a^{N,q} D_{p^j} a^{N,q}
    \\
    &\qquad + \frac{1}{N-1} \sum_{q \neq i} D_{\mu}^a H\big( x^i,p^i,m_{\bx,\ba^N}^{N,-i}, x^q,a^{N,q} \big) D_{p^jp^k} a^{N,q}, 
    \\
    &D_{p^kp^jp^l} \hat{H}^{N,i} = D_{ppp} H\big(x^i,p^i, m_{\bx,\ba^N}^{N,-i}) 1_{i = j = k = l}
    \\
    &\qquad + \frac{1}{N-1} \sum_{q \neq i} D_{\mu}^a D_{pp} H\big(x^i,p^i,m_{\bx,\ba^N}^{N,-i}, x^q,a^{N,q}\big) \Big( D_{p^l} a^{N,q} 1_{i = j = k} + D_{p^k} a^{N,q} 1_{i = j = l} + D_{p^j} a^{N,q} 1_{i = k = l} \Big)
    \\
    &\qquad + \frac{1}{N-1} \sum_{q \neq i} D_a D_{\mu}^a D_p H\big( x^i,p^i,m_{\bx,\ba^N}^{N,-i}, x^q,a^{N,q} \big) \Big( D_{p^l} a^{N,q} D_{p^k} a^{N,q} 1_{i = j} 
    \\
    &\qquad \qquad \qquad \qquad \qquad \qquad \qquad \qquad \qquad \qquad \qquad+ D_{p^l} a^{N,q} D_{p^j} a^{N,q} 1_{i = k} + D_{p^j} a^{N,q} D_{p^k} a^{N,q} 1_{i = l}\Big)
    \\
    &\qquad + \frac{1}{(N-1)^2} \sum_{q \neq i} \sum_{r \neq i} D_{\mu}^a D_{\mu}^a D_p H \big(x^i,p^i,m_{\bx,\ba^N}^{N,-i}, x^q, a^{N,q}, x^r, a^{N,r} \big) \Big( D_{p^k} a^{N,q} D_{p^l} a^{N,r} 1_{i = j}
    \\
    &\qquad \qquad \qquad \qquad \qquad \qquad \qquad \qquad \qquad \qquad \qquad
    + D_{p^j} a^{N,q} D_{p^k} a^{N,r} 1_{i = l} + D_{p^j} a^{N,q} D_{p^l} a^{N,r} 1_{i = k}\Big)
    \\
    &\qquad + \frac{1}{N-1} \sum_{q \neq i} D_{\mu}^a D_p H \big(x^i,p^i,m_{\bx,\ba^N}^{N,-i},x^q,a^{N,q}\big) \Big( D_{p^lp^k} a^{N,q} 1_{i = j} 
     \\
    &\qquad \qquad \qquad \qquad \qquad \qquad \qquad \qquad \qquad \qquad \qquad
    + D_{p^jp^l} a^{N,q} 1_{i = k} + D_{p^jp^k} a^{N,q} 1_{i = l} \Big)
    \\
    &\qquad + \frac{1}{(N-1)^3} \sum_{q \neq i} \sum_{r \neq i} \sum_{s \neq i} D_{\mu}^a D_{\mu}^a D_{\mu}^a H\big( x^i, p^i,m_{\bx,\ba^N}^{N,-i}, x^q, a^{N,q}, x^r, a^{N,r}, x^s, a^{N,s} \big) D_{p^j} a^{N,q} D_{p^k} a^{N,q} D_{p^l} a^{N,k}
    \\
    &\qquad + \frac{1}{(N-1)^2} \sum_{q \neq i} \sum_{r \neq i} D_a D_{\mu}^a D_{\mu}^a H\big(x^i, p^i,m_{\bx,\ba^N}^{N,-i}, x^q, a^{N,q}, x^r, a^{N,r}\big) \Big( D_{p^l} a^{N,q} D_{p^j} a^{N,q} D_{p^k} a^{N,r}
    \\
    &\qquad \qquad \qquad \qquad \qquad \qquad \qquad \qquad \qquad \qquad \qquad
    + D_{p^j} a^{N,q} D_{p^k} a^{N,q} D_{p^l} a^{N,r} \Big)
    \\
    &\qquad + \frac{1}{(N-1)^2} \sum_{q \neq i} \sum_{r \neq i} D_{a'} D_{\mu}^a D_{\mu}^a H\big(x^i, p^i,m_{\bx,\ba^N}^{N,-i}, x^q, a^{N,q}, x^r, a^{N,r}\big)  D_{p^l} a^{N,r} D_{p^j} a^{N,q} D_{p^k} a^{N,r}
    \\
    &\qquad + \frac{1}{(N-1)^2} \sum_{q \neq i} \sum_{r \neq i} D_{\mu}^a D_{\mu}^a H\big(x^i,p^i,m_{\bx,\ba^N}^{N,-i},x^q,a^{N,q} x^r, a^{N,r}\big) \Big( D_{p^l p^j} a^{N,q} D_{p^k} a^{N,r}
    \\
    &\qquad \qquad \qquad \qquad \qquad \qquad \qquad \qquad \qquad \qquad \qquad + D_{p^j} a^{N,q} D_{p^lp^k} a^{N,r} + D_{p^l} a^{N,r} D_{p^j p^k} a^{N,q} \Big)
    \\
    &\qquad + \frac{1}{N-1} \sum_{q \neq i} D_a D_a D_{\mu}^a H\big(x^i,p^i,m_{\bx,\ba^N}^{N,-i},x^q, a^{N,q}\big) D_{p^l} a^{N,q} D_{p^k} a^{N,q} D_{p^j} a^{N,q} 
    \\
    &\qquad + \frac{1}{(N-1)} \sum_{q \neq i} D_a D_{\mu}^a H\big( x^i,p^i,m_{\bx,\ba^N}^{N,-i}, x^q, a^{N,q} \big) \Big( D_{p^l} a^{N,q} D_{p^jp^k} a^{N,q} + D_{p^l p^k} a^{N,q} D_{p^j} a^{N,q} + D_{p^k} a^{N,q} D_{p^l p^j}a^{N,q} \Big)
    \\
    &\qquad +  \frac{1}{(N-1)} \sum_{q \neq i} D_{\mu}^a H\big( x^i,p^i,m_{\bx,\ba^N}^{N,-i}, x^q, a^{N,q} \big) D_{p^lp^jp^k} a^{N,q}
\end{align*}
Because $H$ has bounded second derivatives, we can use the bounds on the derivative of $\ba^N$ from Proposition \ref{prop.aderivscaling} to obtain
\begin{align*}
    |D_pH\big( x^i, p^i,m_{\bx,\ba^N}^{N,-i})| &\leq C \bigg(1 + |x^i| + |p^i| + \frac{1}{N} \sum_{j = 1}^N |x^j| + \frac{1}{N} \sum_{j = 1}^N |a^{N,j}(\bx,\bp)| \bigg)
    \\
    &\leq C \bigg(1 + |x^i| + |p^i| + \frac{1}{N} \sum_{j = 1}^N |x^j| + \frac{1}{N} \sum_{j = 1}^N |p^j| \bigg),
\end{align*}
while the boundedness of $D_{\mu} H$ and Proposition \ref{prop.aderivscaling} yield
\begin{align*}
    \Big| \frac{1}{N-1} \sum_{j \neq i} D_{\mu}^a H\big(x^i, p^i ,m_{\bx,\ba^N}^{N,-i}, x^q, a^{N,q}\big) D_{p^j} a^{N,q} \Big|
    \leq \frac{C}{N} \sum_{q \neq i} |D_{p^j} a^{N,q}| \leq \frac{C}{N}.
\end{align*}
Together, these bounds give 
\begin{align*}
|D_{p^j} \hat{H}^{N,i}(\bx,\bp) | \leq C \bigg(1 + |x^i| + |p^i| + \frac{1}{N} \sum_{j = 1}^N |x^j| + \frac{1}{N} \sum_{j = 1}^N |p^j| \bigg) 1_{i = j} + \frac{C}{N},
\end{align*}
as desired. 

Next, to bound $D_{p^kp^j} \hat{H}^{N,i}$, we note that by the boundedness of the second derivatives of $H$, 
\begin{align*}
    &\big| D_{pp} H\big( x^i,p^i,m_{\bx,\ba^{N}}^{N,-i} \big)1_{i = j = k} \big| \leq C 1_{i = j = k} \leq C \omega_{i,j,k}^N, 
    \\
    &\Big|\frac{1}{N-1} \sum_{q \neq i} D_{\mu}^a D_p H\big( x^i,p^i,m_{\bx,\ba^N}^{N,-i}, x^q,a^{N,q} \big) \Big( D_{p^k} a^{N,q} 1_{i = j} + D_{p^j} a^{N,q} 1_{i = k} \Big)\Big|
    \\
    &\quad \leq \frac{C}{N} \Big( \sum_{q \neq i} |D_{p^k} a^{N,q}|1_{i = j} + \sum_{q \neq i} |D_{p^j} a^{N,q}| 1_{i = k} \Big) 
    \\
    &\quad \leq \frac{C}{N} \Big(1_{i = j} + 1_{i = k}\Big) \leq C \omega_{i,j,k}^N
    \\
    &\Big|\frac{1}{(N-1)^2} \sum_{q \neq i} \sum_{r \neq i} D_{\mu}^a D_{\mu}^a H\big( x^i,p^i,m_{\bx,\ba^N}^{N,-i}, x^q,a^{N,q},x^r,a^{N,r} \big) D_{p^j} a^{N,q} D_{p^k} a^{N,r} \Big| 
    \\
    &\quad \leq \frac{C}{N^2} \sum_{q,r = 1}^N |D_{p^j} a^{N,q}| |D_{p^k} a^{N,r}| 
    \\
    &\quad \leq \frac{C}{N^2} \Big( \sum_{q \neq j} \sum_{r \neq k} |D_{p^j} a^{N,q}| |D_{p^k} a^{N,r}| + \sum_{q \neq j} |D_{p^j} a^{N,q}| |D_{p^k}a^{N,k}| + \sum_{r \neq k} |D_{p^k} a^{N,r}| |D_{p^j} a^{N,j}| + |D_{p^k} a^{N,k}| |D_{p^j} a^{N,j}| \Big)
    \\
    &\quad \leq \frac{C}{N^2}
    \\
    &\Big| \frac{1}{N-1} \sum_{q \neq i} D_a D_{\mu}^a H\big( x^i,p^i,m_{\bx,\ba^N}^{N,-i}, x^q,a^{N,q} \big) D_{p^k} a^{N,q} D_{p^j} a^{N,q} \Big| 
    \\
    &\quad \leq \frac{C}{N} \sum_{q=1}^N |D_{p^k} a^{N,q}| |D_{p^j}a^{N,q}| \leq \frac{C}{N} \sum_{q \neq j,k} |D_{p^j} a^{N,q}| |D_{p^k} a^{N,q}| + \frac{C}{N} |D_{p^k} a^{N,k}| |D_{p^j} a^{N,k}| + \frac{C}{N} |D_{p^j} a^{N,k}| 
    \\
    &\quad \leq C \Big( \frac{1}{N^2} + \frac{1}{N} 1_{j = k} \Big) \leq C \omega_{i,j,k}^N,
    \\
    &\Big|\frac{1}{N-1} \sum_{q \neq i} D_{\mu}^a H\big( x^i,p^i,m_{\bx,\ba^N}^{N,-i}, x^q,a^{N,q} \big) D_{p^jp^k} a^{N,q} \Big|
    \\
    &\quad \leq \frac{C}{N} \sum_{q = 1}^N |D_{p^jp^k} a^{N,q}| \leq \frac{C}{N} \sum_{q = 1}^N \omega_{j,k,q}^N \leq \frac{C}{N} \omega_{j,k}^N \leq \omega_{i,j,k}^N.
\end{align*}
We thus arrive at 
\begin{align*}
    \big| D_{p^j p^k} \hat{H}^{N,i}(\bx,\bp) \big| \leq C \omega_{i,j,k}^N.
\end{align*}
We next turn to estimating $|D_{p^l p^k p^j} \hat{H}^{N,i}|$. Using again the boundedness of the derivatives of $H$ (this time of order $2$ and $3$) and Proposition \ref{prop.aderivscaling}, we have 
\begin{align*}
    &\Big| D_{ppp} H\big(x^i,p^i, m_{\bx,\ba^N}^{N,-i}) 1_{i = j = k = l} \Big| \leq C 1_{i = j = k = l} \leq C \omega_{i,j,k,l}^N, 
    \\
    &\Big| \frac{1}{N-1} \sum_{q \neq i} D_{\mu}^a D_{pp} H\big(x^i,p^i,m_{\bx,\ba^N}^{N,-i}, x^q,a^{N,q}\big) \Big( D_{p^l} a^{N,q} 1_{i = j = k} + D_{p^k} a^{N,q} 1_{i = j = l} + D_{p^j} a^{N,q} 1_{i = k = l} \Big) \Big| 
    \\
    &\quad \leq \frac{C}{N} \sum_{q = 1}^N \Big( |D_{p^l} a^{N,q}| 1_{i = j = k} + |D_{p^k} a^{N,q}| 1_{i = j = l} + |D_{p^j} a^{N,q}| 1_{i = k = l} \Big) 
    \\
    &\quad \leq \frac{C}{N} \Big( 1_{i = j = k} + 1_{i = j = l} + 1_{i = k = l} \Big) \leq C \omega_{i,j,k,l}^N, 
    \\
    & \Big| \frac{1}{N-1} \sum_{q \neq i} D_a D_{\mu}^a D_p H\big( x^i,p^i,m_{\bx,\ba^N}^{N,-i}, x^q,a^{N,q} \big) \Big( D_{p^l} a^{N,q} D_{p^k} a^{N,q} 1_{i = j} 
    \\
    &\qquad \qquad \qquad \qquad \qquad \qquad \qquad \qquad \qquad \qquad \qquad+ D_{p^l} a^{N,q} D_{p^j} a^{N,q} 1_{i = k} + D_{p^j} a^{N,q} D_{p^k} a^{N,q} 1_{i = l}\Big) \Big|
    \\
    &\quad \leq \frac{C}{N}  \sum_{q = 1}^N \Big( |D_{p^l} a^{N,q}| |D_{p^k} a^{N,q}| 1_{i = j} +  |D_{p^l} a^{N,q}| |D_{p^j} a^{N,q} | 1_{i = k} + |D_{p^j} a^{N,q}||D_{p^k} a^{N,q}| 1_{i = l} \Big)
    \\
    &\quad \leq \frac{C}{N} \Big( \omega_{l,k}^N 1_{i = j} + \omega_{j,l}^N 1_{i = k} + \omega_{j,k}^N 1_{i = l} \Big) \leq C \omega_{i,j,k,l}^N, 
    \\
    & \Big|\frac{1}{(N-1)^2} \sum_{q \neq i} \sum_{r \neq i} D_{\mu}^a D_{\mu}^a D_p H \big(x^i,p^i,m_{\bx,\ba^N}^{N,-i}, x^q, a^{N,q}, x^r, a^{N,r} \big) \Big( D_{p^k} a^{N,q} D_{p^l} a^{N,r} 1_{i = j}
    \\
    &\qquad \qquad \qquad \qquad \qquad \qquad \qquad \qquad \qquad \qquad \qquad
    + D_{p^j} a^{N,q} D_{p^k} a^{N,r} 1_{i = l} + D_{p^j} a^{N,q} D_{p^l} a^{N,r} 1_{i = k}\Big)\Big|
    \\
    &\quad \leq \frac{C}{N^2} \sum_{q,r = 1}^N \Big( |D_{p^k} a^{N,q}| | D_{p^l} a^{N,r}| 1_{i = j} + |D_{p^j} a^{N,q}| | D_{p^k} a^{N,r}| 1_{i = l} + |D_{p^j} a^{N,q}| | D_{p^l} a^{N,r}| 1_{i = k} \Big)
    \\
    &\quad \leq \frac{C}{N^2} \Big(1_{i = j} + 1_{i = l} + 1_{i = k} \Big) \leq C \omega_{i,j,k,l}^N, 
    \\
    &\Big| \frac{1}{N-1} \sum_{q \neq i} D_{\mu}^a D_p H \big(x^i,p^i,m_{\bx,\ba^N}^{N,-i},x^q,a^{N,q}\big) \Big( D_{p^lp^k} a^{N,q} 1_{i = j} 
    + D_{p^jp^l} a^{N,q} 1_{i = k} + D_{p^jp^k} a^{N,q} 1_{i = l} \Big) \Big|
    \\
    &\quad \leq \frac{C}{N} \sum_{q = 1}^N \Big( |D_{p^lp^k} a^{N,q}| 1_{i = j} + |D_{p^jp^l} a^{N,q}|1_{i = k} + |D_{p^jp^k} a^{N,q}| 1_{i = l} \Big)
    \\
    &\quad \leq \frac{C}{N} \Big( \omega_{k,l}^N 1_{i = j} + \omega_{j,l}^N 1_{i = k} + \omega_{j,k}^N 1_{i = l} \Big) \leq C \omega_{i,j,k,l}^N, 
    \\
    & \Big| \frac{1}{(N-1)^3} \sum_{q \neq i} \sum_{r \neq i} \sum_{s \neq i} D_{\mu}^a D_{\mu}^a D_{\mu}^a H\big( x^i, p^i,m_{\bx,\ba^N}^{N,-i}, x^q, a^{N,q}, x^r, a^{N,r}, x^s, a^{N,s} \big) D_{p^j} a^{N,q} D_{p^k} a^{N,q} D_{p^l} a^{N,k} \Big| 
    \\
    &\quad \leq \frac{C}{N^3} \sum_{q,r,s = 1}^N |D_{p^j} a^{N,q}| |D_{p^k} a^{N,q}| |D_{p^l} a^{N,k}| \leq  C \omega_{i,j,k,l}^N,
    \\
    & \Big|\frac{1}{(N-1)^2} \sum_{q \neq i} \sum_{r \neq i} D_a D_{\mu}^a D_{\mu}^a H\big(x^i, p^i,m_{\bx,\ba^N}^{N,-i}, x^q, a^{N,q}, x^r, a^{N,r}\big) \Big( D_{p^l} a^{N,q} D_{p^j} a^{N,q} D_{p^k} a^{N,r}
    \\
    &\qquad \qquad \qquad \qquad \qquad \qquad \qquad \qquad \qquad \qquad \qquad
    + D_{p^j} a^{N,q} D_{p^k} a^{N,r} D_{p^l} a^{N,q} \Big) \Big| 
    \\
    &\leq \frac{C}{N^2} \sum_{q,r = 1}^N \Big( |D_{p^l} a^{N,q}| |D_{p^j} a^{N,q}| |D_{p^k} a^{N,r}| + |D_{p^j} a^{N,q}| |D_{p^k} a^{N,q}||D_{p^l} a^{N,r}| \Big)
    \\
    &\leq \frac{C}{N^2} \Big(1_{l = j} + 1_{j = k} + \frac{1}{N} \Big) \leq C\omega_{i,j,k,l}^N, 
    \\
    &\Big|\frac{1}{(N-1)^2} \sum_{q \neq i} \sum_{r \neq i} D_{a'} D_{\mu}^a D_{\mu}^a H\big(x^i, p^i,m_{\bx,\ba^N}^{N,-i}, x^q, a^{N,q}, x^r, a^{N,r}\big)  D_{p^l} a^{N,r} D_{p^j} a^{N,q} D_{p^k} a^{N,r} \Big|
    \\
    &\quad \leq \frac{C}{N^2} \sum_{q,r = 1}^N |D_{p^l} a^{N,r}| |D_{p^j} a^{N,q}| |D_{p^k} a^{N,r}| \leq \frac{C}{N^2} \Big( 1_{l = k}  + \frac{1}{N} \Big) \leq C \omega_{i,j,k,l}^N,
    \\
    &\Big| \frac{1}{(N-1)^2} \sum_{q \neq i} \sum_{r \neq i} D_{\mu}^a D_{\mu}^a H\big(x^i,p^i,m_{\bx,\ba^N}^{N,-i},x^q,a^{N,q} x^r, a^{N,r}\big) \Big( D_{p^l p^j} a^{N,q} D_{p^k} a^{N,r}
    \\
    &\qquad \qquad \qquad \qquad \qquad \qquad \qquad \qquad \qquad \qquad \qquad + D_{p^j} a^{N,q} D_{p^lp^k} a^{N,r} + D_{p^l} a^{N,r} D_{p^j p^k} a^{N,q} \Big) \Big|
    \\
    &\quad \leq \frac{C}{N^2} \sum_{q, r = 1}^N \Big( |D_{p^lp^j} a^{N,q}||D_{p^k} a^{N,r}| + |D_{p^j} a^{N,q}| |D_{p^l p^K} a^{N,r}| + |D_{p^l} a^{N,r}| |D_{p^j p^k} a^{N,q} | \Big)
    \\
    &\quad \leq \frac{C}{N^2} \Big( \omega_{l,j}^N + \omega_{l,k}^N + \omega_{j,k}^N \Big) \leq C \omega_{i,j,k,l}^N 
    \\
    &\Big|\frac{1}{N-1} \sum_{q \neq i} D_a D_a D_{\mu}^a H\big(x^i,p^i,m_{\bx,\ba^N}^{N,-i},x^q, a^{N,q}\big) D_{p^l} a^{N,q} D_{p^k} a^{N,q} D_{p^j} a^{N,q}\Big|
    \\
    &\quad \leq \frac{C}{N} \sum_{q = 1}^N |D_{p^l} a^{N,q}| |D_{p^k} a^{N,q}| |D_{p^j} a^{N,q}| \leq \frac{C}{N} \omega_{j,k,l}^N \leq C \omega_{i,j,k,l}^N
    \\
    &\Big| \frac{1}{(N-1)} \sum_{q \neq i} D_a D_{\mu}^a H\big( x^i,p^i,m_{\bx,\ba^N}^{N,-i}, x^q, a^{N,q} \big) \Big( D_{p^l} a^{N,q} D_{p^jp^k} a^{N,q} + D_{p^l p^k} a^{N,q} D_{p^j} a^{N,q} + D_{p^k} a^{N,q} D_{p^l p^j}a^{N,q} \Big) \Big|
    \\
    &\quad \leq \frac{C}{N} \sum_{q = 1}^N \Big( |D_{p^l} a^{N,q}| | D_{p^jp^k} a^{N,q}| + |D_{p^l p^k} a^{N,q}| |D_{p^j} a^{N,q}| + |D_{p^k} a^{N,q}| |D_{p^lp^j} a^{N,q}| \Big)
    \\
    &\quad \leq \frac{C}{N} \omega_{j,k,l}^N \leq C \omega_{i,j,k,l}^N, 
    \\
    &\Big| \frac{1}{(N-1)} D_{\mu}^a H\big( x^i,p^i,m_{\bx,\ba^N}^{N,-i}, x^q, a^{N,q} \big) D_{p^lp^jp^k} a^{N,l} \Big|
    \\
    &\quad \leq \frac{C}{N} \sum_{q = 1}^N |D_{p^l p^j p^k} a^{N,q}| \leq \frac{C}{N} \omega_{i,j,k}^N \leq C \omega_{i,j,k,l}^N.
\end{align*}
Together, these estimates show that 
\begin{align*}
    \big| D_{p^jp^k p^l} \hat{H}^{N,i} \big| \leq C \omega_{i,j,k,l}^N.
\end{align*}
We have now shown that the derivatives $D_{p^j} \hat{H}^{N,i}$, $D_{p^j p^k} \hat{H}^{N,i}$, and $D_{p^j p^k p^l} \hat{H}^{N,i}$ scale as claimed. The other derivatives up to order three (involving derivatives in $\bx$) are almost identical, and so are omitted.

We next turn to the estimates on $\hat{H}^{i,j}$ and its derivatives. We compute 
\begin{align*}
    D_{p^k} &\hat{H}^{i,j}(\bx,\bp) = \frac{1}{N-1} \sum_{q \neq i} D_{p} D_{\mu}^a H \big( x^i,p^i,m_{\bx,\ba^N}^{N,-i},x^q,a^{N,q} \big) D_{p^j} a^q 1_{i = k}
    \\
    & \quad + \frac{1}{(N-1)^2} \sum_{q \neq i} \sum_{r \neq i} D_{\mu}^a D_{\mu}^a H \big( x^i, p^i, m_{\bx,\ba^N}^{N,-i}, x^q, a^{N,q}, x^r, a^{N,r} \big) D_{p^j} a^{N,q} D_{p^k} a^{N,r}
    \\
    &\quad + \frac{1}{(N-1)} \sum_{q \neq i} D_a D_{\mu}^a H\big( x^i, p^i,m_{\bx,\ba^{N}}^{N,-i}, x^q, a^{N,q}\big) D_{p^k} a^{N,q} D_{p^j} a^{N,q}
    \\
    &\quad + \frac{1}{N-1} \sum_{q \neq i} D_{\mu}^a H\big( x^i,p^i,m_{\bx,\ba^N}^{N,-i}, x^q, a^{N,q} \big) D_{p^kp^j} a^{N,q}, 
    \\
    D_{p^l p^k} &\hat{H}^{i,j}(\bx,\bp) = \frac{1}{N-1} \sum_{q \neq i} D_{pp} D_{\mu}^a H\big(x^i,p^i,m_{\bx,\ba^N}^{N,-i}, x^q, a^{N,q} \big) D_{p^j} a^{N,q} 1_{i = k = k} 
    \\
    & + \frac{1}{(N-1)^2} \sum_{q \neq i} \sum_{r \neq i} D_p D_{\mu}^a D_{\mu}^a H\big(x^i,p^i,m_{\bx,\ba^N}^{N,-i},x^q, a^{N,q} x^r, a^{N,r}\big)  \Big(  D_{p^l} a^{N,r}D_{p^j} a^{N,q} 1_{i = k} + D_{p^j} a^{N,q} D_{p^k} a^{N,r} 1_{i = l}\Big)
    \\
    & + \frac{1}{N-1} \sum_{q \neq i} D_{a} D_p D_{\mu}^a H\big(x^i,p^i,m_{\bx,\ba^N}^{N,-i}, x^q, a^{N,q} \big) \Big( D_{p^l}a^{N,q} D_{p^j} a^{N,q} 1_{i = k} + D_{p^j} a^{N,q} D_{p^k} a^{N,q} 1_{i = l} \Big)
    \\
    &+ \frac{1}{N-1} \sum_{q \neq i} D_p D_{\mu}^a H\big( x^i,p^i,m_{\bx,\ba^N}^{N,-i}, x^q, a^{N,q} \big) \Big( D_{p^l p^j} a^{N,q} 1_{i = k} + D_{p^k p^j} a^{N,q} 1_{i = l} \Big)
    \\
    &+ \frac{1}{(N-1)^3} \sum_{q \neq i} \sum_{r \neq i} \sum_{s \neq i} D_{\mu}^a D_{\mu}^a D_{\mu}^a H\big( x^i,p^i,m_{\bx,\ba^{N}}^{N,-i},x^q, a^{N,q}, x^r, a^{N,r}, x^s, a^{N,s} \big) D_{p^j} a^{N,q} D_{p^k} a^{N,r} D_{p^l} a^{N,s}
    \\
    &+ \frac{1}{(N-1)^2} \sum_{q \neq i} \sum_{r \neq i} D_a D_{\mu}^a D_{\mu}^a H\big( x^i,p^i,m_{\bx,\ba^N}^{N,-i},x^q, a^{N,q} \big) \Big( D_{p^l} a^{N,q} D_{p^j}a^{N,q} D_{p^k} a^{N,r}  + D_{p^l} a^{N,r} D_{p^j} a^{N,q} D_{p^k} a^{N,q} \Big)
    \\
    &+ \frac{1}{(N-1)^2} \sum_{q \neq i} \sum_{r \neq i} D_{a'} D_{\mu}^a D_{\mu}^a H\big(x^i,p^i,m_{\bx,\ba^N}^{N,-i},x^q,a^{N,q},x^r,a^{N,r} \big) D_{p^l} a^{N,r} D_{p^j} a^{N,q} D_{p^k} a^{N,r}
    \\
    &+ \frac{1}{(N-1)^2} \sum_{q \neq i} \sum_{r \neq i} D_{\mu}^a D_{\mu}^a H\big( x^i,p^i,m_{\bx,\ba^N}^{N,-i}, x^q, a^{N,q}, x^r, a^{N,r} \big) \Big( D_{p^lp^j} a^{N,q} D_{p^k} a^{N,r} 
    \\
    &\qquad \qquad \qquad \qquad \qquad \qquad \qquad \qquad \qquad \qquad + D_{p^j} a^{N,q} D_{p^lp^k} a^{N,r} + D_{p^k p^j} a^{N,q} D_{p^l} a^{N,r} \Big)
    \\
    & + \frac{1}{N-1} \sum_{q \neq i} D_a D_a D_{\mu}^a H\big(x^i,p^i,m_{\bx,\ba^{N}}^{N,i}, x^q, a^{N,q}\big) D_{p^l} a^{N,q} D_{p^k} a^{N,q} D_{p^j} a^{N,q} 
    \\
    &+ \frac{1}{N-1} \sum_{q \neq i} D_a D_{\mu}^a H\big( x^i,p^i, m_{\bx,\ba^N}^{N,-i}, x^i, a^{N,q} \big) \Big(D_{p^l p^k} a^{N,q} D_{p^j} a^{N,q} 
    \\
    & \qquad \qquad \qquad \qquad \qquad \qquad \qquad \qquad \qquad \qquad  + D_{p^k} a^{N,q} D_{p^l p^j} a^{N,q} + D_{p^l} a^{N,q} D_{p^k p^j} a^{N,q} \Big) 
    \\
    & + \frac{1}{N-1} \sum_{q \neq i} D_{\mu}^a H\big( x^i, p^i,m_{\bx,\ba^N}^{N,-i}, x^q, a^{N,q} \big) D_{p^l p^k p^j} a^{N,q}.
\end{align*}
First, we note that
\begin{align*}
    |D_{p^k} \hat{H}^{i,j}| \leq \frac{C}{N} \sum_{q \neq i}  |D_{p^j} a^{N,q}| \leq \frac{C}{N}.
\end{align*}
For notational simplicity, we now introduce 
\begin{align*}
    \wt{\omega}_{i,j,k}^N = \begin{cases}
        \omega_{i,j,k}^N & i \neq j, 
        \\
        \frac{1}{N} \omega_{i,k}^N & i = j, 
    \end{cases} \qquad
    \wt{\omega}_{i,j,k,l}^N = \begin{cases}
        \omega_{i,j,k,l}^N & i \neq j, 
        \\
        \frac{1}{N} \omega_{i,k,l}^N & i = j.
    \end{cases}
\end{align*}
We now use the boundedness of the second and third third derivatives of $H$ and Proposition \ref{prop.aderivscaling} to estimate 
\begin{align*}
   &\Big| \frac{1}{N-1} \sum_{q \neq i} D_{p} D_{\mu}^a H \big( x^i,p^i,m_{\bx,\ba^N}^{N,-i},x^q,a^{N,q} \big) D_{p^j} a^q 1_{i = k} \Big|
   \\&\quad \leq \frac{C}{N} \sum_{q =1}^N |D_{p^j} a^{N,q}| 1_{i = k} \leq \frac{C}{N} 1_{i = k} \leq C \wt{\omega}_{i,j,k}^N.
    \\
    & \Big|\frac{1}{(N-1)^2} \sum_{q \neq i} \sum_{r \neq i} D_{\mu}^a D_{\mu}^a H \big( x^i, p^i, m_{\bx,\ba^N}^{N,-i}, x^q, a^{N,q}, x^r, a^{N,r} \big) D_{p^j} a^{N,q} D_{p^k} a^{N,r} \Big|
    \\
    &\quad \leq \frac{C}{N^2} \sum_{q,r = 1}^N |D_{p^j} a^{N,q}| |D_{p^k} a^{N,r}|  \leq \frac{C}{N^2} \leq C \wt{\omega}_{i,j,k}^N, 
    \\
    &\Big|\frac{1}{(N-1)} \sum_{q \neq i} D_a D_{\mu}^a H\big( x^i, p^i,m_{\bx,\ba^{N}}^{N,-i}, x^q, a^{N,q}\big) D_{p^k} a^{N,q} D_{p^j} a^{N,q} \Big|
    \\
    &\quad \leq \frac{C}{N} \sum_{q = 1}^N |D_{p^k} a^{N,q}| |D_{p^j} a^{N,q} | \leq \frac{C}{N} \Big( \frac{1}{N} + 1_{j = k} \Big) \leq C \wt{\omega}_{i,j,k}^N, 
    \\
    &\Big|\frac{1}{N-1} \sum_{q \neq i} D_{\mu}^a H\big( x^i,p^i,m_{\bx,\ba^N}^{N,-i}, x^q, a^{N,q} \big) D_{p^kp^j} a^{N,q}\Big| 
    \\
    &\quad \leq \frac{C}{N} \sum_{q = 1}^N |D_{p^k p^j} a^{N,q}| \leq \frac{C}{N} \omega_{j,k}^N \leq C \wt{\omega}_{i,j,k}^N.
\end{align*}
We have thus verified that 
\begin{align*}
    |D_{p^k p^j} \hat{H}^{i,j}| \leq C \wt{\omega}_{i,j,k}^N, 
\end{align*}
as desired. Next, we estimate 
\begin{align*}
    &\Big|  \frac{1}{N-1} \sum_{q \neq i} D_{pp} D_{\mu}^a H\big(x^i,p^i,m_{\bx,\ba^N}^{N,-i}, x^q, a^{N,q} \big) D_{p^j} a^{N,q} 1_{i = k = l} \Big| 
    \\
    &\leq \frac{C}{N} \sum_{q = 1}^N |D_{p^j} a^{N,q}| 1_{i = k = l} \leq \frac{C}{N} 1_{i = k = l} \leq C \wt{\omega}_{i,j,k,l}^N, 
    \\
    &\Big|\frac{1}{(N-1)^2} \sum_{q \neq i} \sum_{r \neq i} D_p D_{\mu}^a D_{\mu}^a H\big(x^i,p^i,m_{\bx,\ba^N}^{N,-i},x^q, a^{N,q} x^r, a^{N,r}\big)  \Big(  D_{p^l} a^{N,r}D_{p^j} a^{N,q} 1_{i = k} + D_{p^j} a^{N,q} D_{p^k} a^{N,r} 1_{i = l}\Big)\Big|
    \\
    &\quad \leq \frac{C}{N^2} \sum_{q,r = 1}^N \Big( |D_{p^l} a^{N,r} | |D_{p^j} a^{N,q}| 1_{i = k} + |D_{p^j} a^{N,q}| | D_{p^k} a^{N,r}| 1_{i = l} \Big)
    \\
    &\quad \leq \frac{C}{N^2} \Big(1_{i = k} + 1_{i = l} \Big) \leq C \wt{\omega}_{i,j,k,l}^N,
    \\
    &\Big|\frac{1}{N-1} \sum_{q \neq i} D_{a} D_p D_{\mu}^a H\big(x^i,p^i,m_{\bx,\ba^N}^{N,-i}, x^q, a^{N,q} \big) \Big( D_{p^l}a^{N,q} D_{p^j} a^{N,q} 1_{i = k} + D_{p^j} a^{N,q} D_{p^k} a^{N,q} 1_{i = l} \Big) \Big|
    \\
    &\quad \leq \frac{C}{N} \sum_{q = 1}^N \Big( |D_{p^l} a^{N,q}| | D_{p^j} a^{N,q}| 1_{i = k} + |D_{p^j} a^{N,q}| |D_{p^k} a^{N,q}| 1_{i = l} \Big)
    \\
    &\quad \leq \frac{C}{N} \Big( 1_{j = l} 1_{i = k} + 1_{j = k} 1_{i = l} + \frac{1}{N} 1_{i = k} + \frac{1}{N} 1_{i = l} \Big) = \frac{C}{N} \Big(1_{j = l} 1_{i = k} + 1_{j = k} 1_{i = l} \Big) + \frac{C}{N^2} \Big(1_{i = k} + 1_{i = l} \Big)
    \\
    &\quad \leq C \wt{\omega}_{i,j,k,l}^N, 
    \\
    &\Big|\frac{1}{N-1} \sum_{q \neq i} D_p D_{\mu}^a H\big( x^i,p^i,m_{\bx,\ba^N}^{N,-i}, x^q, a^{N,q} \big) \Big( D_{p^l p^j} a^{N,q} 1_{i = k} + D_{p^k p^j} a^{N,q} 1_{i = l} \Big) \Big|
    \\
    &\quad \leq \frac{C}{N} \sum_{q = 1}^N \Big( |D_{p^l p^j} a^{N,q}| 1_{i = k} + |D_{p^k p^j} a^{N,q}| 1_{i = l} \Big) 
    \\
    &\quad \leq \frac{C}{N} \Big(\omega_{j,l}^N 1_{i = k} + \omega_{j,k}^N 1_{i = l} \Big) \leq C \wt{\omega}_{i,j,k,l}^N, 
    \\
    & \Big|\frac{1}{(N-1)^3} \sum_{q \neq i} \sum_{r \neq i} \sum_{s \neq i} D_{\mu}^a D_{\mu}^a D_{\mu}^a H\big( x^i,p^i,m_{\bx,\ba^{N}}^{N,-i},x^q, a^{N,q}, x^r, a^{N,r}, x^s, a^{N,s} \big) D_{p^j} a^{N,q} D_{p^k} a^{N,r} D_{p^l} a^{N,s}\Big| 
    \\
    &\quad \leq \frac{C}{N^3} \sum_{q,r,s = 1}^N |D_{p^j} a^{N,q}| |D_{p^k} a^{N,r}| |D_{p^l} a^{N,s}| \leq \frac{C}{N^3} \leq C \wt{\omega}_{i,j,k,l}^N, 
    \\
    &\Big|\frac{1}{(N-1)^2} \sum_{q \neq i} \sum_{r \neq i} D_a D_{\mu}^a D_{\mu}^a H\big( x^i,p^i,m_{\bx,\ba^N}^{N,-i},x^q, a^{N,q} \big) \Big( D_{p^l} a^{N,q} D_{p^j}a^{N,q} D_{p^k} a^{N,r}  + D_{p^l} a^{N,r} D_{p^j} a^{N,q} D_{p^k} a^{N,q} \Big) \Big|
    \\
    &\quad \leq \frac{C}{N^2} \sum_{q,r = 1}^N \Big( |D_{p^j} a^{N,q}| |D_{p^l} a^{N,q}| |D_{p^k} a^{N,r}| + |D_{p^l} a^{N,r}| | D_{p^j} a^{N,q}| | D_{p^k} a^{N,q} | \Big)
    \\
    &\quad \leq \frac{C}{N^2} \Big( \frac{1}{N} + 1_{j = l} + 1_{j = k} \Big) \leq C \wt{\omega}_{i,j,k,l}^N, 
    \\
    &\Big|\frac{1}{(N-1)^2} \sum_{q \neq i} \sum_{r \neq i} D_{a'} D_{\mu}^a D_{\mu}^a H\big(x^i,p^i,m_{\bx,\ba^N}^{N,-i},x^q,a^{N,q},x^r,a^{N,r} \big) D_{p^l} a^{N,r} D_{p^j} a^{N,q} D_{p^k} a^{N,r} \Big|
    \\
    &\quad \leq \frac{C}{N^2} \sum_{q,r = 1}^N |D_{p^l} a^{N,r}| | D_{p^j} a^{N,q}| | D_{p^k} a^{N,r} | \leq \frac{C}{N^2}\big( \frac{1}{N} + 1_{l = k} \big) \leq C \wt{\omega}_{i,j,k,l}^N,
    \\
    &\Big|\frac{1}{(N-1)^2} \sum_{q \neq i} \sum_{r \neq i} D_{\mu}^a D_{\mu}^a H\big( x^i,p^i,m_{\bx,\ba^N}^{N,-i}, x^q, a^{N,q}, x^r, a^{N,r} \big) \Big( D_{p^lp^j} a^{N,q} D_{p^k} a^{N,r} 
    \\
    &\qquad \qquad \qquad \qquad \qquad \qquad \qquad \qquad \qquad \qquad + D_{p^j} a^{N,q} D_{p^lp^k} a^{N,r} + D_{p^k p^j} a^{N,q} D_{p^l} a^{N,r} \Big)\Big| 
    \\
    &\quad \leq \frac{C}{N^2} \sum_{q,r = 1}^N \Big( |D_{p^lp^j} a^{N,q}| |D_{p^k} a^{N,r}| + |D_{p^j} a^{N,q}| D_{p^l p^k} a^{N,r}| + |D_{p^kp^j} a^{N,q}| |D_{p^l} a^{N,r}| \Big)
    \\
    &\quad \leq \frac{C}{N^2} \Big(\omega_{j,l}^N + \omega_{k,l}^N + \omega_{j,k}^N \Big) \leq C \wt{\omega}_{i,j,k,l}^N, 
    \\
    & + \frac{1}{N-1} \sum_{q \neq i} D_a D_a D_{\mu}^a H\big(x^i,p^i,m_{\bx,\ba^{N}}^{N,i}, x^q, a^{N,q}\big) D_{p^l} a^{N,q} D_{p^k} a^{N,q} D_{p^j} a^{N,q} 
    \\
    &\Big|\frac{1}{N-1} D_a D_{\mu}^a H\big( x^i,p^i, m_{\bx,\ba^N}^{N,-i}, x^i, a^{N,q} \big) \Big(D_{p^l p^k} a^{N,q} D_{p^j} a^{N,q} 
    \\
    & \qquad \qquad \qquad \qquad \qquad \qquad \qquad \qquad \qquad \qquad  + D_{p^k} a^{N,q} D_{p^l p^j} a^{N,q} + D_{p^l} a^{N,q} D_{p^k p^j} a^{N,q} \Big)\Big|
    \\
    &\quad \leq \frac{C}{N} \sum_{q = 1}^N \Big( |D_{p^l p^k} a^{N,q}| | D_{p^j} a^{N,q}| + |_D{p^k} a^{N,q}| | D_{p^lp^j} a^{N,q}| + |D_{p^l} a^{N,q}| |D_{p^kp^j} a^{N,q}| \Big)
    \\
    &\quad \leq \frac{C}{N}  \omega_{j,k,l}^N \leq C \wt{\omega}_{i,j,k,l}^N, 
    \\
    & \Big|\frac{1}{N-1} \sum_{q \neq i} D_{\mu}^a H\big( x^i, p^i,m_{\bx,\ba^N}^{N,-i}, x^q, a^{N,q} \big) D_{p^l p^k p^j} a^{N,q}\Big|
    \\
    &\quad \leq \frac{C}{N} \sum_{q = 1}^N |D_{p^l p^k p^j} a^{N,q}| \leq \frac{C}{N} \omega_{j,k,l}^N \leq C \wt{\omega}_{i,j,k,l}^N.
\end{align*}
Thus yields 
\begin{align*}
   | D_{p^l p^k} \hat{H}^{N,i,j} | \leq C \wt{\omega}_{i,j,k,l}^N.
\end{align*}
We have now shown that $\hat{H}^{N,i,j}$, $D_{p^k} \hat{H}^{N,i,j}$ and $D_{p^l p^k} \hat{H}^{N,i,j}$ scale as claimed; the other derivatives (involving derivatives in the $\bx$ variable) are similar and so are omitted. This completes the proof of Lemma \ref{lem.hatH}.
\end{proof}

We now turn to the bounds on the coefficients appearing in the equations for the third deriatives of $u^{N,i}$.

\begin{proof}[Proof of Lemma \ref{lem.third.order.coeff}]
We begin by using Propositions \ref{prop.aderivscaling}, \ref{prop.uiji} and \ref{prop.iii} to estimate 
\begin{align*}
    &|D_{x^l} a^{N,l}| \leq C, \qquad |D_{x^j} a^{N,j}| \leq C, 
    \\
    &\Big| \sum_{n = 1}^N D_{p^n} a^{N,l} u^{n,l,n} \Big| \leq |D_{p^l} a^{N,l}| |u^{l,l,l}| + \sum_{n \neq l} |D_{p^n} a^{N,l} u^{N,l,n}| \leq C, 
\end{align*}
and similarly 
\begin{align*}
    \Big| \sum_{n = 1}^N D_{p^n} a^{N,k} u^{N,k,n} \Big| \leq C, 
    \qquad 
    &\Big| \sum_{n = 1}^N D_{p^n} a^{N,j}\Big| \leq C.
\end{align*}
It follows that 
\begin{align*}
    \| \cA^{i,j,k,l} \|_{\infty} \leq C. 
\end{align*}
Next, for $r \neq l$, 
\begin{align*}
    &|D_{x^l} a^{N,r}| \leq C/N, 
    \\
    &\Big| \sum_{n = 1}^N D_{p^n} a^{N,r} u^{n,l,n} \Big| \leq |D_{p^r} a^{N,r}| |u^{r,l,r}| + |D_{p^l} a^{N,r}| |u^{l,l,l}| + \sum_{n \neq r,l} |D_{p^n} a^{N,r}| |u^{n,l,n}| \leq \frac{C}{N}, 
\end{align*}
so that 
\begin{align*}
    \| \cA_1^{i,j,k,l; r} \|_{\infty} \leq C/N. 
\end{align*}
Almost identical arguments show that for $r \neq k$, $\|\cA_{2}^{i,j,k,l; r}\|_{\infty} \leq C/N$, and for $r \neq j$, $\|\cA_3^{i,j,k,l; r}\|_{\infty} \leq C/N$. 

Next, by Lemma \ref{lem.hatH}, together with Proposition \ref{prop.aderivscaling} and Proposition \ref{prop.uij},
\begin{align*}
    |\hat{H}^{N,i,r}| \leq \frac{C}{N}, \quad \Big| \sum_{n \neq i} D_{p^r} a^{N,n} u^{i,n} \Big| \leq \frac{C}{N} \sum_{n = 1}^N |D_{p^j} a^{N,n}| \leq \frac{C}{N}, 
\end{align*}
and so we have 
\begin{align*}
    \| \cB^{i,j,k,l; r} \|_{\infty} \leq C/N.
\end{align*}
Finally, we turn to estimating $\cC^{i,j,k,l}$. We will use heavily the fact that for any $i,j,k,l \in \{1,...,N\}$, 
\begin{align*}
    &\sum_{q = 1}^N \omega_{i,j,q}^N \leq C \omega_{i,j}^N, 
   \qquad \sum_{q = 1}^N \omega_{i,j,k,q}^N \leq C \omega_{i,j,k}^N, 
    \qquad \sum_{q = 1}^N \omega_{i,q}^N \omega_{j,q}^N \leq C \omega_{i,j}^N, 
    \\
    & \sum_{q = 1}^N \omega_{i,j,q}^N \omega_{k,q}^N \leq C \omega_{i,j,k}^N, 
    \qquad \sum_{q = 1}^N \omega_{i,j,q}^N \omega_{k,l,q}^N \leq C \omega_{i,j,k,l}^N, 
   \qquad \sum_{q = 1}^N \omega_{i,j,k,q}^N \omega_{l,q}^N \leq C \omega_{i,j,k,l}^N,  
\end{align*}
and the fact that by Propositions \ref{prop.uiji}, \ref{prop.iii}, and \ref{prop.ijk}, we know that 
\begin{align*}
   \|u^{i,j,k,l}\|_{L^1(\cL^N)} \leq \| u^{i,j,k,l} \|_{L^2(\cL^N)} \leq C \omega^N,{i,j,k}, \quad \forall i,j,k,l \in \{1,...,N\}.
\end{align*}
Repeatedly applying these bounds together with the estimates $|u^{i,j}| \leq C\omega_{i,j}^N$, $|u^{i,j,k}| \leq C \omega_{i,j,k}^N$ and the bounds on the derivatives of $\hat{H}^{N,i}$ and $\hat{H}^{N,i,j}$ from Lemma \ref{lem.hatH}, allows us to bound each of the terms appearing in $\cC^{i,j,k,l}$:
\begin{align*}
 &\norm{  - \sum_{r \neq i} \sum_{n=1}^N D_{p^n} a^{N,r} u^{i,r,l} u^{n,j,n,k} \Big|}_{L^1(\cL^N)}  \leq C \sum_{r,n = 1}^N \omega_{r,n}^N \omega_{i,r,l}^N \norm{u^{n,j,n,k}}_{L^1(\cL^N)}
 \\
 &\qquad  \leq C \sum_{n = 1}^N \omega_{i,l,n}^N \omega_{j,k,n}^N \leq C \omega_{i,j,k,l}^N,
 \\
 &\norm{ \sum_{r \neq i} \sum_{n=1}^N \Big( D_{x^lp^n} a^{N,r} + \sum_{q=1}^N D_{p^q p^n} a^{N,r} u^{q,l,q} \Big) u^{i,r} u^{n,j,n,k}}_{L^1(\cL^N)} \leq \frac{C}{N} \sum_{r,n = 1}^N \omega_{l,n,r}^N \norm{u^{n,j,n,k}}_{L^1(\cL^N)} 
 \\
 &\qquad \leq \frac{C}{N} \sum_{n = 1}^N \omega_{l,n}^N \omega_{j,k,n}^N \leq \frac{C}{N} \omega_{j,k,l}^N \leq C \omega_{i,j,k,l}^N, 
    \\
    &\norm{ \sum_{r=1}^N \Big( D_{x^l} \hat{H}^{N,i,r} + \sum_{n=1}^N D_{p^n} \hat{H}^{N,i,r} u^{n,l,n} \Big) u^{r,j,r,k} }_{L^1(\cL^N)}
    \leq \frac{C}{N} \omega_{i,l}^N \norm{u^{i,j,i,k}}_{L^1(\cL^N)} + C \sum_{r = 1}^N \omega_{i,r,l}^N \norm{u^{r,j,r,k}}_{L^1(\cL^N)}
    \\
    &\qquad \leq \frac{C}{N} \omega_{i,l}^N \omega_{i,j,k}^N + C \sum_{r = 1}^N \omega_{i,r,l}^N \omega_{j,k,r}^N \leq C \omega_{i,j,k,l}^N, 
    \\
    &\norm{ \sum_{r,n=1}^N D_{p^n} a^{N,r} u^{i,j,r} u^{n,k,n,l} }_{L^1(\cL^N)}
    \leq C \sum_{r,n = 1}^N \omega_{r,n}^N \omega_{i,j,r}^N \norm{u^{n,k,n,l}}_{L^1(\cL^N)} \leq C \sum_{n = 1}^N \omega_{i,j,n}^N \norm{u^{n,k,n,l}}_{L^1(\cL^N)}
    \\
    &\qquad \leq C \sum_{n = 1}^N \omega_{i,j,n}^N \omega_{k,l,n}^N \leq C \omega_{i,j,k,l}^N, 
    \\
    &\norm{ \sum_{r,n=1}^N \Big( D_{x^l p^n} a^{N,r} + \sum_{q=1}^N D_{p^q p^n} a^{N,r} u^{q,l,q} \Big) u^{n,k,n} u^{i,j,r} }_{\infty} 
    \\
    &\quad \leq C \sum_{r,n = 1}^N \omega_{l,n,r}^N \omega_{k,n}^N \omega_{i,j,r} \leq C \sum_{r = 1}^N \omega_{l,r,k}^N \omega_{i,j,r}^N \leq C \omega_{i,j,k,l}^N, 
    \\
    &\norm{\sum_{r=1}^N \Big( D_{x^lx^k} a^{N,r} + \sum_{n=1}^N D_{p^n x^k}a^{N,r} u^{n,l,n} \Big) u^{i,j,r}}_{\infty} \leq C\sum_{r= 1}^N \omega_{k,l,r}^N \omega_{i,j,r}^N \leq C \omega_{i,j,k,l}^N
    \\
    &\norm{ \sum_{r \neq i} \sum_{n=1}^N D_{p^n} a^{N,r} u^{n,j,n,l} u^{i,r,k} }_{\infty} \leq C \sum_{r,n = 1}^N \omega_{r,n}^N \omega_{i,k,r}^N |u^{n,j,n,l}| \leq C \sum_{n = 1}^N \omega_{i,k,n}^N |u^{n,j,n,l}|, 
    \\
    & \norm{\sum_{r \neq i} \sum_{n=1}^N \Big( D_{x^l p^n} a^{N,r} + \sum_{q=1}^N D_{p^q p^n} a^{N,r} u^{q,l,q} \Big) u^{n,j,n} u^{i,r,k}}_{\infty} 
    \\
    &\quad \leq C \sum_{r,n = 1}^N \omega_{l,n,r}^N \omega_{k,n}^N \omega_{i,j,r}^N \leq C \sum_{r = 1}^N \omega_{k,l,r}^N \omega_{i,j,r}^N \leq c \omega_{i,j,k,l}^N, 
    \\
    &\norm{\sum_{r \neq i} \Big( D_{x^l x^j} a^{N,r} + \sum_{n=1}^N D_{p^n x^j} a^{N,r} u^{n,l,n} \Big) u^{i,k,r} }_{\infty} \leq \sum_{r = 1}^N \omega_{j,l,r}^N \omega_{i,k,r}^N \leq C \omega_{i,j,k,l}^N
    \\
    &\norm{ \sum_{r \neq i} \sum_{n,q=1}^N D_{p^n p^q} a^{N,r} u^{n,j,n,l} u^{q,k,q} u^{i,r} }_{L^1(\cL^N)}
    \\
    &\quad \leq \frac{C}{N} \sum_{r,n,q = 1}^N \omega_{r,n,q}^N \omega_{q,k}^N \norm{u^{n,j,n,l}}_{L^1(\cL^N)} \leq \frac{C}{N} \sum_{n,q = 1}^N \omega_{n,q}^N \omega_{k,q}^N \omega_{j,l,n}^N \leq \frac{C}{N} \sum_{n = 1}^N \omega_{k,n}^N \omega_{j,l,n}^N \leq \frac{C}{N} \omega_{j,k,l}^N \leq C \omega_{i,j,k,l}^N, 
    \\
    &\norm{ \sum_{r \neq i} \sum_{n,q=1}^N D_{p^np^q} a^{N,r} u^{n,j,n} u^{q,k,q,l} u^{i,r} }_{L^1(\cL^N)}
    \\
    &\quad \leq \frac{C}{N} \sum_{r,n,q = 1}^N \omega_{r,n,q}^N \omega_{j,n}^N \norm{u^{q,k,q,l}}_{L^1(\cL^N)} \leq \frac{C}{N} \sum_{n,q = 1}^n \omega_{n,q}^N \omega_{j,n}^N \leq \frac{C}{N} \sum_{q = 1}^N \omega_{j,q} \omega_{k,l,q}^N \leq \frac{C}{N} \omega_{j,k,l}^N \leq C \omega_{i,j,k,l}^N, 
    \\
    &\norm{\sum_{r,n=1}^N D_{p^n} \hat{H}^{N,i,r} u^{n,k,n} u^{r,j,r,l}}_{L^1(\cL^N)}
    \\
        &\quad \leq \frac{C}{N} \sum_{n = 1}^N \omega_{i,n}^N \omega_{n,k}^N \norm{u^{i,j,i,l}}_{L^1(\cL^N)} + \sum_{r,n = 1}^N \omega_{i,r,n}^N \omega_{k,n}^N \norm{u^{r,j,r,l}}_{L^1(\cL^N)} \leq \frac{C}{N} \omega_{i,k}^N \omega_{i,j,l}^N + \sum_{r = 1}^N \omega_{i,k,r}^N \omega_{j,l,r}^N \leq C \omega_{i,j,k,l}^N,  
   \\
    & \norm{\sum_{r,n=1}^N D_{p^n} \hat{H}^{N,i,r} u^{n,k,n,l} u^{r,j,r} }_{L^1(\cL^N)}
    \\
    &\qquad \leq \frac{C}{N} \omega_{i,j}^N \sum_{n = 1}^N \omega_{i,n}^N \norm{u^{n,k,n,l}}_{L^1(\cL^N)} + C \sum_{r,n = 1}^N \omega_{i,r,n}^N \omega_{j,r}^N \norm{u^{n,k,n,l}}_{L^1(\cL^N)}
    \\
    &\qquad \leq  \frac{C}{N} \omega_{i,j}^N \sum_{n = 1}^N \omega_{i,n}^N \omega_{k,l,n}^N + C \sum_{n = 1}^N \omega_{i,j,n}^N\omega_{k,l,n}^n \leq C \omega_{i,j,k,l}^N, 
    \\
    &\norm{ \sum_{r \neq i} \sum_{n=1}^N D_{x^k p^n} a^{N,r} u^{n,j,n,l} u^{i,r}}_{L^1(\cL^N)} \leq \frac{C}{N} \sum_{r,n = 1}^N \omega_{k,n,r}^N \norm{u^{n,j,n,l}}_{L^1(\cL^N)}
    \\
    &\qquad \leq \frac{C}{N} \sum_{n = 1}^N \omega_{k,n}^N \omega_{j,l,n}^N \leq C \omega_{i,j,k,l}^N,
    \\
    &\norm{ \sum_{r \neq i} \sum_{n=1} D_{p^n x^j} a^{N,r} u^{n,k,n,l} u^{i,r} }_{L^1(\cL^N)} \leq \frac{C}{N} \sum_{r,n = 1}^N \omega_{j,n,r}^N \norm{u^{n,k,n,l}}_{L^1(\cL^N)} 
    \\
    &\qquad \leq \frac{C}{N} \sum_{n = 1}^N \omega_{j,n}^N \omega_{k,l,n}^N \leq \frac{C}{N} \omega_{j,k,l}^N \leq C \omega_{i,j,k,l}^N,
    \\
    &\norm{\sum_{r=1}^N D_{x^k} \hat{H}^{N,i,r} u^{r,j,r,l}}_{L^1(\cL^N)} \leq \frac{C}{N} \omega_{i,k}^N \omega_{i,j,l}^N  + \sum_{r = 1}^N \omega_{i,k,r}^N \omega_{j,l,r}^N \leq C \omega_{i,j,k,l}^N, 
    \\
    &\norm{\sum_{r=1}^N D_{p^r x^j} \hat{H}^{N,i} u^{r,k,r,l} }_{L^1(\cL^N)} \leq C \sum_{r = 1}^N \omega_{i,j,r}^N \norm{u^{r,k,r,l}}_{L^1(\cL^N)}
    \\
    &\qquad \leq C \sum_{r = 1}^N \omega_{i,j,r}^N \omega_{k,l,r}^N \leq C \omega_{i,j,k,l}^N, 
    \\
    &\norm{ \sum_{r \neq i} \sum_{n,q=1}^N D_{p^np^q} a^{N,r} u^{n,j,n} u^{q,k,q} u^{i,r,l} }_{\infty}
    \\
    &\quad \leq C \sum_{r,n,q = 1}^N \omega_{r,n,q}^N \omega_{j,n}^N \omega_{k,q}^N \omega_{i,l,r}^N \leq C\sum_{r,n = 1}^N \omega_{r,n,k}^N \omega_{j,n}^N \omega_{i,l,r}^N \leq C \sum_{r = 1}^N \omega_{j,k,r}^N \omega_{i,l,r}^N \leq C \omega_{i,j,k,l}^N, 
    \\
    &\norm{\sum_{r \neq i} \sum_{n,q=1}^N \Big( D_{x^l p^np^q} a^{N,r} + \sum_{\hat r=1}^N D_{p^{\hat r} p^n p^q} a^{N,r} u^{\hat r,l,\hat r} \Big) u^{n,j,n} u^{q,k,q} u^{i,r} }_{\infty} 
    \\
    &\quad \leq \frac{C}{N} \sum_{r,n,q = 1}^N \omega_{l,n,q,r}^N \omega_{j,n}^N \omega_{k,q}^N \leq \frac{C}{N} \sum_{n,q = 1}^N \omega_{l,n,q}^N \omega_{j,n}^N \omega_{k,q}^N \leq \frac{C}{N} \sum_{n = 1}^N \omega_{l,k,n}^N \omega_{j,n}^N 
    \\
    &\quad \leq \frac{C}{N} \omega_{j,k,l}^N \leq C \omega_{i,j,k,l}^N, 
    \\
    & \norm{ \sum_{r,n=1}^N \Big( D_{x^l p^n} \hat{H}^{N,i,r} + \sum_{q=1}^N D_{p^qp^n} \hat{H}^{N,i,r} u^{q,l,q} \Big) u^{n,k,n} u^{r,j,r} }_{\infty}
    \\
    &\quad \leq \frac{C}{N} \sum_{n = 1}^N \omega_{i,l,n}^N \omega_{k,n}^N \omega_{i,j}^N + C \sum_{r,n = 1}^N \omega_{i,l,r,n}^N \omega_{k,n}^N \omega_{j,r}^N \leq \frac{C}{N} \omega_{i,k,l}^N \omega_{i,j}^N + C \sum_{r = 1}^N \omega_{i,k,l,r}^N \omega_{j,r}^N \leq C \omega_{i,j,k,l}^N, 
    \\
    &\norm{ \sum_{r \neq i} \sum_{n=1}^N D_{x^k p^n} a^{N,r} u^{n,j,n} u^{i,r,l} }_{\infty} \leq C \sum_{r,n = 1}^N \omega_{k,r,n}^N \omega_{j,n}^N \omega_{i,l,r}^N \leq C \sum_{r = 1}^N \omega_{k,j,r}^N \omega_{i,l,r}^N \leq C \omega_{i,j,k,l}^N, 
   \\
    &\norm{ \sum_{r \neq i} \sum_{n=1}^N \Big( D_{x^lx^k p^n} a^{N,r} + \sum_{q=1}^N D_{p^q x^k p^n} a^r u^{q,l,q} \Big) u^{n,j,n} u^{i,r} }_{\infty} 
    \\
    &\quad \leq \frac{C}{N} \sum_{r,n = 1}^N \omega_{k,l,n,r}^N \omega_{j,n}^N \leq \frac{C}{N} \sum_{n = 1}^N \omega_{k,l,n}^N \omega_{j,n}^N \leq \frac{C}{N} \omega_{j,k,l}^N \leq C \omega_{i,j,k,l}^N, 
    \\
    &\norm{\sum_{r \neq i} \sum_{n=1}^N D_{p^n x^j} a^{N,r} u^{n,k,n} u^{i,r,l} }_{\infty} \leq C \sum_{r,n = 1}^N \omega_{j,n,r}^N \omega_{k,n}^N \omega_{i,l,r}^N \leq C \sum_{r = 1}^N \omega_{j,k,r}^N \omega_{i,l,r}^N \leq C \omega_{i,j,k,l}^N
    \\
    &\norm{\sum_{r \neq i} \sum_{n=1}^N \Big( D_{x^l p^n x^j} a^{N,r} + \sum_{q=1}^N D_{p^q p^n x^j} a^{N,r} u^{q,l,q} \Big) u^{n,k,n} u^{i,r} }_{\infty}
    \\
    &\quad \leq \frac{C}{N} \sum_{r,n = 1}^N \omega_{j,l,n,r}^N \omega_{n,k}^N \leq \frac{C}{N} \sum_{n = 1}^N \omega_{j,l,n}^N \omega_{n,k}^N \leq \frac{C}{N} \omega_{j,l,k}^N \leq C \omega_{i,j,k,l}^N. 
    \\
    &\norm{\sum_{r=1}^N \Big( D_{x^lx^k} \hat{H}^{N,i,r} + \sum_{n=1}^N D_{p^nx^k} \hat{H}^{N,i,r} u^{n,l,n} \Big)u^{r,j,r}}_{\infty}
     \leq C \sum_{r = 1}^N \omega_{i,k,l,r}^N \omega_{j,r}^N \leq C \omega_{i,j,k,l}^N, 
    \\
    &\norm{\sum_{r=1}^N \Big( D_{x^l p^r x^j} \hat{H}^{N,i} + \sum_{n=1}^N D_{p^np^r x^j} \hat{H}^{N,i} u^{n,l,n} \Big) u^{r,k,r}}_{\infty}
    \leq  C \sum_{r = 1}^N \omega_{i,j,l,r}^N \omega_{k,r}^N \leq C \omega_{i,j,k,l}^N, 
    \\
    &\norm{ \sum_{r \neq i} D_{x^kx^j} a^{N,r} u^{i,r,l}}_{\infty}
     \leq C \sum_{r = 1}^N \omega_{j,k,r}^N \omega_{i,l,r}^N \leq C \omega_{i,j,k,l}^N, 
    \\
    &\norm{\sum_{r \neq i} \Big( D_{x^lx^kx^j} a^{N,r}+\sum_{n=1}^N D_{p^n x^kx^j} a^{N,r} u^{n,l,n} \Big) u^{i,r} } \leq \frac{C}{N} \sum_{r = 1}^N \omega_{j,k,l,r}^N
     \leq \frac{C}{N} \omega_{j,k,l}^N \leq C \omega_{i,j,k,l}^N, 
    \\
    &\norm{D_{x^lx^kx^j} \hat{H}^{N,i}}_{\infty} \leq C \omega_{i,jk,l}^N, 
    \\
    &\norm{\sum_{r=1}^N D_{p^r x^kx^j} \hat{H}^{N,i} u^{r,l,r}}_{\infty} \leq C \sum_{r = 1}^N \omega_{i,j,k,r}^N \omega_{l,r}^N \leq C \omega_{i,j,k,l}^N. 
\end{align*}
This gives the bound $\| \cC^{i,j,k,l}\|_{L^1(\cL^N)} \leq C \omega_{i,j,k,l}^N$. 
\end{proof}

\bibliographystyle{alpha}
\bibliography{joe_alpar}

\end{document}